\documentclass[11pt,reqno]{amsart}

\usepackage{amsmath,amsthm,amssymb,amsfonts}
\usepackage{bm}
\usepackage{natbib}
\usepackage{bbm}
\usepackage{graphicx}
\usepackage{multirow}
\usepackage{here}
\usepackage{fullpage}
\usepackage{url}
\usepackage{color}
\usepackage{xcolor}
\usepackage{mathrsfs}
\usepackage{enumitem}
\usepackage{comment}

\makeatletter

\@addtoreset{equation}{section}
\makeatletter

\newtheorem{theorem}{Theorem}[section]
\newtheorem{corollary}{Corollary}[section]
\newtheorem{lemma}{Lemma}[section]
\newtheorem{proposition}{Proposition}[section]
\newtheorem{assumption}{Assumption}[section]
\theoremstyle{definition}

\newtheorem{remark}{Remark}[section]

\def\X{{\text{\boldmath $X$}}}

\DeclareMathOperator{\Var}{Var}
\DeclareMathOperator{\Cov}{Cov}

\newcommand{\argmin}{\mathop{\rm arg~min}\limits}
\def\diag{{\rm diag}}

\newcommand{\R}{\mathbb{R}}

\renewcommand{\tilde}{\widetilde}
\renewcommand{\hat}{\widehat}

\allowdisplaybreaks

\begin{document}

\title[]{Local polynomial trend regression for spatial data on $\mathbb{R}^d$}
\thanks{D. Kurisu is partially supported by JSPS KAKENHI Grant Number 20K13468. The authors would like to thank Takuya Ishihara, Taisuke Otsu, Peter Robinson, Masayuki Sawada, and Yoshihiro Yajima for their helpful comments and suggestions.
} 

\author[D. Kurisu]{Daisuke Kurisu}
\author[Y. Matsuda]{Yasumasa Matsuda}

\date{First version: July 18, 2022. This version: \today}

\address[D. Kurisu]{Center for Spatial Information Science, The University of Tokyo\\
5-1-5, Kashiwanoha, Kashiwa-shi, Chiba 277-8568, Japan.
}
\email{daisukekurisu@csis.u-tokyo.ac.jp}
\address[Y. Matsuda]{Graduate School of Economics and Management, Tohoku University\\
Sendai 980-8576, Japan.
}
\email{yasumasa.matsuda.a4@tohoku.ac.jp}

\begin{abstract}
This paper develops a general asymptotic theory of local polynomial (LP) regression for spatial data observed at irregularly spaced locations in a sampling region $R_n \subset \mathbb{R}^d$. We adopt a stochastic sampling design that can generate irregularly spaced sampling sites in a flexible manner including both pure increasing and mixed increasing domain frameworks. We first introduce a nonparametric regression model for spatial data defined on $\mathbb{R}^d$ and then establish the asymptotic normality of LP estimators with general order $p \geq 1$. We also propose methods for constructing confidence intervals and establishing uniform convergence rates of LP estimators. Our dependence structure conditions on the underlying processes cover a wide class of random fields such as L\'evy-driven continuous autoregressive moving average random fields. As an application of our main results, we discuss a two-sample testing problem for mean functions and their partial derivatives. 
\end{abstract}

\keywords{irregularly spaced spatial data, L\'evy-driven moving average random field, local polynomial regression, two-sample test\\ \indent
\textit{MSC2020 subject classifications}: 62M30, 62G08, 62G20}

\maketitle
\section{Introduction}

The goal of this paper is to develop a general asymptotic theory for local polynomial (LP) estimators of any order $p \geq 1$ for spatial data under irregular sampling on $\mathbb{R}^d$. We propose a nonparametric regression model for spatial data $\{Y(\bm{x}_{n,i})\}_{i=1}^{n}$ observed at irregularly spaced sampling sites $\{\bm{x}_{n,i}\}_{i=1}^{n}$ over a sampling region $R_n \subset \mathbb{R}^d$ ($d \geq 1$). Precisely, each $Y(\bm{x}_{n,i})$ is explained by the sum of a deterministic spatial trend function (i.e. mean function), a random field on $\mathbb{R}^d$ that represents spatial dependence, and a location specific measurement error (see Section \ref{sec:model} for details). In many scientific fields, such as ecology, geology, meteorology, and seismology, spatial samples are often collected over irregularly spaced points from continuous random fields because of physical constraints. To cope with irregularly spaced sampling sites, we adopt the stochastic sampling scheme of \cite{La03a}, which allows the sampling sites to have a non-uniform density in the sampling region and allows the number of sampling sites $n$ to grow at a different rate from the volume of the sampling region $A_n$. We design this scheme to accommodates both the pure increasing domain case ($\lim_{n \to \infty}A_n/n = \kappa \in (0,\infty)$) and the mixed increasing domain case ($\lim_{n \to \infty}A_n/n = 0$).  We note that this scheme covers possible asymptotic regimes that would validate asymptotic inference for spatial data. Although the infill asymptotics is excluded from our regime, our sampling design is general enough as it is known that the infill asymptotics does not work for that of even sample mean (cf.\cite{La96}). Refer to \cite{La03b}, \cite{LaZh06}, \cite{MaYa09}, \cite{BaLaNo15}, \cite{KuKaSh21}, and \cite{Ku22a} for discussions on the stochastic spatial sampling design. Further, our model can be seen as a spatial extension of locally stationary time series introduced in \cite{Da97}.

The contributions of this paper are as follows. First, we (i) establish the asymptotic normality of LP estimators of the mean function of the proposed model, (ii) construct consistent estimators of their asymptotic variances, and (iii) derive uniform convergence rates of LP estimators over a compact set. The results (i) and (ii) enable us to evaluate the bias and variance/covariance matrix (of the asymptotic distribution) of LP estimators and, as a result, to construct confidence intervals of LP estimators, which would work for a hypothesis testing on the mean function. We discuss a two-sample test for the partial derivatives as well as the mean function as an application of our results. Additionally, in the literature of causal inference, local polynomial fitting is known as an important tool to analyze average treatment effect of interventions, an example of which is the regression discontinuity desings (RDDs) (cf. \cite{HaTova01} and \cite{CaCaTi14}). Existing methods for RDDs often assume i.i.d. even for spatial data (cf. \cite{KeTi15} and \cite{EhSe18}). We claim our results pave the way for a new framework of RDDs for spatially dependent data. To establish the result (iii), we first consider general kernel estimators and derive their uniform convergence rates. The uniform convergence rates of LP estimators can be given as special cases of the results. Since the general estimators include many kernel-based estimators such as, kernel density, local constant (LC), local linear (LL), and LP estimators for random fields on $\mathbb{R}^d$ with irregularly spaced sampling sites, the results are of independent theoretical interest. We note that the general results are also useful for evaluating both the bias and variance terms of LP estimators. Particularly, the results on uniform convergence rates enable us to predict the values of the mean function uniformly on a spatial region that does not contain sampling sites. 

Second, we provide examples of random fields that satisfy the mixing assumptions under which the asymptotic normality of LP estimators will be established. Specifically, we show that a broad class of L\'evy-driven moving average (MA) random fields, which include continuous autoregressive moving average (CARMA) random fields (cf. \cite{BrMa17}), satisfies our assumptions. The CARMA random fields are known as a rich class of models for spatial data that can represent non-Gaussian random fields by introducing non-Gaussian L\'evy random measures (cf. \cite{BrMa17}, \cite{MaYa18}, and \cite{Ku22a}).
However, mixing properties of L\'evy-driven MA random fields have not been investigated since it is often difficult to check mixing conditions in the ways considered by \cite{LaZh06} and \cite{BaLaNo15} for general (possibly non-Gaussian) random fields on $\mathbb{R}^d$, which will be discussed later from the viewpoint of our theoretical analysis. We show that a wide class of L\'evy-driven MA random fields can be approximated by $m_n$-dependent random fields with $m_n \to \infty$ as $n \to \infty$. We claim that the approximation will work for the flexible modeling of nonparametric, nonstationary and possibly non-Gaussian spatial data on $\mathbb{R}^d$ by addressing an open question on dependence structure of statistical models built on L\'evy-driven MA random fields.

\subsection*{Connections to the literature}
There is fairly extensive literature on LC, LL, and LP estimators for dependent data. For stationary and regularly spaced time series (this case corresponds to stationary random fields with regular sampling on $\mathbb{Z}$), we refer to \cite{Ha08} and \cite{ZhWu08} for LC estimators and \cite{Ma96a, Ma96b}, and \cite{MaFa97} for LP estimators. For nonstationary and regularly spaced time series, we refer to \cite{Kr09} and \cite{Vo12} for LC estimators, and \cite{ZhWu09} and \cite{ZhWu15} for LL estimators of quantile curves and conditional mean functions, respectively. For stationary spatial data with regular sampling on $\mathbb{Z}^d$, we refer to \cite{MaSt10} for local constant (LC) estimation of the spatial trend function with stationary and spatially dependent errors, and \cite{LuCh02, LuCh04} for LC estimation and \cite{HaLuTr04} for local linear (LL) estimation of the conditional mean function with covariates, and \cite{HaLuYu09} for LL estimation of the conditional quantile function with covariates. For stationary spatial data with irregular sampling on $\mathbb{Z}^d$, we refer to \cite{MaSeOu17} for LL estimation of the conditional mean function with covariates. For nonstationary spatial data with (possibly) irregular sampling on $\mathbb{Z}^d$, we refer to \cite{Ro11} for LC estimation and \cite{Je12} for LL estimation of the conditional mean function with covariates. For spatial data with irregular sampling on $\mathbb{R}^d$, we refer to \cite{Ku19b} and \cite{Ku22a} who investigate LC estimators for the conditional mean function with stationary and nonstationary covariates, respectively. There is a large number of studies on the parametric estimation of the trend function in a spatial trend model with stationary and spatially dependent errors for spatial data on $\mathbb{R}^d$ (e.g. \cite{MaMa84}, \cite{DiTaMo98}, and \cite{Zh02}, just to name a few) and existing results on local polynomial (LP) estimators are available only for stationary random fields under regular sampling on $\mathbb{Z}$, i.e., regularly spaced stationary time series, while no studies on LL and LP estimation of the trend function in a spatial trend model with stationary and spatially dependent errors have been known under irregular sampling on $\mathbb{R}^d$ with $d \geq 2$.

To the best of our knowledge, our work is the first attempt to establish an asymptotic theory on local polynomial fitting for the spatial trend function of spatial data on $\mathbb{R}^d$ by (i) establishing the asymptotic normality and uniform convergence rates of LP estimators, (ii) providing a way to construct confidence intervals of LP estimators, and (iii) showing the applicability of our theoretical results to a wide class of L\'evy-driven MA random fields. From a theoretical point of view, this paper has advantages over the existing studies of \cite{La03a} and \cite{LaZh06} in the fields of irregularly spaced data analysis. Specifically, (i) we extend the coupling technique used in \cite{Yu94} for time series to that for irregularly spatial data to establish uniform convergence rates of LP estimators. The difficulties in the extension come from no natural ordering for spatial data and the number of observations in each block constructed is random, and hence our approach to blocking construction for establishing uniform rates is quite different from those in \cite{La03a} and \cite{LaZh06} whose proofs essentially rely on approximating the characteristic function of the weighted sample mean by that of independent blocks. (ii) We have confirmed concrete examples of random fields that satisfy our assumptions in detail. Verification of our regularity conditions to L\'evy-driven MA fields is indeed non-trivial and relies on several probabilistic techniques from L\'evy process theory and theory of infinitely divisible random measures (cf. \cite{Be96}, \cite{Sa99}, and \cite{RaRo89}). 

The rest of the paper is organized as follows. In Section \ref{sec:settings}, we introduce our nonparametric regression model for spatial data with irregularly spaced sampling sites. In Section \ref{sec:LPR}, we define local polynomial estimators as solutions of a multivariate weighted least squares problem. In Section \ref{sec:main}, we establish the asymptotic normality of LP estimators. In Section \ref{sec:LP-unif}, we provide the uniform convergence rates of LP estimators and construct estimators of their asymptotic variances. Appendix includes the proof of the asymptotic normality of LP estimators (Theorem \ref{thm: LP-CLT}). The supplementary material contains discussion on a two-sample test for the mean functions and their partial derivatives and examples of the random fields that satisfies our assumptions, and proofs for other results.

\subsection{Notation}
For any vector $\bm{x} = (x_{1},\dots, x_{q})' \in \R^{q}$, let $|\bm{x}| =\sum_{j=1}^{q}|x_{j}|$ and $\|\bm{x}\| =\sqrt{\sum_{j=1}^{q}x_j^2}$ 
denote the $\ell^{1}$-norm and $\ell^{2}$-norms of $\bm{x}$, respectively. For any set $A \subset \R^{d}$ and any vector $\bm{a}=(a_1,\dots, a_d)' \in (0,\infty)^d$, let $|A|$ denote the Lebesgue measure of $A$, let $[\![A]\!]$ denote the number of elements in $A$, and let $\bm{a}A = \{(a_1x_1,\dots a_dx_d): \bm{x} = (x_1,\dots,x_d) \in A\}$. 
For any positive sequences $a_{n}, b_{n}$, we write $a_{n} \lesssim b_{n}$ if there is a constant $C >0$ independent of $n$ such that $a_{n} \leq Cb_{n}$ for all $n$,  $a_{n} \sim b_{n}$ if $a_{n} \lesssim b_{n}$ and $b_{n} \lesssim a_{n}$. For a sequence of random variables $\{\bm{X}_i\}_{i \geq 1}$, let $\sigma(\{\bm{X}_{i}\}_{i \geq 1})$ denote the $\sigma$-field generated by $\{\bm{X}_i\}_{i \geq 1}$. Let $E_{\bm{X}}$ denote the expectation with respect to a sequence of random variables $\{\bm{X}_{i}\}_{i \geq 1}$ and let $P_{\cdot \mid \bm{X}}$ and $E_{\cdot \mid \bm{X}}$ denote the conditional probability and expectation given $\sigma(\{\bm{X}_{i}\}_{i \geq 1})$, respectively. For any real-valued random variable $X$ and $\tau \in (0,1)$, let $q_{1-\tau} = \inf \{x \in \mathbb{R}: P(X \leq x) \geq 1-\tau\}$ be the $(1-\tau)$-quantile of $X$. For $a \in \mathbb{R}$ and $b>0$, we use the shorthand notation $[a \pm b] = [a-b,a+b]$. 

\section{Settings}\label{sec:settings}

In this section, we discuss the mathematical settings of our model (Section \ref{sec:model}), sampling design (Section \ref{sec:sampling}), and spatial dependence structure (Section \ref{sec:dependence}). 

\subsection{Model}\label{sec:model}

Usually it is impossible to estimate consistently a model for nonstationary processes, since the domain of functions to be estimated gets larger. Dahlhaus avoids the difficulty by desigining a function over a fixed interval in the following way.
\cite{Da97} introduced a locally stationary process with a time-varying mean function for the modeling of nonstationary time series: $Y_T(t) = m\left({t \over T}\right) + \xi_{T}(t),\ t = 1,\dots,T,$
where $m:[0,1] \to \mathbb{R}$ is a (time-varying) mean function and $\{\xi_{T}(t)\}$ is a sequence of zero-mean locally stationary time series with a time-varying transfer function (see Definition 2.1 in \cite{Da97} for details). The model setting of $m(t/T)$ instead of $m(t)$ makes the mean function have the fixed domain of $[0,1]$, which provides the asymptotic scheme on which consistent estimation is available. We extend his framework to spatial data with irregular sampling on $\mathbb{R}^d$. 

In particular, consider the following nonparametric regression model: 
\begin{align}\label{eq:model}
Y(\bm{x}_{n,i}) &= m\left({\bm{x}_{n,i} \over A_n}\right) + \eta\left({\bm{x}_{n,i} \over A_n}\right)e(\bm{x}_{n,i}) + \sigma_{\varepsilon}\left({\bm{x}_{n,i} \over A_n}\right)\varepsilon_{i},\\
&:= m\left({\bm{x}_{n,i} \over A_n}\right) + e_{n,i} + \varepsilon_{n,i},\ \bm{x}_{n,i} = (x_{ni,1},\dots, x_{ni,d})' \in R_n,\ i=1,\dots,n, \nonumber
\end{align}
where $R_n = \prod_{j=1}^{d}[-A_{n,j}/2,A_{n,j}/2]$, $A_n = \prod_{j=1}^{d}A_{n,j}$, ${\bm{x}_{n,i} \over A_n} = \left({x_{ni,1} \over A_{n,1}},\dots, {x_{ni,d} \over A_{n,d}}\right)'$ with $A_{n,j} \to \infty$ as $n \to \infty$, $m: [-1/2,1/2]^d \to \mathbb{R}$ is the mean function, $\bm{e} = \{e(\bm{x}): \bm{x} \in \mathbb{R}^d\}$ is a stationary random field defined on $\mathbb{R}^d$ with $E[e(\bm{x})] = 0$ and $E[e^2(\bm{x})]=1$ for any $\bm{x} \in \mathbb{R}^d$, $\eta:[-1/2,1/2]^d \to (0,\infty)$ is the variance function of spatially dependent random variables $\{e_{n,i}\}$, $\{\varepsilon_{i}\}$ is a sequence of i.i.d. random variables such that $E[\varepsilon_{i}] = 0$ and $E[\varepsilon_{i}^2] = 1$, and $\sigma_{\varepsilon}: [-1/2,1/2]^d \to (0,\infty)$ is the variance function of random variables $\{\varepsilon_{n,i}\}$. The mean function $m$ represents deterministic spatial trend, the random field $\bm{e}$ represents spatial correlation, and the random variables $\{\varepsilon_{n,i}\}$ can represent location specific measurement error. 


\begin{remark}[Discussion on the model]
Our model, simplified for brief arguments here, is given by, for $\bm{x}_{n,i}\in R_n$, $Y(\bm{x}_{n,i})=m(\bm{x}_{n,i}/A_n)+e(\bm{x}_{n,i})+\varepsilon_{n,i}$, $i=1,\ldots,n$, where $e(\bm{x})$ is a stationary random field on $\mathbb{R}^d$ and $\varepsilon_i$ is a sequence of i.i.d. random variables. The trend function to be estimated in our model depends on the sampling region and hence the population model depends on the sample size. Discussions regarding models dependent on sample sizes have been prevalent in the context of nonstationary time series analysis. Furthermore, it is worth noting that no known asymptotic regime exists to validate our local polynomial estimation when the error incorporates a stationary random field component. On the other hand, for i.i.d. error cases without stationary components, it is known that our local polynomial estimation is validated asymptotically as the sample size tends to be infinity over a fixed domain $D$, given as $Y(\bm{x}_i)=m(\bm{x}_i)+\varepsilon_i$, $\bm{x}_i\in D$, $i=1,\ldots,n$. We contribute by providing an asymptotic framework in validating the local polynomial estimation for spatial data that includes stationary random field components in the error.

The idea of the dependencies of population models on sample sizes 
was proposed by \cite{Da97}, which is reviewed in \cite{Da12} with recent developments. The papers introduced locally stationary processes to tackle the difficulties caused by time-varying features in nonstationary time series. We apply the idea of modeling local stationary processes to our local polynomial estimation. To derive CLT under our setting, we need to satisfy the two conflicting necessities for asymptotic validations. A trend function must be on a fixed domain, while a stationary random field component needs to have an increasing domain. We employ the idea of local stationarity works to satisfy the conflicting necessities.
\end{remark}

We assume the following conditions on the mean function $m$, the variance function $\eta$, and $\{\varepsilon_{n,j}\}$:
\begin{assumption}\label{Ass:mean-func}
Let $U_{\bm{z}}$ be a neighborhood of $\bm{z}=(z_1,\dots,z_d)\in (-1/2,1/2)^d$. 
\begin{itemize}
\item[(i)] The mean function $m$ is $(p+1)$-times continuously partial differentiable on $U_{\bm{z}}$ and define $\partial_{j_1\dots j_L}m(\bm{z}):= \partial m(\bm{z})/\partial z_{j_1}\dots \partial z_{j_L}$, $1 \leq j_1,\dots, j_L \leq d$, $0 \leq L \leq p+1$. When $L=0$, we set $\partial_{j_1 \dots j_L}m(\bm{z}) = \partial_{j_0}m(\bm{z})= m(\bm{z})$. 
\item[(ii)] The function $\eta$ is continuous over $U_{\bm{z}}$ and $\eta(\bm{z})>0$. 
\item[(iii)] The random variables $\{\varepsilon_{i}\}_{i = 1}^{n}$ are i.i.d. with $E[\varepsilon_1] = 0$,  $E[\varepsilon_1^2] = 1$, $E[|\varepsilon_1|^{q_1}]<\infty$ for some integer $q_1 > 4$, and the function $\sigma_{\varepsilon}(\cdot)$ is continuous over $U_{\bm{z}}$ with $\sigma_{\varepsilon}(\bm{z})>0$. 
\end{itemize}
\end{assumption}

\subsection{Sampling design}\label{sec:sampling}

To account for irregularly spaced data, we consider a stochastic sampling design. First, we define the sampling region $R_{n}$. For $j=1,\dots, d$, let $\{A_{n,j}\}_{n \geq 1}$ be a sequence of positive numbers such that $A_{n,j} \to \infty$ as $n \to \infty$.  We consider the following set as the sampling region. 
\begin{align}\label{sampling-region}
R_{n} = \prod_{j=1}^{d}[-A_{n,j}/2, A_{n,j}/2]. 
\end{align}
Next, we introduce our (stochastic) sampling designs. Let $g(\bm{z}) = g(z_1,\dots,z_d)$ be a probability density function on $R_{0} = [-1/2,1/2]^{d}$, and let $\{\bm{X}_{n,i}\}_{i \geq 1}$ be a sequence of i.i.d. random vectors with probability density $A_n^{-1}g(\bm{x} /A_n) = A_n^{-1}g(x_1/A_{n,1},\dots, x_d/A_{n,d})$ where $A_n = \prod_{j=1}^{d}A_{n,j}$. We assume that the sampling sites $\bm{x}_{n,1},\hdots, \bm{x}_{n,n}$ are obtained from the realizations of random vectors $\bm{X}_{n,1},\hdots, \bm{X}_{n,n}$. To simplify the notation, we will write $\bm{x}_{n,i}$ and $\bm{X}_{n,i}$ as $\bm{x}_i=(x_{i,1},\dots, x_{i,d})'$ and $\bm{X}_{i}=(X_{i,1},\dots,X_{i,d})'$, respectively. 

We summarize conditions on the stochastic sampling design as follows:
\begin{assumption}\label{Ass:sample}
Recall that $U_{\bm{z}}$ is a neighborhood of $\bm{z} \in (-1/2,1/2)^d$. Let $g$ be a probability density function with support $R_0=[-1/2,1/2]^d$. 
\begin{itemize}
\item[(i)] $A_n/n \to \kappa \in [0,\infty)$ as $n \to \infty$, 
\item[(ii)] $\{\bm{X}_{i}=(X_{i,1},\dots,X_{i,d})'\}_{i = 1}^{n}$ is a sequence of i.i.d. random vectors with density $A_n^{-1}g(\cdot/A_n)$ and $g$ is continuous over $U_{\bm{z}}$ and $g(\bm{z})>0$.
\item[(iii)] $\{\bm{X}_i\}_{i=1}^{n}$, $\bm{e} = \{e(\bm{x}): \bm{x} \in \mathbb{R}^{d}\}$, and $\{\varepsilon_i\}_{i=1}^{n}$ are mutually independent. 
\end{itemize}
\end{assumption}

Condition (i) implies that our sampling design allows both the pure increasing domain case ($\lim_{n \to \infty}A_{n}/n = \kappa \in (0,\infty)$) and the mixed increasing domain case ($\lim_{n \to \infty}A_n/n = 0$). Condition (ii) implies that the sampling density can be nonuniformly distributed over the sampling region $R_n=\prod_{j=1}^{d}[-A_{n,j}/2, A_{n,j}/2]$. The definition (\ref{sampling-region}) is only for convenience, since it is possible to consider sampling regions of various shapes including non-standard shapes (e.g., ellipsoids, polyhedrons, and non-convex sets) by adjusting the support of the density $g$. 


\begin{remark}
In \cite{LuTj14}, they show asymptotic normality of a kernel density estimator of a strictly stationary random field on $\mathbb{R}^2$ under the domain expanding and infill (DEI) asymptotics, which is a non-stochastic design for irregularly spaced observations. In the DEI asymptotics, it is assumed from the outset that the sampling sites are countably infinite on $\mathbb{R}^d$, whereas in the mixed increasing domain (MID) asymptotics, a finite number of points within a finite observation region $R_n$ are assumed to be obtained as sampling sites. We see the DEI asymptotics as an alternative that may possibly enable us to construct a consistent estimator for
the mean function $m(x_0), x_0\in\mathbb{R}^2$ without imposing the mean function dependent on $A$, i.e., in the more reasonable setting of $Y(x_i)=m(x_i)+e(x_i)+\varepsilon_i$, $i=1,\ldots,n$, $x_i\in\mathbb{R}^2$. However, we expect in this alternative case that the rate of convergence will be differently specified as $\sqrt{nh}$. We leave the rigorous proof to future studies.
\end{remark}

\subsection{Dependence structure}\label{sec:dependence}

We assume that random field $\bm{e}$ satisfies a mixing condition. First, we define the $\alpha$- and $\beta$-mixing coefficients for the random field $\bm{e}$. Let $\mathcal{F}_{\bm{e}}(T) = \sigma(\{e(\bm{x}): \bm{x} \in T\})$ be the $\sigma$-field generated by the variables $\{e(\bm{x}):\bm{x} \in T\}$, $T \subset \mathbb{R}^d$. For any two subsets $T_1$ and $T_2$ of $\mathbb{R}^d$, let 
$\bar{\alpha}(T_1,T_2) = \sup\{|P(A \cap B) - P(A)P(B)|: A \in \mathcal{F}_{\bm{e}}(T_1), B \in \mathcal{F}_{\bm{e}}(T_2)\}$, $\bar{\beta}(T_1,T_2) = \sup {1 \over 2}\sum_{j=1}^{J}\sum_{k=1}^{K}|P(A_{j}\cap B_{k}) - P(A_{j})P(B_{k})|$ where the supremum for $\bar{\beta}(T_1,T_2)$ is taken over all pairs of (finite) partitions $\{A_{1},\hdots,A_{J}\}$ and $\{B_{1},\hdots, B_{K}\}$ of $\mathbb{R}^{d}$ such that $A_{j} \in \mathcal{F}_{\bm{e}}(T_1)$ and $B_{k} \in \mathcal{F}_{\bm{e}}(T_2)$.
The $\alpha$- and $\beta$-mixing coefficients of the random field $\bm{e}$ are defined as $\alpha(a;b) = \sup\{\bar{\alpha}(T_1,T_2): d(T_1,T_2) \geq a, T_1,T_2 \in \mathcal{R}(b)\}$, $\beta(a;b) = \sup\{\bar{\beta}(T_1,T_2): d(T_1,T_2) \geq a, T_1,T_2 \in \mathcal{R}(b)\}$ where $a,b>0$, $d(T_1,T_2) = \inf\{|\bm{x} - \bm{y}| : \bm{x} \in T_1, \bm{y} \in T_2\}$, and $\mathcal{R}(b)$ is the collection of all the finite disjoint unions of cubes in $\mathbb{R}^d$ with a total volume not exceeding $b$. Moreover, we assume that there exist non-increasing functions $\alpha_1$ and $\beta_1$ with $\alpha_1(a), \beta_1(a) \to 0$ as $a \to \infty$ and non-decreasing functions $\varpi_1$ and $\varpi_2$ (that may be unbounded) such that $\alpha(a;b) \leq \alpha_1(a)\varpi_1(b)$, $\beta(a;b) \leq \beta_1(a)\varpi_2(b)$. See the supplementary material for a discussion on the $\alpha$- and $\beta$-mixing coefficients. 

For the asymptotic normality of LP estimators, we assume the following conditions for the random field $\bm{e}$: 
\begin{assumption}\label{Ass:RF}
For $j = 1,\dots, d$, let $\{A_{n1,j}\}_{n \geq 1}$ and $\{A_{n2,j}\}_{n \geq 1}$ be sequences of positive numbers such that $\min\left\{A_{n2,j},  {A_{n1,j} \over A_{n2,j}}\right\} \to \infty$ as $n \to \infty$. 
\begin{itemize}
\item[(i)] The random field $\bm{e}$ is stationary and $E[|\bm{e}(\bm{0})|^{q_2}]<\infty$ for some integer $q_2 > 4$. 
\item[(ii)] Define $\sigma_{\bm{e}}(\bm{x}) = E[e(\bm{0})e(\bm{x})]$. Assume that $\sigma_{\bm{e}}(\bm{0})=1$ and $\int_{\mathbb{R}^d} |\sigma_{\bm{e}}(\bm{v})|d\bm{v}<\infty$.
\item[(iii)] The random field $\bm{e}$ is $\alpha$-mixing with mixing coefficients $\alpha(a;b)$ such that as $n \to \infty$, 
\begin{align*}
A_n^{(1)} \left(\alpha_1^{1-2/q}(\underline{A}_{n2}) + \sum_{k=\underline{A}_{n1}}^{\infty}k^{d-1}\alpha_1^{1-2/q}(k)\right)\varpi_1^{1-2/q}(A_n^{(1)}) \to 0, 
\end{align*}
where $q = \min\{q_1,q_2\}$, $A_n^{(1)} = \prod_{j=1}^{d}A_{n1,j}$, $\underline{A}_{n1}=\min_{1 \leq j \leq d}A_{n1,j}$, $\underline{A}_{n2}=\min_{1 \leq j \leq d}A_{n2,j}$.
\end{itemize}
\end{assumption}

The sequences $\{A_{n1,j}\}$ and $\{A_{n2,j}\}$ will be used in the large-block-small-block argument, which is commonly used in proving CLTs for sums of mixing random variables. Specifically, $A_{n1,j}$ corresponds to the side length of large blocks, while $A_{n2,j}$ corresponds to the side length of small blocks. In the supplementary material, we provide examples of random fields that satisfy Assumptions \ref{Ass:RF} and \ref{Ass:band} below. In particular, a wide class of L\'evy-driven moving average (MA) random fields that includes continuous autoregressive and moving average (CARMA) random fields (cf. \cite{BrMa17}) satisfies our assumptions (see the supplementary material for details). 

\section{Local polynomial regression of order $p$}\label{sec:LPR}

In this section, we introduce local polynomial (LP) estimators of order $p \geq 1$ for the estimation of the mean function $m$ of the model (\ref{eq:model}) and their partial derivatives.

Define $D = [\![\{(j_1,\dots ,j_L): 1 \leq j_1 \leq \dots \leq j_L \leq d, 0 \leq L \leq p\}]\!]$, $\bar{D} =[\![\{(j_1,\dots,j_{p+1}): 1 \leq j_1 \leq \dots \leq j_{p+1} \leq d\}]\!]$,
$(s_{j_1\dots j_L1},\dots, s_{j_1\dots j_Ld}) \in \mathbb{Z}_{\geq 0}^d$ such that $s_{j_1\dots j_Lk} = [\![\{j_\ell : j_\ell = k, 1\leq \ell \leq L\}]\!]$, and define 
$\bm{s}_{j_1\dots j_L}!=\prod_{k=1}^ds_{j_1\dots j_Lk}!$.
When $L = 0$, we set $(j_1,\dots,j_L) = j_0 = 0$ and $\bm{s}_{j_1\dots j_L}! = 1$. Note that $\sum_{k=1}^{d}s_{j_1\dots j_L k} = L$. Further, for $p \geq 1$ and $\bm{z} \in [-1/2,1/2]^d$, define 
\begin{align*}
\bm{M}(\bm{z}) &:= \left(m(\bm{z}),\partial_1 m(\bm{z}),\dots, \partial_d m(\bm{z}), {\partial_{11}m(\bm{z}) \over 2!}, {\partial_{12}m(\bm{z}) \over 1!1!}, \dots, {\partial_{dd}m(\bm{z}) \over 2!}, \right. \\ 
&\left. \quad \quad \dots, {\partial_{1\dots1}m(\bm{z}) \over p!}, {\partial_{1\dots 2}m(\bm{z}) \over (p-1)!1! }\dots, {\partial_{d\dots d}m(\bm{z}) \over p!} \right)'\\
&= \left({1 \over \bm{s}_{j_1\dots j_L}!}\partial_{j_1,\dots j_L}m(\bm{z})\right)'_{1\leq j_1 \leq \dots \leq j_L\leq d, 0 \leq L \leq p} \in \mathbb{R}^{D}.
\end{align*} 
We define the local polynomial regression estimator of order $p$ for $\bm{M}(\bm{z})$ as a solution of the following problem:  
\begin{align}\label{LL-minimize}
\hat{\bm{\beta}}(\bm{z}) &:= \argmin_{\bm{\beta} \in \mathbb{R}^D}\sum_{i=1}^{n}\! \left(\!Y(\bm{X}_i) \!-\! \sum_{L=0}^{p}\sum_{1 \leq j_1 \leq \dots \leq j_L \leq d}\!\!\!\!\!\!\!\!\!\beta_{j_1\dots j_L}\! \prod_{\ell=1}^{L} \!\! \left(\!{X_{i,j_\ell} \!-\! A_{n,j_\ell}z_{j_\ell} \over A_{n,j_\ell}}\! \right)\!\! \right)^2\!\!\! K_{Ah}\!\left(\bm{X}_i \!-\! A_n\bm{z}\right)\\
&= (\hat{\beta}_0(\bm{z}), \hat{\beta}_1(\bm{z}),\dots, \hat{\beta}_d(\bm{z}),\hat{\beta}_{11}(\bm{z}),\dots, \hat{\beta}_{dd}(\bm{z}),\dots,\hat{\beta}_{1\dots1}(\bm{z}),\dots,\hat{\beta}_{d\dots d}(\bm{z}))' \nonumber \\
&= (\hat{\beta}_{j_1\dots j_L}(\bm{z}))'_{1\leq j_1\leq \dots \leq j_L\leq d, 0 \leq L \leq p}, \nonumber
\end{align}
where $\bm{\beta} = (\beta_{j_1\dots j_L})'_{1 \leq j_1 \leq \dots \leq j_L \leq d, 0 \leq L \leq p}$, $K:\mathbb{R}^d \to \mathbb{R}$ is a kernel function, and each $h_j$ is a sequence of positive constants (bandwidths) such that $h_j \to 0$ as $n \to \infty$, and where  
\[
K_{Ah}(\bm{X}_i - A_n\bm{z}) = K\left({X_{i,1} - A_{n,1}z_1 \over A_{n,1}h_1}, \dots, {X_{i,d} - A_{n,d}z_d \over A_{n,d}h_d}\right)
\]
and $\sum_{1 \leq j_1 \leq \dots \leq j_L \leq d}\beta_{j_1\dots j_L}\prod_{\ell=1}^{L}(X_{i,j_\ell} - A_{n,j_\ell}z_{j_\ell})/A_{n,j_\ell} = \beta_0$ when $L=0$.

To compute LP estimators, we introduce some notations: $\bm{Y}:= (Y(\bm{X}_1),\dots,Y(\bm{X}_n))'$, 
\begin{align*}
\bm{X} &:= (\tilde{\bm{X}}_1,\dots,\tilde{\bm{X}}_n) \!=\! 
\left(
\begin{array}{ccc}
1 & \dots & 1  \\
{\left(\bm{X}_{1} - A_n\bm{z}\right)_1 \over A_n} & \dots & {\left(\bm{X}_{n} - A_n\bm{z} \right)_1 \over A_n}\\
\vdots & \dots & \vdots \\
{\left(\bm{X}_{1} - A_n\bm{z}\right)_p \over A_n} & \dots & {\left(\bm{X}_{n} - A_n\bm{z}\right)_p \over A_n}
\end{array}
\right) \!=\! 
\left(\!
\begin{array}{ccc}
1 & \dots & 1\\
\check{(\bm{X}_1-A_n \bm{z})} &\dots & \check{(\bm{X}_n - A_n\bm{Z})}
\end{array}
\! \right),\\
\bm{W} &:= \diag\left(K_{Ah}\left(\bm{X}_1 - A_n\bm{z} \right),\dots,K_{Ah}\left(\bm{X}_n - A_n\bm{z}\right)\right),
\end{align*}
where 
\begin{align*}
{\left(\bm{X}_i - A_n\bm{z} \right)_L \over A_n} &= \left(\prod_{\ell = 1}^{L}\left({X_{i,j_{\ell}} - A_{n,j_\ell}z_{j_{\ell}} \over A_{n,j_\ell}}\right)\right)'_{1 \leq j_1 \leq \dots \leq j_L \leq d}.
\end{align*}
The minimization problem (\ref{LL-minimize}) can be rewritten as $\hat{\bm{\beta}}(\bm{z}) = \argmin_{\bm{\beta} \in \mathbb{R}^D}(\bm{Y} - \bm{X}'\bm{\beta})'\bm{W}(\bm{Y} - \bm{X}'\bm{\beta}) =: \argmin_{\bm{\beta} \in \mathbb{R}^D}Q_n(\bm{\beta})$.
Then the first order condition of the problem (\ref{LL-minimize}) is given by ${\partial \over \partial \bm{\beta}}Q_n(\bm{\beta}) = -2\bm{X} \bm{W} \bm{Y} + 2\bm{X} \bm{W} \bm{X}'\bm{\beta} = 0$.
Hence the solution of the problem (\ref{LL-minimize}) is given by
\begin{align*}
\hat{\bm{\beta}}(\bm{z}) &= (\bm{X} \bm{W} \bm{X}')^{-1}\bm{X} \bm{W} \bm{Y}\\
&= \left[\sum_{i=1}^{n}K_{Ah}\left(\bm{X}_i - A_n\bm{z} \right)\tilde{\bm{X}}_i \tilde{\bm{X}}'_i\right]^{-1}\sum_{i=1}^{n}K_{Ah}\left(\bm{X}_i - A_n\bm{z}\right)\tilde{\bm{X}}_i Y(\bm{X}_i).
\end{align*}

We assume the following conditions on the kernel function $K$:
\begin{assumption}\label{Ass:kernel}
Let $K :\mathbb{R}^{d} \to \mathbb{R}$ be a kernel function such that
\begin{itemize}
\item[(i)] $\int K(\bm{z})d\bm{z} = 1$.
\item[(ii)] The kernel function $K$ is bounded and supported on $S_K \subset [-1/2,1/2]^d$ with $U_{\bm{z}} \subset S_K$.
\item[(iii)] Define $\kappa_0^{(r)}:= \int K^r(\bm{z})d\bm{z}$, $\kappa_{j_1\dots j_M}^{(r)}:= \int \prod_{\ell=1}^{M}z_{j_\ell}K^r(\bm{z})d\bm{z}$, and 
\[
\check{\bm{z}} := (1, (\bm{z})'_1,\dots, (\bm{z})'_p)',\ (\bm{z})_{L} = \left(\prod_{\ell=1}^{L}z_{j_\ell}\right)'_{1 \leq j_1 \leq \dots \leq j_L \leq d},\ 1 \leq L \leq p. 
\]
The matrix $S = \int \left(
\begin{array}{c}
1 \\
\check{\bm{z}}
\end{array}
\right)
(1\ \check{\bm{z}}')K(\bm{z})d\bm{z}$ is non-singular. 
\end{itemize}
\end{assumption}

\section{Main results}\label{sec:main}

In this section, we discuss asymptotic properties of LP estimators defined in Section \ref{sec:LPR}. In particular, we establish the asymptotic normality of LP estimator (Section \ref{sec:AN-LPR}). In the supplementary material, we discuss a two-sample test for the mean functions and their partial derivatives as an application of our main results.  

\subsection{Asymptotic normality of local polynomial estimators}\label{sec:AN-LPR}

We assume the following conditions for the sample size $n$, bandwidths $h_j$, constants $A_{n,j}$, $A_{n1,j}$, and $A_{n2,j}$, and mixing coefficients $\alpha(a;b)$: 
\noindent
\begin{assumption}\label{Ass:band}
Recall $q = \min\{q_1,q_2\}$, $A_n^{(1)}=\prod_{j=1}^{d}A_{n1,j}$, $\underline{A}_{n1}=\min_{1 \leq j \leq d}A_{n1,j}$. Define $\overline{A}_{n1}=\max_{1 \leq j \leq d}A_{n1,j}$, $\overline{A}_{n2}=\max_{1 \leq j \leq d}A_{n2,j}$, and $\overline{A_{n}h} =\max_{1 \leq j \leq d}A_{n,j}h_j$. As $n \to \infty$, 
\begin{itemize}
\item[(i)] $h_j \to 0$, ${A_{n,j}h_j \over A_{n1,j}} \to \infty$ for $1 \leq j \leq d$.
\item[(ii)] $nh_1 \dots h_d \to \infty$. 
\item[(iii)] $A_n h_1 \dots h_d \times h_{j_1}^2 \dots h_{j_p}^2 \to \infty$ for $1 \leq j_1 \leq \dots \leq j_{p} \leq d$. 
\item[(iv)] $A_n h_1 \dots h_d \times h_{j_1}^2 \dots h_{j_p}^2h_{j_{p+1}}^2 \to c_{j_1\dots j_{p+1}} \in [0,\infty)$ for $1 \leq j_1 \leq \dots \leq j_{p+1} \leq d$. 
\item[(v)] 
\begin{align}
&\left({A_n h_1 \dots h_d \over A_n^{(1)}}\right)\alpha_1(\underline{A}_{n2})\varpi_1(A_n h_1 \dots h_d) \to 0, \label{Cond(v)1}\\
&\left({A_n^{(1)} \over A_n h_1 \dots h_d}\right)\sum_{k=1}^{\overline{A}_{n1}}k^{2d-1}\alpha_1^{1-4/q}(k) \to 0, \label{Cond(v)2}\\ 
&\left\{\left({\overline{A}_{n1} \over \underline{A}_{n1}}\right)^d\left({\overline{A}_{n2} \over \overline{A}_{n1}}\right) + \left({A_n^{(1)} \over \underline{A}_{n1}^d}\right)\left({\left(\overline{A_n h}\right)^d \over A_n h_1 \dots h_d}\right)\left({\overline{A}_{n1} \over \overline{A_n h}}\right)\right\}\sum_{k=1}^{\overline{A}_{n1}}k^{d-1}\alpha_1^{1-2/q}(k) \to 0.\label{Cond(v)3}
\end{align}
\end{itemize} 
\end{assumption}

We need Condition (ii) to compute the asymptotic variances of LP estimators. Conditions (iii) and (iv) are concerned with the rates of convergence of variance and bias terms of LP estimators, respectively. Condition (v) is concerned with the large-block-small-block argument to show the asymptotic normality of LP estimators. Indeed, we use the condition (\ref{Cond(v)1}) to approximate a weighted sum of spatially dependent data of the form
\[\sum_{i=1}^{n}K_{Ah}\left(\bm{X}_i\right)H^{-1}
\left(
\begin{array}{c}
1 \\
\check{\bm{X}}_i
\end{array}
\right)(e_{n,i}+\varepsilon_{n,i})
\]
by a sum of independent large blocks where 
\[
H := \diag(1,h_1,\dots, h_d,h_1^2, h_1h_2, \dots,h_d^2,\dots, h_1^p, h_1^{p-1}h_2,\dots, h_d^p) \in \mathbb{R}^{D \times D}.
\]
The condition (\ref{Cond(v)2}) is used to show asymptotic normality of the sum of independent large blocks. The condition (\ref{Cond(v)3}) is used to show the asymptotic negligibility of a sum of small blocks. See the proof of Theorem \ref{thm: LP-CLT} for detailed definitions of large and small blocks. 

Throughout Sections \ref{sec:AN-LPR} and \ref{sec:AV-LPR}, we set $\bm{z} = \bm{0}$ without loss of generality. Extending the results in this section to the case $\bm{z} \in (-1/2,1/2)^{d}$ is straightforward. 

\begin{theorem}[Asymptotic normality of local polynomial estimators]\label{thm: LP-CLT}
Suppose Assumptions \ref{Ass:mean-func}, \ref{Ass:sample}, \ref{Ass:RF}, \ref{Ass:kernel}, and \ref{Ass:band} hold. 
Then, as $n \to \infty$, the following result holds:
\begin{align*}
&\sqrt{A_nh_1 \dots h_d}\left(H\left(\hat{\bm{\beta}}(\bm{0}) - \bm{M}(\bm{0})\right) - S^{-1}B^{(d,p)}\bm{M}_n^{(d,p)}(\bm{0})\right)\\ 
&\quad \stackrel{d}{\to} N\left( \left(
\begin{array}{c}
0 \\
\vdots \\
0
\end{array}
\right), \left\{{\kappa(\eta^2(\bm{0}) + \sigma_\varepsilon^2(\bm{0})) \over g(\bm{0})}+\eta^2(\bm{0})\int \sigma_{\bm{e}}(\bm{v})d\bm{v}\right\}S^{-1}\mathcal{K}S^{-1}\right), 
\end{align*}
where 
\begin{align*}
B^{(d,p)} &= \int \left(
\begin{array}{c}
1 \\
\check{\bm{z}}
\end{array}
\right)(\bm{z})'_{p+1}K(\bm{z})d\bm{z} \in \mathbb{R}^{D \times \bar{D}},\ \mathcal{K}= \int \left(
\begin{array}{c}
1 \\
\check{\bm{z}}
\end{array}
\right)
(1\ \check{\bm{z}}')K^2(\bm{z})d\bm{z} \in \mathbb{R}^{D \times D},\\
\bm{M}_n^{(d,p)}(\bm{z}) &= \left({\partial_{j_1\dots j_{p+1}}m(\bm{z}) \over \bm{s}_{j_1 \dots j_{p+1}}!}\prod_{\ell=1}^{p+1}h_{j_{\ell}}\right)'_{1 \leq j_1 \leq \dots \leq j_{p+1}\leq d}\\
&= \left({\partial_{1\dots 1}m(\bm{z}) \over (p+1)!}h_1^{p+1},{\partial_{1\dots 2}m(\bm{z}) \over p!}h_1^{p}h_2,\dots, {\partial_{d\dots d}m(\bm{z}) \over (p+1)!}h_d^{p}\right)' \in \mathbb{R}^{\bar{D}}. 
\end{align*} 
\end{theorem}

Theorem \ref{thm: LP-CLT} differs from the asymptotic normality of LP estimators under i.i.d. observations in several points. First, the convergence rates of LP estimators depends not on the sample size $n$ explicitly but on the volume of the sampling region $A_n$. Second, the asymptotic variance is represented as a sum of two components $\{\kappa(\eta^2(\bm{0}) + \sigma_{\varepsilon}^2(\bm{0}))\}S^{-1}\mathcal{K}S^{-1}/g(\bm{0})$ and $\eta^2(\bm{0})\left(\int \sigma_{\bm{e}}(\bm{v})d\bm{v}\right)S^{-1}\mathcal{K}S^{-1}$. When the sampling design satisfies the mixed increasing domain asymptotics, that is, $\kappa = 0$, then the asymptotic variance depends only on the second term, which represents the effect of the spatial dependence, and does not includes $\sigma_{\varepsilon}^2(\bm{0})$, the effect of the measurement error $\{\varepsilon_{n,j}\}$. This is completely different from i.i.d. case. We also note that the form of the asymptotic variance in Theorem \ref{thm: LP-CLT} is different from that of Theorem 4 in \cite{Ma96b} who investigates asymptotic properties of LP estimators for equidistant time series. Indeed, in his result, the variance term that corresponds to the second term of the asymptotic variance in our result does not appear. When the sampling design satisfies the pure increasing domain asymptotics, that is, $\kappa \in (0,\infty)$, then the asymptotic variance depends on both first and second terms. In this case, the asymptotic variance includes the effect of the sampling design $1/g(\bm{0})$, which implies that the more likely the sampling sites are distributed around $\bm{0}$, the more accurate the estimation of $M(\bm{0})$. Moreover, if $\eta(\cdot) \equiv 0$, then the asymptotic variance coincides with that of i.i.d. case.

\begin{remark}[General form of the mean squared error of $\partial_{j_1\dots j_L}\widehat{m}(\bm{0})$]
Define
\begin{align*}
\bm{b}_n^{(d,p)}(\bm{x}) &:= B^{(d,p)}\bm{M}_n^{(d,p)}(\bm{x})\\ 
&= \left(b_{n,0}(\bm{x}), b_{n,1}(\bm{x}),\dots,b_{n,d}(\bm{x}), \right. \\
&\left. \quad \quad \quad b_{n,11}(\bm{x}),b_{n,12}(\bm{x}),\dots,b_{n,dd}(\bm{x}), \dots, b_{n,1\dots,1}(\bm{x}), b_{n,1\dots 2}(\bm{x}),\dots,b_{n,d \dots d}(\bm{x})\right)'
\end{align*}
and let $e_{j_1\dots j_L}= (0,\dots,0,1,0,\dots,0)'$ be a $D$-dimensional vector such that $e_{j_1\dots j_L}'\bm{b}_n^{(d,p)}(\bm{x}) = b_{j_1\dots j_L}(\bm{x})$. 
Theorem \ref{thm: LP-CLT} yields that 
\begin{align*}
b_{n,j_1,\dots,j_L}(\bm{0}) &= \sum_{1 \leq j_{1,1} \leq \dots \leq j_{1,p+1}\leq d}{\partial_{j_{1,1}\dots j_{1,p+1}}m(\bm{0}) \over \bm{s}_{j_{1,1}\dots j_{1,p+1}}!}\prod_{\ell_1=1}^{p+1}h_{j_{1,\ell_1}}\kappa_{j_1\dots j_L j_{1,1} \dots j_{1,p+1}}^{(1)},\\
\end{align*}
for $1 \leq j_1 \leq \dots \leq  j_L \leq d$, $0 \leq L \leq p$, and the mean squared error (MSE) of LP estimator $\partial_{j_1\dots j_L}\widehat{m}(\bm{0})$ is given as follows: 
\begin{align}\label{MSE-mx}
&\text{MSE}(\partial_{j_1\dots j_L}\widehat{m}(\bm{0})) = E\left[\left(\partial_{j_1\dots j_L}m(\bm{0}) - \partial_{j_1\dots j_L}\widehat{m}(\bm{0})\right)^2\right] \nonumber \\
&= \left\{\bm{s}_{j_1\dots j_L}!{(S^{-1}e_{j_1\dots j_L})'B^{(d,p)}\bm{M}_n^{(d,p)}(\bm{0}) \over \prod_{\ell=1}^{L}h_{j_\ell}}\right\}^2 \nonumber \\ 
&\quad + \left({\kappa(\eta^2(\bm{0}) + \sigma_{\varepsilon}^2(\bm{0})) \over g(\bm{0})} + \eta^2(\bm{0})\int \sigma_{\bm{e}}(\bm{v})d\bm{v}\right)\left(\bm{s}_{j_1\dots j_L}!\right)^2{e'_{j_1\dots j_L}S^{-1}\mathcal{K}S^{-1}e_{j_1\dots j_L} \over A_nh_1\dots h_d \times \left(\prod_{\ell=1}^{L}h_{j_\ell}\right)^2}.
\end{align}
\end{remark}

\section{Uniform convergence rates of local polynomial estimators}\label{sec:LP-unif}

In this section, we derive the uniform convergence rates of LP estimators for the mean function of the model (\ref{eq:model}) and their partial derivatives. We note that these results can be derived as special cases of the results on the uniform convergence rates of more general kernel estimators provided in the supplementary material. Moreover, we construct estimators of the asymptotic variances of LP estimators (Section \ref{sec:AV-LPR}). We assume the following conditions on the mean function $m$, the variance function $\eta$, and $\{\varepsilon_{n,i}\}$:
\begin{assumption}\label{Ass:mean-func-unif}
\begin{itemize}
Recall $R_{0} = [-1/2,1/2]^d$. 
\item[(i)] The mean function $m$ is $(p+1)$-times continuously partial differentiable on $R_{0}$ and define $\partial_{j_1\dots j_L}m(\bm{z}):= \partial m(\bm{z})/\partial z_{j_1}\dots \partial z_{j_L}$, $1 \leq j_1,\dots, j_L \leq d$, $0 \leq L \leq p+1$. When $L=0$, we set $\partial_{j_1 \dots j_L}m(\bm{z}) = \partial_{j_0}m(\bm{z})= m(\bm{z})$. 
\item[(ii)] The function $\eta$ is continuous over $R_{0}$ and $\inf_{\bm{z} \in R_{0}}\eta(\bm{z})>0$. 
\item[(iii)] The sequence of random variables $\{\varepsilon_{i}\}_{i = 1}^{n}$ are i.i.d. with $E[\varepsilon_1] = 0$,  $E[\varepsilon_1^2] = 1$, $E[|\varepsilon_1|^{q_1}]<\infty$ for some integer $q_1 > 4$ and the function $\sigma_{\varepsilon}(\cdot)$ is continuous over $R_{0}$ and $\inf_{\bm{z} \in R_{0}}\sigma_{\varepsilon}(\bm{z})>0$. 
\end{itemize}
\end{assumption}

For the sampling sites $\{\bm{X}_i\}_{i=1}^{n}$, we assume the following conditions: 
\begin{assumption}\label{Ass:sample-unif}
Let $g$ be a probability density function with support $R_0=[-1/2,1/2]^d$. 
\begin{itemize}
\item[(i)] $A_n/n \to \kappa \in [0,\infty)$ as $n \to \infty$, 
\item[(ii)] $\{\bm{X}_{i}=(X_{i,1},\dots,X_{i,d})'\}_{i = 1}^{n}$ is a sequence of i.i.d. random vectors with density $A_n^{-1}g(\cdot/A_n)$ and $g$ is continuous and positive on $R_0$. 
\item[(iii)] $\{\bm{X}_i\}_{i=1}^{n}$, $\bm{e}=\{e(\bm{x}): \bm{x} \in \mathbb{R}^d\}$, and $\{\varepsilon_i\}_{i=1}^{n}$ are mutually independent. 
\end{itemize}
\end{assumption}

We also assume the following conditions on the bandwidth $h_j$ and the random field $\bm{e}=\{e(\bm{x}): \bm{x} \in \mathbb{R}^d\}$:  
\begin{assumption}\label{Ass:RF-unif}
For $j = 1,\dots, d$, let $\{A_{n1,j}\}_{n \geq 1}$, $\{A_{n2,j}\}_{n \geq 1}$ be sequence of positive numbers.  
\begin{itemize}
\item[(i)] The random field $\bm{e}$ is stationary and $E[|e(\bm{0})|^{q_2}]<\infty$ for some integer $q_2 > 4$. 
\item[(ii)] Define $\sigma_{\bm{e}}(\bm{x}) = E[e(\bm{0})e(\bm{x})]$. Assume that  $\int_{\mathbb{R}^d} |\sigma_{\bm{e}}(\bm{v})|d\bm{v}<\infty$.
\item[(iii)] $\min\left\{A_{n2,j},  {A_{n1,j} \over A_{n2,j}}, {A_{n,j}h_j \over A_{n1,j}}\right\} \to \infty$ as $n \to \infty$.
\item[(iv)] The random field $\bm{e}$ is $\beta$-mixing with mixing coefficients $\beta(a;b) \leq \beta_1(a)\varpi_2(b)$ such that as $n \to \infty$, $h_j \to 0$, $1 \leq j \leq d$, 
\begin{align}
&{A_n^{(1)} \over (\overline{A}_{n1})^d} \sim 1,\ {A_n^{{1 \over 2}}(h_1 \dots h_d)^{{1 \over 2}} \over n^{1/q_2}(\overline{A}_{n1})^d (\log n)^{{1 \over 2}+\iota}} \gtrsim 1\ \text{for some $\iota \in (0,\infty)$}, \label{unif-F2}\\
&{n^d A_n^{1-d/2}(h_1 \dots h_d)^{d/2} \over (\log n)^{d/2} A_n^{(1)}}\beta_1(\underline{A}_{n2})\varpi_2(A_n h_1 \dots h_d) \to 0, \label{unif-F3}
\end{align}
where $A_n^{(1)} =\prod_{j=1}^{d}A_{n1,j}$, $\overline{A}_{n1}=\max_{1 \leq j \leq d}A_{n1,j}$, $\underline{A}_{n1}=\min_{1 \leq j \leq d}A_{n1,j}$,\\ 
$\overline{A}_{n2} =\max_{1 \leq j \leq d}A_{n2,j}$, and $\underline{A}_{n2}=\min_{1 \leq j \leq d}A_{n2,j}$.
\end{itemize}
\end{assumption}

Condition (\ref{unif-F3}) is concerned with large-block-small-block argument for $\beta$-mixing sequences. In order to derive uniform convergence rates of LP estimators, more careful arguments on the effects of non-equidistant sampling sites are necessary than those for proving asymptotic normality and this also requires additional works in comparison with the equidistant time series or spatial data. We first approximate LP estimators excluding bias terms, which can be written as a sum of spatially dependent data, by a sum of independent blocks by extending the blocking technique in \cite{Yu94}(Corollary 2.7) that does not require regularly spaced sampling sites. Then we derive the uniform convergence rates of LP estimators by applying maximum inequalities for independent and possibly not identically distributed random variables to the independent blocks. In the supplementary material, we will show that a wide class of L\'evy-driven MA random fields satisfies our $\beta$-mixing conditions. 

We assume the following conditions on the kernel function $K$:
\begin{assumption}\label{Ass:kernel-unif}
Let $K :\mathbb{R}^{d} \to \mathbb{R}$ be a kernel function such that
\begin{itemize}
\item[(i)] $\int K(\bm{z})d\bm{z} = 1$.
\item[(ii)] The kernel function $K$ is bounded and supported on $[-C_K, C_K]^d \subset [-1/2,1/2]^d$ for some $C_K>0$. Moreover, $K$ is Lipschitz continuous on $\mathbb{R}^d$, i.e., $|K(\bm{v}_1) - K(\bm{v}_2)| \leq L_K|\bm{v}_1 - \bm{v}_2|$ for some $L_K \in (0,\infty)$ and all $\bm{v}_1,\bm{v}_2 \in \mathbb{R}^d$. 
\item[(iii)] Define $\kappa_0^{(r)}:= \int K^r(\bm{z})d\bm{z}$, $\kappa_{j_1,\dots, j_M}^{(r)}:= \int \prod_{\ell=1}^{M}z_{j_\ell}K^r(\bm{z})d\bm{z}$, and 
\[
\check{\bm{z}} := (1, (\bm{z})'_1,\dots, (\bm{z})'_p)',\ (\bm{z})_{L} = \left(\prod_{\ell=1}^{L}z_{j_\ell}\right)'_{1 \leq j_1 \leq \dots \leq j_L \leq d},\ 1 \leq L \leq p. 
\]
The matrix $S = \int \left(
\begin{array}{c}
1 \\
\check{\bm{z}}
\end{array}
\right)
(1\ \check{\bm{z}}')K(\bm{z})d\bm{z}$ is non-singular. 
\end{itemize}
\end{assumption}

The next result provides uniform convergence rates of LP estimators $\partial_{j_1 \dots j_L}\hat{m}(\bm{z})$. 
\begin{theorem}\label{thm:LPR-unif}
Define $\mathrm{T}_n = \prod_{j=1}^{d}[-1/2+C_Kh_j, 1/2-C_Kh_j]$. Suppose that Assumptions \ref{Ass:mean-func-unif}, \ref{Ass:sample-unif}, \ref{Ass:RF-unif}, and \ref{Ass:kernel-unif} hold with $q_1 \geq q_2$. Then for $1 \leq j_1 \leq \dots \leq j_L\leq d$, $0 \leq L \leq p$, as $n \to \infty$, we have
\begin{align*}
&\sup_{\bm{z} \in \mathrm{T}_n}\left|\partial_{j_1\dots j_L}\hat{m}(\bm{z}) - \partial_{j_1\dots j_L}\hat{m}(\bm{z})\right|\\ 
&\quad= O_p\left( {\sum_{1 \leq j_1 \leq \dots \leq j_{p+1}\leq d}\prod_{\ell=1}^{p+1}h_{j_\ell} \over \prod_{\ell=1}^{L}h_{j_\ell}} + \sqrt{{\log n \over A_n h_1 \dots h_d \left(\prod_{\ell=1}^{L}h_{j_\ell}\right)^2 }}\right). 
\end{align*}
\end{theorem}

\subsection{Estimation of asymptotic variances of LP estimators}\label{sec:AV-LPR}

An estimator of the asymptotic variance of the LP estimators $\hat{\bm{\beta}}(\bm{0})$ can be constructed. Define $\hat{g}(\bm{0}) = {1 \over nh_1 \dots h_d}\sum_{i=1}^{n}K_{Ah}(\bm{X}_i)$,
\begin{align*}
\hat{W}_{n,1}(\bm{0}) &= {A_n \over n^2h_1 \dots h_d}\sum_{i,j=1}^{n}K_{Ah}(\bm{X}_i)K_{Ah}(\bm{X}_j)\bar{K}_b(\bm{X}_i - \bm{X}_j)\\
&\quad \times \left(Y(\bm{X}_i) - \widehat{m}(\bm{X}_i/A_n)\right)\left(Y(\bm{X}_j) - \widehat{m}(\bm{X}_j/A_n)\right),
\end{align*}
where $\bar{K}(\bm{w}) :\mathbb{R}^d \to [0,1]$ is a kernel function, $\bar{K}_b(\bm{w}) = \bar{K}\left({w_1 \over b_1},\dots,{w_d \over b_d}\right)$, and $b_j$ is a sequence of positive constants such that $b_j \to \infty$ as $n \to \infty$. We assume the following conditions for $\bar{K}$:

\begin{assumption}\label{Ass: cov-kernel}
Let $\bar{K}:\mathbb{R}^d \to [0,1]$ is a continuous function such that
\begin{itemize}
\item[(i)] $\bar{K}(\bm{0}) = 1$, $\bar{K}(\bm{w}) = 0$ for $\|\bm{w}\|>1$. 
\item[(ii)] $|1 - \bar{K}(\bm{w})| \leq \bar{C}\|\bm{w}\|$ for $\|\bm{w}\| \leq \bar{c}$ where $\bar{C}$ and  $\bar{c}$ are some positive constants.  
\end{itemize}
\end{assumption}
An example of $\bar{K}$ is the Bartlett kernel: $\bar{K}(\bm{w})=(1-\|\bm{w}\|)$ for $\|\bm{w}\| \leq 1$ and $0$ for $\|\bm{w}\|>1$.

\begin{proposition}\label{prp: var-est}
Assume $b_j \to \infty$ and ${b_j \over  A_{n,j}h_j} \to 0$, $j=1,\dots,d$ as $n \to \infty$. Suppose that Assumptions \ref{Ass:RF} 
(iii) and \ref{Ass:band} (ii)-(v) hold with $\alpha$-mixing coefficients replaced by $\beta$-mixing coefficients, and that Assumptions \ref{Ass:mean-func-unif}, \ref{Ass:sample-unif}, \ref{Ass:RF-unif}, \ref{Ass:kernel-unif}, and \ref{Ass: cov-kernel} hold with $q_1 \geq q_2$. Then, as $n \to \infty$, the following result holds:
\begin{align*}
\hat{W}_n(\bm{0}):=  {(\kappa_0^{(2)})^{-1}\hat{W}_{n,1}(\bm{0})\over \hat{g}^2(\bm{0})}  \stackrel{p}{\to}  {\kappa(\eta^2(\bm{0}) + \sigma_{\varepsilon}^2(\bm{0})) \over g(\bm{0})} + \eta^2(\bm{0})\int \sigma_{\bm{e}}(\bm{v})d\bm{v}.
\end{align*}
\end{proposition}

Theorem \ref{thm: LP-CLT} and Proposition \ref{prp: var-est} enable us to construct confidence intervals of $\partial_{j_1 \dots j_L}m(\bm{0})$. Consider a confidence interval of the form
\begin{align*}
C_{n,j_1 \dots j_L}(1-\tau) = \left[\partial_{j_1 \dots j_L}\widehat{m}(\bm{0}) \pm \sqrt{{\hat{W}_n(\bm{0})\left(\bm{s}_{j_1\dots j_L}!\right)^2\left(e'_{j_1 \dots j_L}S^{-1}\mathcal{K}S^{-1}e_{j_1 \dots j_L}\right) \over A_n h_1 \dots h_d \left(\prod_{\ell=1}^{L}h_{j_\ell}\right)^2}}q_{1-\tau/2}\right],
\end{align*}
where $q_{1-\tau}$ is the $(1-\tau)$-quantile of the standard normal random variable. Then we can show the asymptotic validity of the confidence interval as follows: 
\begin{corollary}\label{cor: LP-CI}
Let $\tau \in (0,1)$. Under the assumptions of Proposition \ref{prp: var-est} with 
\[
A_n h_1 \dots h_d \left((S^{-1}e_{j_1\dots j_L})'B^{(d,p)}M_n^{(d,p)}(\bm{0}) \right)^2 \to 0
\] 
as $n \to \infty$. Then, $\lim_{n \to \infty}P(\partial_{j_1 \dots j_L}m(\bm{0}) \in C_{n,j_1 \dots j_L}(1-\tau)) = 1-\tau$. 
\end{corollary}

In the supplementary material, we see the finite sample properties of the confidence interval and find that it performs well.

\begin{remark}
As shown in Theorem 4.1, the expressions of asymptotic bias and variance of the LP estimators are very similar in structure to those from a standard random design for stationary time series and random fields. Therefore, we conjecture that plug-in methods to choose the bandwidth in such a design can be adapted to our setting.
\end{remark}

\section{Conclusion}

In this paper, we have advanced statistical theory of nonparametric regression for irregularly spaced spatial data. For this, we introduced a nonparametric regression model defined on a sampling region $R_n \subset \mathbb{R}^d$ and derived asymptotic normality and uniform convergence rates of the local polynomial estimators of order $p \geq 1$ for the mean function of the model under a stochastic sampling design. As an application of our main results, we discussed a two-sample test for the mean functions and their partial derivatives.  We also provided examples of random fields that satisfy our assumptions. In particular, our assumptions hold for a wide class of random fields that includes L\'evy-driven moving average random fields and popular Gaussian random fields as special cases. 

\clearpage

\appendix

\section{Proof of Theorem \ref{thm: LP-CLT}}

Now we prove Theorem \ref{thm: LP-CLT}. The proofs of other results are given in the supplementary material. 

\begin{proof}
Define $\bm{h}:= (h_1,\dots, h_d)'$ and for $\bm{x}, \bm{y} \in \mathbb{R}^{d}$, let $\bm{x} \circ \bm{y} = (x_1y_1 ,\dots, x_dy_d)'$ be the Hadamard product. Considering Taylor's expansion of $m(\bm{z})$ around $\bm{z}$, 
\begin{align*}
m(\bm{X}_i /A_n) &= (1, \check{\X}'_i)M(\bm{z}) + {1 \over (p+1)!}\sum_{1 \leq j_1 \leq \dots \leq j_{p+1}\leq d}{(p+1)! \over \bm{s}_{j_1\dots j_{p+1}}!}\partial_{j_1,\dots,j_{p+1}}m(\dot{\X}_i/A_n) \prod_{\ell=1}^{p+1}{X_{i,j_\ell} \over A_{n,j_\ell}},
\end{align*}
where $\dot{\bm{X}}_i = \bm{z} + \theta_i (\bm{X}_i-\bm{z})$ for some $\theta_i \in [0,1]$. Then we have
\begin{align*}
\hat{\bm{\beta}}(\bm{0}) - \bm{M}(\bm{0}) &= (\bm{X} \bm{W} \bm{X}')^{-1}\bm{X} \bm{W} (\bm{Y} - \bm{X}'\bm{M}(\bm{0}))\\
&= \left[\sum_{i=1}^{n}K_{Ah}\left(\bm{X}_i\right)
\left(
\begin{array}{c}
1 \\
\check{\bm{X}}_i 
\end{array}
\right)
(1\ \check{\bm{X}}'_i)\right]^{-1}
\sum_{i=1}^{n}K_{Ah}\left(\bm{X}_i\right)
\left(
\begin{array}{c}
1 \\
\check{\X}_i
\end{array}
\right)\\
&\quad \times \left(e_{n,i} + \varepsilon_{n,i} + \sum_{1 \leq j_1 \leq \dots \leq j_{p+1}\leq d}{1 \over \bm{s}_{j_1\dots j_{p+1}}!}\partial_{j_1,\dots,j_{p+1}}m(\dot{\bm{X}}_i/A_n)\prod_{\ell=1}^{p+1}{X_{i,j_\ell} \over A_{n,j_\ell}}\right).
\end{align*}
This yields $\sqrt{A_nh_1 \dots h_d}H(\hat{\bm{\beta}}(\bm{0}) - \bm{M}(\bm{0})) = S_n^{-1}(\bm{0})(V_n(\bm{0}) + B_n(\bm{0}))$, where
\begin{align*}
S_n(\bm{0}) &= {1 \over nh_1 \dots h_d}\sum_{i=1}^{n}K_{Ah}\left(\bm{X}_i\right)H^{-1}
\left(
\begin{array}{c}
1 \\
\check{\bm{X}}_i
\end{array}
\right)
(1\ \check{\bm{X}}'_i)H^{-1},\\
V_n(\bm{0}) &= {\sqrt{A_n h_1 \dots h_d} \over nh_1 \dots h_d}\sum_{i=1}^{n}K_{Ah}\left(\bm{X}_i\right)H^{-1}
\left(
\begin{array}{c}
1 \\
\check{\bm{X}}_i
\end{array}
\right)(e_{n,i}+\varepsilon_{n,i}) =: (V_{n,j_1\dots j_L}(\bm{0}))'_{1 \leq j_1\leq \dots \leq j_L \leq d, 0 \leq L \leq p},\\
B_n(\bm{0}) &= {\sqrt{A_n h_1 \dots h_d} \over nh_1 \dots h_d}\sum_{i=1}^{n}K_{Ah}\left(\bm{X}_i\right)H^{-1}
\left(
\begin{array}{c}
1 \\
\check{\bm{X}}_i
\end{array}
\right) \\
&\quad \times \sum_{1 \leq j_1 \leq \dots \leq j_{p+1}\leq d}{1 \over \bm{s}_{j_1\dots j_{p+1}}!}\partial_{j_1,\dots,j_{p+1}}m(\dot{\bm{X}}_i/A_n)\prod_{\ell=1}^{p+1}{X_{i,j_\ell} \over A_{n,j_\ell}} \\ 
&=: (B_{n,j_1\dots j_L}(\dot{\bm{X}}))'_{1 \leq j_1\leq \dots \leq j_L \leq d, 0 \leq L \leq p}.
\end{align*}

\noindent
(Step 1) In the supplementary material, we will show $S_{n}(\bm{0}) \stackrel{p}{\to} g(\bm{0})S$.

\noindent
(Step 2) Now we evaluate $V_n(\bm{0})$. For any $\bm{t} = (t_0,t_1,\dots, t_d, t_{11},\dots,t_{dd},\dots,t_{1\dots1},\dots, t_{d\dots d})' \in \mathbb{R}^D$, we define
\begin{align*}
\tilde{V}_n(\bm{0}) &:= {nh_1\dots h_d \over \sqrt{A_nh_1 \dots h_d}}\bm{t}'V_n(\bm{0}) = \sum_{i=1}^{n}K_{Ah}\left(\bm{X}_i\right)\left[\bm{t}'H^{-1}
\left(
\begin{array}{c}
1 \\
\check{\bm{X}}_i
\end{array}
\right)\right](e_{n,i} + \varepsilon_{n,i}).
\end{align*}

In this step, we will show that 
\begin{align}
\bm{t}'V_n(\bm{0}) &\stackrel{d}{\to} 
N\left(\bm{0}, g(\bm{0})\left\{\kappa(\eta^2(\bm{0}) + \sigma_{\varepsilon}^2(\bm{0})) + \eta^2(\bm{0})g(\bm{0})\int \sigma_{\bm{e}}(\bm{v})d\bm{v}\right\}\int K^2(\bm{z})\left[\bm{t}'
\left(
\begin{array}{c}
1 \\
\check{\bm{z}}
\end{array}
\right)\right]^2d\bm{z} \right). \label{Vn-CLT}
\end{align}

Before we show (\ref{Vn-CLT}), we introduce some notations. For $\bm{z}_0 = (z_{0,1},\dots z_{0,d})' \in \mathbb{R}^d$ and $\bm{\ell} = (\ell_1,\dots, \ell_d)' \in \mathbb{Z}^d$, let 
\[
\Gamma_{n,\bm{z}_0}(\bm{\ell};\bm{0}) = \prod_{j=1}^{d}(A_{n,j}z_{0,j} + (\ell_j-1/2)A_{n3,j}, A_{n,j}z_{0,j} + (\ell_j+1/2)A_{n3,j}]
\] 
with $A_{n3,j} = A_{n1,j} + A_{n2,j}$, and define the following hypercubes, $\Gamma_{n,\bm{z}_0}(\bm{\ell};\bm{\Delta}) = \prod_{j=1}^{d}I_{j,\bm{z}_0}(\Delta_j)$, $\bm{\Delta} = (\Delta_1,\dots, \Delta_d)' \in \{1,2\}^d$, where 
\begin{align*}
I_{j,\bm{z}_0}(\Delta_j) &= 
\begin{cases}
(A_{n,j}z_{0,j} + (\ell_j-1/2)A_{n3,j}, A_{n,j}z_{0,j} + (\ell_j-1/2)A_{n3,j} + A_{n1,j}] & \text{if $\Delta_j = 1$},\\
(A_{n,j}z_{0,j} + (\ell_j-1/2)A_{n3,j} + A_{n1,j}, A_{n,j}z_{0,j} + (\ell_j+1/2)A_{n3,j}] & \text{if $\Delta_j = 2$}. 
\end{cases}
\end{align*}
Let $\bm{\Delta}_{0} = (1,\dots, 1)'$. The partitions $\Gamma_{n,\bm{z}_0}(\bm{\ell};\bm{\Delta}_{0})$ correspond to ``large blocks'' and the partitions $\Gamma_{n,\bm{z}_0}(\bm{\ell};\bm{\Delta})$ for $\bm{\Delta} \neq \bm{\Delta}_{0}$ correspond to ``small blocks''. 
Let $L_{n1}(\bm{z}_0) = \{\bm{\ell} \in \mathbb{Z}^{d}: \Gamma_{n,\bm{z}_0}(\bm{\ell};\bm{0}) \subset R_n \cap (\bm{h} R_{n} + A_n \bm{z}_0)\}$ denote the index set of all hypercubes $\Gamma_{n,\bm{z}_0}(\bm{\ell};\bm{0})$ that are contained in $R_n \cap (\bm{h}R_{n}+A_n \bm{z}_0)$, and let $L_{n2}(\bm{z}_0) = \{\bm{\ell} \in \mathbb{Z}^{d}:  \Gamma_{n,\bm{z}_0}(\bm{\ell};\bm{0}) \cap R_n \cap (\bm{h} R_{n} + A_n \bm{z}_0) \neq 0,  \Gamma_{n}(\bm{\ell};\bm{0}) \cap (R_n \cap (\bm{h} R_{n}+A_n \bm{z}_0))^{c} \neq \emptyset \}$ be the index set of boundary hypercubes. Define $\Gamma_n(\bm{\ell};\bm{\Delta}) = \Gamma_{n,\bm{0}}(\bm{\ell};\bm{\Delta})$, $L_{n1} = L_{n1}(\bm{0})$, $L_{n2} = L_{n2}(\bm{0})$, and 
\[
\tilde{V}_n(\bm{\ell}; \bm{\Delta}) = \sum_{i: \bm{X}_i \in \Gamma_n(\bm{\ell};\bm{\Delta}) \cap \bm{h} R_n}K_{Ah}\left(\bm{X}_i\right)\left[\bm{t}'H^{-1}
\left(
\begin{array}{c}
1 \\
\check{\bm{X}}_i
\end{array}
\right)\right](e_{n,i} + \varepsilon_{n,i}).
\]
Note that by our summation convention, ${V}_n(\bm{\ell}; \bm{\Delta}) = 0$ if the set $\{i: \bm{X}_i \in \Gamma_n(\bm{\ell};\bm{\Delta}) \cap \bm{h}R_n\}$ is empty for some $\bm{\ell}$. 
Then we have 
\begin{align*}
\tilde{V}_n(\bm{0}) &= \sum_{\bm{\ell} \in L_{n1}}\tilde{V}_n(\bm{\ell};\bm{\Delta}_{0}) + \sum_{\bm{\Delta} \neq \bm{\Delta}_{0}}\sum_{\bm{\ell}\in L_{n1}}\tilde{V}_n(\bm{\ell};\bm{\Delta}) + \sum_{\bm{\Delta} \in \{1,2\}^d}\sum_{\bm{\ell} \in L_{n2}}\tilde{V}_n(\bm{\ell};\bm{\Delta}) =: \tilde{V}_{n1} + \tilde{V}_{n2} + \tilde{V}_{n3}. 
\end{align*}
Note that for $\bm{\ell}_1, \bm{\ell}_2 \in L_{n1}$, 
\begin{align}\label{Ga12-dist}
d\left(\Gamma_n(\bm{\ell}_1;\bm{\Delta}_0), \Gamma_n(\bm{\ell}_2;\bm{\Delta}_0)\right) &\geq \max\{|\bm{\ell}_1 - \bm{\ell}_2|-d,0\}\underline{A}_{n3} + \underline{A}_{n2},
\end{align}
where $\underline{A}_{n3} = \min_{1 \leq j \leq d}A_{n3,j}$ and $\underline{A}_{n2} = \min_{1 \leq j \leq d}A_{n2,j}$. 

Hence, by the Volkonskii-Rozanov inequality (cf. Proposition 2.6 in \cite{FaYa03}), we have
\begin{align}\label{Vn1-indep-approx}
\left|E[\exp (\mathrm{i}u\tilde{V}_{n1})] - \prod_{\ell \in L_{n1}}E[\exp(\mathrm{i}u\tilde{V}_n(\bm{\ell};\bm{\Delta}_0))]\right| &\lesssim \left({A_n h_1 \dots h_d \over A_n^{(1)}}\right)\alpha(\underline{A}_{n2};A_n h_1 \dots h_d). 
\end{align}

From Lyapounov's CLT, it is sufficient to verify the following conditions to show (\ref{Vn-CLT}): As $n \to \infty$, 
\begin{align}
{A_n \over n^2 h_1 \dots h_d}E[\tilde{V}_n^2(\bm{0})] &\to g(\bm{0})\! \left\{\kappa(\eta^2(\bm{0}) \!+\! \sigma_{\varepsilon}^2(\bm{0})) \!+\! \eta^2(\bm{0})g(\bm{0}) \!\!\!\int  \!\!\!\! \sigma_{\bm{e}}(\bm{v})d\bm{v}\right\}\!\! \int \!\!\!\! K^2(\bm{z})\left[\bm{t}'
\left(
\begin{array}{c}
1 \\
\check{\bm{z}}
\end{array}
\right)\right]^2\!\!\! d\bm{z}, \label{Vn-variance}\\
\sum_{\bm{\ell} \in L_{n1}}E[\tilde{V}_n^2(\bm{\ell}; \bm{\Delta}_0)] - E[\tilde{V}_n^2(\bm{0})] &= o\left(n^2A_n^{-1}h_1 \dots h_d\right), \label{Vn-var-diff}\\
\sum_{\bm{\ell} \in L_{n1}}E[\tilde{V}_n^4(\bm{\ell}; \bm{\Delta}_0)] &= o\left(\left(n^2A_n^{-1}h_1 \dots h_d\right)^2\right), \label{Vn-4}\\
\Var(\tilde{V}_{n2}) &= o\left(n^2A_n^{-1}h_1 \dots h_d\right), \label{Vn2-var}\\
\Var(\tilde{V}_{n3}) &= o\left(n^2A_n^{-1}h_1 \dots h_d\right). \label{Vn3-var}
\end{align}

In the following steps, we show (\ref{Vn-variance}) (Step 2-1), (\ref{Vn-4}) (Step 2-2), (\ref{Vn2-var}) and (\ref{Vn3-var}) (Step 2-3), and (\ref{Vn-var-diff}) (Step 2-4). 

(Step 2-1) Now we show (\ref{Vn-variance}). Let $\delta_{ij}$ be a function such that $\delta_{ij} = 1$ if $i =j$ and $\delta_{ij} = 0$ if $i \neq j$. Observe that 
\begin{align*}
\sigma_n^2(\bm{0}) &:= E_{\cdot \mid \bm{X}}\left(\tilde{V}_n^2(\bm{0})\right) = \sum_{i,j=1}^{n}\bm{t}'H^{-1}
\left(
\begin{array}{c}
1 \\
\check{\bm{X}_i}
\end{array}
\right)\bm{t}'H^{-1}
\left(
\begin{array}{c}
1 \\
\check{\bm{X}}_j
\end{array}
\right)K_{Ah}\left(\bm{X}_i\right)K_{Ah}\left(\bm{X}_j\right)\\
&\quad \quad \quad \quad \quad \quad \quad \quad \times \left\{\eta(\bm{X}_i/A_n)\eta(\bm{X}_j/A_n)\sigma_{\bm{e}}(\bm{X}_i-\bm{X}_j) + \sigma_{\varepsilon}^2(\bm{X}_i/A_n)\delta_{ij}\right\}.
\end{align*}
Thus we have
\begin{align*}
E_{\bm{X}}\left[\sigma_n^2(\bm{0})\right] &= nA_n^{-1}\int \left[\bm{t}'H^{-1}
\left(
\begin{array}{c}
1 \\
\check{(\bm{x}/A_n)}
\end{array}
\right)\right]^2K^2_{Ah}(\bm{x})\left\{\eta^2(\bm{x}/A_n) + \sigma_{\varepsilon}^2(\bm{x}/A_n)\right\}g(\bm{x}/A_n)d\bm{x}\\
&\quad \quad + n(n-1)A_n^{-2}\int \bm{t}'H^{-1}
\left(
\begin{array}{c}
1 \\
\check{(\bm{x}_1/A_n)}
\end{array}
\right)\bm{t}'H^{-1}
\left(
\begin{array}{c}
1 \\
\check{(\bm{x}_2/A_n)}
\end{array}
\right)K_{Ah}(\bm{x}_1)K_{Ah}(\bm{x}_2)\\
&\quad \quad \quad  \times \eta(\bm{x}_1/A_n)\eta(\bm{x}_2/A_n)\sigma_{\bm{e}}(\bm{x}_1 - \bm{x}_2)g(\bm{x}_1/A_n)g(\bm{x}_2/A_n)d\bm{x}_1d\bm{x}_2\\
&\quad =: \sigma^2_{n,1} + \sigma^2_{n,2}. 
\end{align*}
For $\sigma^2_{n,1}$, we have 
\begin{align}\label{sig2_n1}
\sigma^2_{n,1} &= nh_1\dots h_d \int \left[\bm{t}'
\left(
\begin{array}{c}
1 \\
\check{\bm{z}}
\end{array}
\right)\right]^2K^2(\bm{z})\left\{\eta^2(\bm{z} \circ \bm{h}) + \sigma^2_{\varepsilon}(\bm{z} \circ \bm{h})\right\}g(\bm{z} \circ \bm{h})d\bm{z} \nonumber \\
&= nh_1 \dots h_d (\eta^2(\bm{0}) + \sigma^2_{\varepsilon}(\bm{0}))g(\bm{0})\left(\int  K^2(\bm{z})\left[\bm{t}'
\left(
\begin{array}{c}
1 \\
\check{\bm{z}}
\end{array}
\right)\right]^2 d\bm{z}\right)(1+o(1)).
\end{align}
For $\sigma^2_{n,2}$, we have
\begin{align*}
\sigma^2_{n,2} &= n(n-1)\int_{R_0^2} \sigma_{\bm{e}}(A_n(\bm{y}_1 - \bm{y}_2)) \left[\bm{t}'H^{-1}
\left(
\begin{array}{c}
1 \\
\check{\bm{y}}_1
\end{array}
\right)\right] \left[\bm{t}'H^{-1}
\left(
\begin{array}{c}
1 \\
\check{\bm{y}}_2
\end{array}
\right)\right] \\
&  \times K_{h}(\bm{y}_1)K_{h}(\bm{y}_2)\eta(\bm{y}_1)\eta(\bm{y}_2)g(\bm{y}_1)g(\bm{y}_2)d\bm{y}_1d\bm{y}_2\\
&= n(n-1)(h_1 \dots h_d)^2\int_{\bm{h}^{-1} R_0^2} \sigma_{\bm{e}}(A_n(\bm{z}_1 - \bm{z}_2) \circ \bm{h}) \left[\bm{t}'
\left(
\begin{array}{c}
1 \\
\check{\bm{z}}_1
\end{array}
\right)\right] \left[\bm{t}'
\left(
\begin{array}{c}
1 \\
\check{\bm{z}}_2
\end{array}
\right)\right] \\
&  \times K(\bm{z}_1)K(\bm{z}_2)\eta(\bm{z}_1 \circ \bm{h})\eta(\bm{z}_2 \circ \bm{h})g(\bm{z}_1 \circ \bm{h})g(\bm{z}_2 \circ \bm{h})d\bm{z}_1d\bm{z}_2\\
& = n(n-1)(h_1 \dots h_d)^2\int_{R'_{\bm{h},0}} \sigma_{\bm{e}}(A_n\bm{w} \circ \bm{h}) \left(\int_{R_{\bm{h},0}(\bm{w})} \left[\bm{t}'
\left(
\begin{array}{c}
1 \\
\check{(\bm{z}_2 + \bm{w})}
\end{array}
\right)\right] \left[\bm{t}'
\left(
\begin{array}{c}
1 \\
\check{\bm{z}}_2
\end{array}
\right)\right] \right. \\
&\left.    \times K(\bm{z}_2 + \bm{w})K(\bm{z}_2)\eta((\bm{z}_2 + \bm{w})\circ \bm{h})\eta(\bm{z}_2 \circ \bm{h})g((\bm{z}_2 + \bm{w})\circ \bm{h})g(\bm{z}_2 \circ \bm{h})d\bm{z}_2\right)d\bm{w}\\
& = n(n-1)h_1 \dots h_d\int_{\bm{h} R'_{\bm{h}, 0}} \sigma_{\bm{e}}(A_n\bm{u}) \left(\int_{R_{\bm{h}, 0}(\bm{u} /\bm{h})} \left[\bm{t}'
\left(
\begin{array}{c}
1 \\
\check{(\bm{z}_2 + \bm{u} \circ \bm{h}^{-1})}
\end{array}
\right)\right] \left[\bm{t}'
\left(
\begin{array}{c}
1 \\
\check{\bm{z}}_2
\end{array}
\right)\right]  \right. \\
&\left.  \times K(\bm{z}_2 + \bm{u} \circ \bm{h}^{-1})K(\bm{z}_2)\eta(\bm{z}_2 \circ \bm{h} + \bm{u})\eta(\bm{z}_2 \circ \bm{h})g((\bm{z}_2 \circ \bm{h} + \bm{u})g(\bm{z}_2 \circ \bm{h})d\bm{z}_2\right)d\bm{u}\\
& = {n(n-1) \over A_n}h_1 \dots h_d\int_{A_n \bm{h}R'_{\bm{h},0}}\!\!\!\!\!\!\!\!\!\!\! \sigma_{\bm{e}}(\bm{v}) \left(\int_{R_{\bm{h},0}((\bm{v} \circ \bm{h}^{-1})/A_n)} \!\left[\bm{t}'
\left(
\begin{array}{c}
1 \\
\check{\left(\bm{z}_2 + {\bm{v} \circ \bm{h}^{-1}\over A_n}\right)}
\end{array}
\right)\right]\!\! \left[\bm{t}'
\left(
\begin{array}{c}
1 \\
\check{\bm{z}}_2
\end{array}
\right)\right]  \right. \\
&\left.   \times K\!\! \left(\!\bm{z}_2  \!+\! {\bm{v} \circ \bm{h}^{-1}\over A_n}\!\right)\! K(\bm{z}_2)\eta\! \left(\!\bm{z}_2 \circ \bm{h} \!+\! {\bm{v} \over A_n}\!\right)\!\eta(\bm{z}_2 \circ \bm{h})g\! \left(\! \bm{z}_2 \circ \bm{h} \!+\! {\bm{v} \over A_n}\!\right)g(\bm{z}_2 \circ \bm{h})d\bm{z}_2 \!\right)d\bm{v}
\end{align*}
where 
\begin{align*}
R_{\bm{h},0}' &\!= \{\bm{w} = \bm{z}_1 \!-\! \bm{z}_2: \bm{z}_1, \bm{z}_2 \in \bm{h}^{-1}\!R_0\},\ R_{\bm{h},0}(\bm{w}) \!= \{\bm{z}_2 : \bm{z}_2 \in \bm{h}^{-1}\!R_0 \! \cap \! (\bm{h}^{-1}R_0 \!+\! \bm{w})\},\\
A_n \bm{h}R_{\bm{h},0}' &=\{(A_{n,1}x_1, \dots, A_{n,d}x_d): \bm{x} = (x_1,\dots,x_d)' \in \bm{h}R_{\bm{h},0}'\}.
\end{align*}
We divide the integral $\int_{A_n \bm{h} R_{\bm{h},0}'} $ into two parts $\int_{A_n \bm{h} R_{\bm{h},0}' \cap \{|\bm{v}|\leq M\}}$ and $\int_{A_n \bm{h} R_{\bm{h},0}' \cap \{|\bm{v}| > M\}}$ for some $M>0$ and define these as $\sigma_{n,21}^2$ and $\sigma_{n,22}^2$, respectively. Observe that as $n \to \infty$, $|\sigma_{n,22}^2| \lesssim \int_{\{|\bm{v}|>M\}} |\sigma_{\bm{e}}(\bm{v})| d\bm{v}$ which can be made arbitrary small by choosing a large $M$. Further, observe that as $n \to \infty$
\begin{align*}
&1\{A_n \bm{h} R_{\bm{h},0}' \cap \{|\bm{v}| \leq M\}\}\int_{R_{\bm{h},0}(\bm{v}/(A_n\bm{h}))} \left[\bm{t}'
\left(
\begin{array}{c}
1 \\
\check{\left(\bm{z}_2 + {\bm{v} \circ \bm{h}^{-1} \over A_n}\right)}
\end{array}
\right)\right] \left[\bm{t}'
\left(
\begin{array}{c}
1 \\
\check{\bm{z}}_2
\end{array}
\right)\right]  \\
&\quad \quad  \times K\! \left(\!\bm{z}_2 \!+\! {\bm{v} \circ \bm{h}^{-1} \over A_n}\right)\!\!K(\bm{z}_2)\eta\left(\!\bm{z}_2 \circ \bm{h} \!+\! {\bm{v} \over A_n}\right)\!\eta(\bm{z}_2 \circ \bm{h})g\!\left(\!\bm{z}_2 \circ \bm{h} \!+\! {\bm{v}   \over A_n}\right)g(\bm{z}_2 \circ \bm{h})d\bm{z}_2\\
&= 1\{|\bm{v}|\leq M\}\eta^2(\bm{0})g^2(\bm{0})\left(\int  K^2(\bm{z}_2)\left[\bm{t}'
\left(
\begin{array}{c}
1 \\
\check{\bm{z}}_2
\end{array}
\right)\right]^2d\bm{z}_2\right)(1+o(1)).
\end{align*}
Then as $n \to \infty$, we have
\begin{align*}
\sigma_{n,21}^2 &= \eta^2(\bm{0})g^2(\bm{0})\left(\int_{\{|\bm{v}| \leq M\}}\sigma_{\bm{e}}(\bm{v})d\bm{v}\right)\left(\int  K^2(\bm{z}_2)\left[\bm{t}'
\left(
\begin{array}{c}
1 \\
\check{\bm{z}}_2
\end{array}
\right)\right]^2d\bm{z}_2\right)(1+o(1)).
\end{align*}
Therefore, we have
\begin{align}\label{sig2_n2}
\sigma^2_{n,2}& = n^2A_n^{-1}h_1 \dots h_d\eta^2(\bm{0})g^2(\bm{0})\left(\int \sigma_{\bm{e}}(\bm{v})d\bm{v}\right)\left(\int K^2(\bm{z})\left[\bm{t}'\left(
\begin{array}{c}
1 \\
\check{\bm{z}}
\end{array}
\right)\right]^2d\bm{z}\right)(1 + o(1)). 
\end{align}
By (\ref{sig2_n1}) and (\ref{sig2_n2}), we have
\begin{align*}
\Var(\bm{t}'V_n(\bm{0})) &= g(\bm{0})\left\{\kappa( \eta^2(\bm{0}) + \sigma_\varepsilon^2(\bm{0}))+ \eta^2(\bm{0})g(\bm{0})\int \sigma_{\bm{e}}(\bm{v})d\bm{v}\right\}\!\left(\!\int \!\! K^2(\bm{z})\!\left[\bm{t}'\left(
\begin{array}{c}
1 \\
\check{\bm{z}}
\end{array}
\right)\right]^2\!\! d\bm{z}\right)(1 \!+\! o(1)). 
\end{align*}

(Step 2-2) Now we show (\ref{Vn-4}). Define $I_{n}(\bm{\ell}) \!=\! \{\bm{i} \! \in \! \mathbb{Z}^d \!: \bm{i} \!+\! (-1/2,1/2]^d \! \subset \! \Gamma_n(\bm{\ell}; \bm{\Delta}_0)\}$ for $\bm{\ell} \in L_{n1}$ and 
\[
\tilde{V}_n(\bm{i}) = \sum_{i=1}^{n}K_{Ah}\left(\bm{X}_i\right)\left[\bm{t}'H^{-1}
\left(
\begin{array}{c}
1 \\
\check{\bm{X}}_i
\end{array}
\right)\right](e_{n,i} + \varepsilon_{n,i})1\{\bm{X}_i \in [\bm{i} + (-1/2,1/2]^d] \cap R_n\}. 
\]

Observe that 
\begin{align*}
&E[\tilde{V}_n^4(\bm{\ell}; \bm{\Delta}_0)] = E\left[\left(\sum_{\bm{i} \in I_n(\bm{\ell})}\tilde{V}_n(\bm{i})\right)^4\right]\\ 
&= \sum_{\bm{i} \in I_n(\bm{\ell})}E\left[\tilde{V}_n^4(\bm{i})\right] + \sum_{\bm{i}, \bm{j} \in I_n(\bm{\ell}), \bm{i} \neq \bm{j}}E\left[\tilde{V}_n^3(\bm{i})\tilde{V}_n(\bm{j})\right] + \sum_{\bm{i}, \bm{j} \in I_n(\bm{\ell}), \bm{i} \neq \bm{j}}E\left[\tilde{V}_n^2(\bm{i})\tilde{V}_n^2(\bm{j})\right]\\
&\quad + \sum_{\bm{i}, \bm{j}, \bm{k} \in I_n(\bm{\ell}), \bm{i} \neq \bm{j} \neq \bm{k}}\!\!\!\!\!\!\!\!\!\!\!\!E \! \left[\tilde{V}_n^2(\bm{i})\tilde{V}_n(\bm{j})\tilde{V}_n(\bm{k})\right] + \sum_{\bm{i}, \bm{j}, \bm{k}, \bm{p} \in I_n(\bm{\ell}), \bm{i} \neq \bm{j} \neq \bm{k} \neq \bm{p}}\!\!\!\!\!\!\!\!\!\!\!\!E \! \left[\tilde{V}_n(\bm{i})\tilde{V}_n(\bm{j})\tilde{V}_n(\bm{k})\tilde{V}_n(\bm{p})\right]\\
& =: Q_{n1} + Q_{n2} + Q_{n3} + Q_{n4} + Q_{n5}. 
\end{align*}

For $Q_{n1}$, we have
\begin{align*}
&E[\tilde{V}_n^4(\bm{i})] = E_{\bm{X}}[E_{\cdot \mid \bm{X}}[\tilde{V}_n^4(\bm{i})]]\\
&= \sum_{j_1,j_2,j_3,j_4 =1}^{n}E\left[\prod_{k=1}^{4}K_{Ah}\left(\bm{X}_{j_k}\right)\left[\bm{t}'H^{-1}
\left(
\begin{array}{c}
1 \\
\check{\bm{X}}_{j_k}
\end{array}
\right)\right]1\{\bm{X}_{j_k} \in [\bm{i} + (-1/2,1/2]^d] \cap R_n\} \right. \\
&\left. \quad \quad \quad \quad \times E_{\cdot \mid \bm{X}}[e_{n,j_k} + \varepsilon_{n,j_k}]\right]\\
&\lesssim \sum_{j_1,j_2,j_3,j_4 =1}^{n}\!\!\!\!\!\!\!E\left[\prod_{k=1}^{4}\left|K_{Ah}\! \left(\bm{X}_{j_k}\right)\!\!\left[\! \bm{t}'H^{-1}
\!\!\left(\!
\begin{array}{c}
1 \\
\check{\bm{X}}_{j_k}
\end{array}
\!\right)\!\right]\right|\! 1\{\bm{X}_{j_k}\!\! \in \! [\bm{i} \!+\! (-1/2,1/2]^d] \cap R_n\}\eta\!\left(\!{\bm{X}_{j_k} \over A_n}\!\right)\!\right]\\
& + \sum_{j_1,j_2,j_3,j_4 =1}^{n}\!\!\!\!\!\!\!\! E\left[\prod_{k=1}^{4}\left|K_{Ah}\left(\bm{X}_{j_k}\right)\!\! \left[\! \bm{t}'H^{-1}
\!\!\left(\!
\begin{array}{c}
1 \\
\check{\bm{X}}_{j_k}
\end{array}
\!\right)\!\right]\right|\! 1\{\bm{X}_{j_k} \!\! \in \!  [\bm{i} \!+\! (-1/2,1/2]^d] \!\cap\! R_n\}\sigma_{\varepsilon}\!\left(\!{\bm{X}_{j_k} \over A_n}\!\right)\!\right]\\
&=: Q_{n11} + Q_{n12}. 
\end{align*}

For $Q_{n11}$, we have
\begin{align*}
Q_{n11} &\lesssim nE\left[\left|K_{Ah}\left(\bm{X}_{1}\right)\left[\bm{t}'H^{-1}
\left(
\begin{array}{c}
1 \\
\check{\bm{X}}_{1}
\end{array}
\right)\right]\right|^41\{\bm{X}_{1} \in [\bm{i} + (-1/2,1/2]^d] \cap R_n\}\eta^4(\bm{X}_{1}/A_n)\right] \\
& + n^2E\left[\left|K_{Ah}\left(\bm{X}_{1}\right)\left[\bm{t}'H^{-1}
\left(
\begin{array}{c}
1 \\
\check{\bm{X}}_{1}
\end{array}
\right)\right]\right|^31\{\bm{X}_{1} \in [\bm{i} + (-1/2,1/2]^d] \cap R_n\} \right. \\
&\left.  \times \left|K_{Ah}\left(\bm{X}_{2}\right) \! \left[\bm{t}'H^{-1} \!
\left( \!
\begin{array}{c}
1 \\
\check{\bm{X}}_{2}
\end{array}
\! \right)\right]\right|1\{\bm{X}_{2} \in [\bm{i} + (-1/2,1/2]^d] \cap R_n\} \eta^3(\bm{X}_{1}/A_n)\eta(\bm{X}_{2}/A_n)\right] \\
& + n^2E\left[\prod_{\ell=1}^{2}\left|K_{Ah}\left(\bm{X}_{\ell}\right) \! \left[\bm{t}'H^{-1} \!
\left(
\begin{array}{c}
1 \\
\check{\bm{X}}_{\ell}
\end{array}
\right)\right]\right|^2 \!\! 1\{\bm{X}_{\ell} \in [\bm{i} + (-1/2,1/2]^d] \cap R_n\} \eta^2(\bm{X}_{\ell}/A_n)\right]\\
& + n^3E\left[\left|K_{Ah}\left(\bm{X}_{1}\right)\left[\bm{t}'H^{-1}
\left(
\begin{array}{c}
1 \\
\check{\bm{X}}_{1}
\end{array}
\right)\right]\right|^21\{\bm{X}_{1} \in [\bm{i} + (-1/2,1/2]^d] \cap R_n\}\eta^2(\bm{X}_1/A_n) \right. \\
&\left.  \times \prod_{\ell = 2}^{3}\left|K_{Ah}\left(\bm{X}_{\ell}\right)\left[\bm{t}'H^{-1}
\left(
\begin{array}{c}
1 \\
\check{\bm{X}}_{\ell}
\end{array}
\right)\right]\right|1\{\bm{X}_{\ell} \in [\bm{i} + (-1/2,1/2]^d] \cap R_n\}\eta(\bm{X}_{\ell}/A_n)\right]\\
& + n^4E\left[\prod_{\ell=1}^{4}\left|K_{Ah}\left(\bm{X}_{\ell}\right)\left[\bm{t}'H^{-1}
\left(
\begin{array}{c}
1 \\
\check{\bm{X}}_{\ell}
\end{array}
\right)\right]\right|1\{\bm{X}_{\ell} \in [\bm{i} + (-1/2,1/2]^d] \cap R_n\} \eta(\bm{X}_{\ell}/A_n)\right]\\
&=: Q_{n111} + Q_{n112} + Q_{n113} + Q_{n114}. 
\end{align*}

For $Q_{n111}$, we have
\begin{align*}
Q_{n111} &= nA_n^{-1}\int \left|K_{Ah}(\bm{x})\left[\bm{t}'H^{-1}
\left(
\begin{array}{c}
1 \\
\check{(\bm{x}/A_n)}
\end{array}
\right)\right]\right|^41\{\bm{x} \in [\bm{i} + (-1/2,1/2]^d] \cap R_n\}\eta^4(\bm{x}/A_n)g(\bm{x}/A_n)d\bm{x}\\
&= nA_n^{-1}\!A_nh_1 \dots h_d \!\! \int \! \left|K(\bm{z}) \!\left[\!\bm{t}' \!
\left(\!
\begin{array}{c}
1 \\
\check{\bm{z}}
\end{array}
\!\right)\! \right]\right|^4 \!\!\! 1\{\bm{z} \! \circ \! \bm{h} \!\in \! [\bm{i} \!+\! (-1/2,1/2]^d]/A_n \! \cap \! [-1/2,1/2]^d\} \\
&\quad \times \eta^4(\bm{z} \! \circ \! \bm{h})g(\bm{z} \!\circ \! \bm{h})d\bm{z} = O\left(nA_n^{-1}\right). 
\end{align*}
Likewise, $Q_{n112} = O(n^2A_n^{-2})$, $Q_{n113} = O(n^3A_n^{-3})$, and $Q_{n114} = O(n^4A_n^{-4})$.
Then we have $Q_{n11} = O(n^4A_n^{-4})$. We can also show that $Q_{n12} = O(n^4A_n^{-4} )$.
Therefore, we have
\begin{align}\label{Qn1}
Q_{n1} &\lesssim  [\![I_n(\bm{\ell})]\!]n^4A_n^{-4} \lesssim A_n^{(1)}(nA_n^{-1})^4.
\end{align}

For $Q_{n2}$,  by the $\alpha$-mixing property of $\bm{e}$ and Proposition 2.5 in \cite{FaYa03}, we have
\begin{align}
Q_{n2} &\lesssim \sum_{k=1}^{\overline{A}_{n1}} \sum_{\bm{i}, \bm{j} \in I_n(\bm{\ell}), |\bm{i} - \bm{j}| = k}\alpha^{1-4/q}(\max\{k-d,0\};1)E[|\tilde{V}_n(\bm{i})|^q]^{3/q}E[|\tilde{V}_n(\bm{j})|^q]^{1/q} \nonumber \\
&\lesssim A_n^{(1)}(nA_n^{-1})^4\left(1 + \sum_{k=1}^{\overline{A}_{n1}}k^{d-1}\alpha_1^{1-4/q}(k)\right).  \label{Qn2}
\end{align}
where $\overline{A}_{n1} = \max_{1 \leq j \leq d}A_{n1,j}$. Likewise, 
\begin{align}
Q_{n3} &\lesssim A_n^{(1)}(nA_n^{-1})^4\left(1 + \sum_{k=1}^{\overline{A}_{n1}}k^{d-1}\alpha_1^{1-4/q}(k)\right).  \label{Qn3}
\end{align}

Now we evaluate $Q_{n4}$ and $Q_{n5}$. For distinct indices $\bm{i}, \bm{j}, \bm{k}, \bm{p} \in I_n(\bm{\ell})$, let 
\begin{align*}
d_1(\bm{i}, \bm{j}, \bm{k}) &= \max\{d(\{\bm{i}\}, \{\bm{j}, \bm{k}\}),d(\{\bm{k}\}, \{\bm{i}, \bm{j}\})\},\\
d_2(\bm{i}, \bm{j}, \bm{k}, \bm{p}) &= \max\{d(J, \{\bm{i}, \bm{j}, \bm{k}, \bm{p}\}): J \subset \{\bm{i}, \bm{j}, \bm{k}, \bm{p}\}, [\![ J]\!] = 1\},\\
d_3(\bm{i}, \bm{j}, \bm{k}, \bm{p}) &= \max\{d(J, \{\bm{i}, \bm{j}, \bm{k}, \bm{p}\}): J \subset \{\bm{i}, \bm{j}, \bm{k}, \bm{p}\}, [\![ J]\!] = 2\}. 
\end{align*}
Here, $d_1$ denotes the maximal gap in the set of integer-indices $\{\bm{i},\bm{j},\bm{k}\}$ from either $\bm{j}$ or $\bm{k}$ which corresponds to $E\left[\tilde{V}_n^2(\bm{i})\tilde{V}_n(\bm{j})\tilde{V}_n(\bm{k})\right]$. Similarly, $d_2$ and $d_3$ are the maximal gap in the index set $\{\bm{i}, \bm{j}, \bm{k}, \bm{p}\}$ from any of its single index-subsets or two-index subsets, respectively. Applying the argument in the proof of Lemma 4.1 of \cite{La99}, for any given values $1 \leq d_{01}, d_{02}, d_{03} < [\![I_{n}(\bm{\ell})]\!]$, we have 
\begin{align}
&[\![\{(\bm{i}, \bm{j},\bm{k}) \in I_n^3(\bm{\ell}): \bm{i} \neq \bm{j} \neq \bm{k}\ \text{and}\ d_1(\bm{i},\bm{j},\bm{k})=d_{01}\}]\!] \lesssim d_{01}^{2d-1}[\![I_n(\bm{\ell})]\!],  \label{card-d01}\\
&[\![\{(\bm{i}, \bm{j},\bm{k}, \bm{p}) \in I_n^4(\bm{\ell}): \bm{i} \neq \bm{j} \neq \bm{k} \neq \bm{p},\ d_2(\bm{i},\bm{j},\bm{k}, \bm{p})=d_{02},\ \text{and}\ d_3(\bm{i},\bm{j},\bm{k}, \bm{p})=d_{03}\}]\!] \nonumber \\ 
&\quad \lesssim (d_{02} + d_{03})^{3d-1}[\![I_n(\bm{\ell})]\!].  \label{card-d023}
\end{align}

For $Q_{n4}$, by (\ref{card-d01}) and applying the same argument to show (\ref{Qn2}), we have
\begin{align}
Q_{n4} &\lesssim A_n^{(1)}\sum_{k=1}^{\overline{A}_{n1}} k^{2d-1}\alpha^{1-4/q}(\max\{k-d,0\};2)E[|\tilde{V}_n(\bm{i})|^q]^{2/q}E[|\tilde{V}_n(\bm{j})|^q]^{1/q}E[|\tilde{V}_n(\bm{k})|^q]^{1/q} \nonumber \\
&\lesssim A_n^{(1)}(nA_n^{-1})^4\left(1 + \sum_{k=1}^{\overline{A}_{n1}}k^{2d-1}\alpha_1^{1-4/q}(k)\right).  \label{Qn4}
\end{align}

Define
\begin{align*}
I_{n1}(\bm{\ell}) &= \{(\bm{i}, \bm{j},\bm{k}, \bm{p}) \in I_n^4(\bm{\ell}): \bm{i} \neq \bm{j} \neq \bm{k} \neq \bm{p},\ d_2(\bm{i},\bm{j},\bm{k}, \bm{p}) \geq d_3(\bm{i},\bm{j},\bm{k}, \bm{p})\},\\
I_{n2}(\bm{\ell}) &= \{(\bm{i}, \bm{j},\bm{k}, \bm{p}) \in I_n^4(\bm{\ell}): \bm{i} \neq \bm{j} \neq \bm{k} \neq \bm{p},\ d_2(\bm{i},\bm{j},\bm{k}, \bm{p}) < d_3(\bm{i},\bm{j},\bm{k}, \bm{p})\}.
\end{align*}

For $Q_{n5}$, by (\ref{card-d023}) and applying the same argument to show (\ref{Qn2}), we have
\begin{align}
Q_{n5} &= \sum_{(\bm{i}, \bm{j},\bm{k}, \bm{p}) \in I_{n1}(\bm{\ell})}\!\!\!\!\!\!E\left[\tilde{V}_n(\bm{i})\tilde{V}_n(\bm{j})\tilde{V}_n(\bm{k})\tilde{V}_n(\bm{p})\right] + \sum_{(\bm{i}, \bm{j},\bm{k}, \bm{p}) \in I_{n2}(\bm{\ell})}\!\!\!\!\!\!E\left[\tilde{V}_n(\bm{i})\tilde{V}_n(\bm{j})\tilde{V}_n(\bm{k})\tilde{V}_n(\bm{p})\right] \nonumber \\
&\lesssim A_n^{(1)}\sum_{k=1}^{\overline{A}_{n1}} k^{3d-1}\alpha^{1-4/q}(\max\{k-d,0\};3) \nonumber \\
&\quad \times E[|\tilde{V}_n(\bm{i})|^q]^{1/q}E[|\tilde{V}_n(\bm{j})|^q]^{1/q}E[|\tilde{V}_n(\bm{k})|^q]^{1/q}E[|\tilde{V}_n(\bm{p})|^q]^{1/q} \nonumber \\
&\quad + \left(\sum_{\bm{i}, \bm{j} \in I_n(\bm{\ell}), \bm{i} \neq \bm{j}}\left|E[\tilde{V}_n(\bm{i})\tilde{V}_n(\bm{j})]\right|\right)^2 + A_n^{(1)}\sum_{k=1}^{\overline{A}_{n1}} k^{3d-1}\alpha^{1-4/q}(\max\{k-d,0\};2) \nonumber \\
&\quad \times E[|\tilde{V}_n(\bm{i})|^q]^{1/q}E[|\tilde{V}_n(\bm{j})|^q]^{1/q}E[|\tilde{V}_n(\bm{k})|^q]^{1/q}E[|\tilde{V}_n(\bm{p})|^q]^{1/q} \nonumber \\
&\lesssim (A_n^{(1)})^2(nA_n^{-1})^4\left(1 + \sum_{k=1}^{\overline{A}_{n1}}k^{2d-1}\alpha_1^{1-4/q}(k)\right).  \label{Qn5}
\end{align}
Combining (\ref{Qn1}),  (\ref{Qn2}),  (\ref{Qn3}),  (\ref{Qn4}), and (\ref{Qn5}), we have
\begin{align*}
&\sum_{\bm{\ell} \in L_{n1}}E[\tilde{V}_n^4(\bm{\ell}; \bm{\Delta}_0)]\\ 
&= \sum_{\bm{\ell} \in L_{n1}}E\left[\left(\sum_{\bm{i} \in I_n(\bm{\ell})}\tilde{V}_n(\bm{i})\right)^4\right] \lesssim [\![L_{n1}]\!](A_n^{(1)})^2(nA_n^{-1})^4\left(1 + \sum_{k=1}^{\overline{A}_{n1}}k^{2d-1}\alpha_1^{1-4/q}(k)\right)\\
&\lesssim \left({A_n h_1 \dots h_d \over A_n^{(1)}}\right)(A_n^{(1)})^2(nA_n^{-1})^4\left(1 + \sum_{k=1}^{\overline{A}_{n1}}k^{2d-1}\alpha_1^{1-4/q}(k)\right) = o\left((n^2A_n^{-1}h_1 \dots h_d)^2\right). 
\end{align*}

(Step 2-3) Now we show (\ref{Vn2-var}) and (\ref{Vn3-var}). Define 
\begin{align*}
J_n &= \{\bm{i} \in \mathbb{Z}^d: (\bm{i} + (-1/2,1/2]^d) \cap \bm{h} R_n \neq \emptyset\},\ J_{n1} = \cup_{\bm{\ell} \in L_{n1}}I_{n}(\bm{\ell}),\\
J_{n2} &= \{\bm{i} \in J_n: \bm{i} + (-1/2,1/2]^d \subset \Gamma_n(\bm{\ell};\bm{\Delta})\ \text{for some}\ \bm{\ell} \in L_{n1}, \bm{\Delta} \neq \bm{\Delta}_0\},\ J_{n3} = J_n \backslash (J_{n1} \cup J_{n2}). 
\end{align*}
Note that $[\![J_{n2}]\!] \lesssim (\overline{A}_{n1})^{d-1}\overline{A}_{n2}\left({A_nh_1 \dots h_d \over A_n^{(1)}}\right)$ and $[\![J_{n3}]\!] \lesssim A_n^{(1)}\left({  \overline{A_nh} \over \underline{A}_{n1}}\right)^{d-1}$. Then, applying the same argument to show (\ref{Qn2}), we have
\begin{align*}
\Var(\tilde{V}_{n2}) &\lesssim [\![J_{n2}]\!](nA_n^{-1})^2\left(1 + \sum_{k=1}^{\overline{A}_{n1}}k^{d-1}\alpha_1^{1-2/q}(k)\right)\\
&\lesssim \left({\overline{A}_{n1} \over \underline{A}_{n1}}\right)^d\left({\overline{A}_{n2} \over \overline{A}_{n1}}\right)A_n h_1 \dots h_d (nA_n^{-1})^2\left(1 + \sum_{k=1}^{\overline{A}_{n1}}k^{d-1}\alpha_1^{1-2/q}(k)\right) = o\left( n^2A_n^{-1}h_1 \dots h_d\right). \\
\Var(\tilde{V}_{n3}) &\lesssim [\![J_{n3}]\!](nA_n^{-1})^2\left(1 + \sum_{k=1}^{\overline{A}_{n1}}k^{d-1}\alpha_1^{1-2/q}(k)\right)\\
&\lesssim\!  \left(\!{A_n^{(1)} \over \underline{A}_{n1}^d} \! \right)\!\! \left(\! {\left(\overline{A_n h}\right)^d \over A_n h_1 \dots h_d} \! \right)\!\!\left(\! \underline{A}_{n1} \over \overline{A_n h} \! \right) A_n h_1 \dots h_d (nA_n^{-1})^2 \!\left(\!\! 1 \!+\! \sum_{k=1}^{\overline{A}_{n1}}k^{d-1}\alpha_1^{1-2/q}(k) \!\! \right)\\
&= o\left( n^2A_n^{-1}h_1 \dots h_d\right).  
\end{align*}

(Step 2-4) Now we show (\ref{Vn-var-diff}). By (\ref{Vn2-var}) and (\ref{Vn3-var}), we have for sufficiently large $n$, 
\begin{align*}
E[\tilde{V}_{n1}^2] &= E[(\tilde{V}_n(\bm{0}) - (\tilde{V}_{n2} + \tilde{V}_{n3}))^2] \leq 2\left(E[(\tilde{V}_n(\bm{0}))^2] + E[(\tilde{V}_{n2} + \tilde{V}_{n3})^2]\right) \leq 4E[\tilde{V}_n^2(\bm{0})]. 
\end{align*}
Thus, by (\ref{Ga12-dist}), (\ref{Vn2-var}), and (\ref{Vn3-var}), we have
\begin{align*}
&\left|\sum_{\bm{\ell} \in L_{n1}}E[\tilde{V}_n^2(\bm{\ell}; \bm{\Delta}_0)] - E[\tilde{V}_n^2(\bm{0})]\right|\\
&\leq  \left|\sum_{\bm{\ell} \in L_{n1}}E[\tilde{V}_n^2(\bm{\ell}; \bm{\Delta}_0)] - E[\tilde{V}_{n1}^2]\right| + 2E[(\tilde{V}_{n2} + \tilde{V}_{n3})^2]^{1/2}E[\tilde{V}_{n1}^2]^{1/2} + E[(\tilde{V}_{n2} + \tilde{V}_{n3})^2]\\
&\lesssim \left(A_n^{(1)}nA_n^{-1}\right)^2 \!\!\sum_{\bm{\ell}_1 \neq \bm{\ell}_2} \!\! \alpha^{1-2/q}(\max\{|\bm{\ell}_1 \!-\! \bm{\ell}_2|-d, 0\}\underline{A}_{n3} \!+\! \underline{A}_{n2};A_n^{(1)}) \!+\! o\left(n^2A_n^{-1}h_1 \dots h_d\right)\\
&\lesssim \left(A_n^{(1)}nA_n^{-1}\right)^2\left({A_nh_1 \dots h_d \over A_n^{(1)}}\right)  \\ 
&\quad \times \left(\alpha^{1-2/q}(\underline{A}_{n2};A_n^{(1)}) + \sum_{k=1}^{\overline{A}_{n}/\underline{A}_{n1}}k^{d-1}\alpha^{1-2/q}(\max\{k-d, 0\}\underline{A}_{n3} + \underline{A}_{n2};A_n^{(1)})\right)\\ 
&\quad + o\left(n^2A_n^{-1}h_1 \dots h_d\right) = o\left(n^2A_n^{-1}h_1 \dots h_d\right),
\end{align*}
where $\overline{A}_n = \max_{1 \leq j \leq d}A_{n,j}$.

\noindent
(Step 3) In the supplementary material, we will show
\begin{align*}
B_{n,j_1\dots j_L}(\dot{\bm{X}}) &= g(\bm{0})\sqrt{A_nh_1 \dots h_d}(B^{(d,p)}\bm{M}_n^{(d,p)}(\bm{0}))_{j_1\dots j_L} + o_p(1). 
\end{align*}

\noindent
(Step 4) Combining the results in Steps 2 and 3, we have
\begin{align*}
A_n(\bm{0}) &:=V_n(\bm{0}) + \left(B_n(\bm{0}) - g(\bm{0})\sqrt{A_n h_1 \dots h_d}B^{(d,p)}M_n^{(d,p)}(\bm{0})\right) \\ 
&\stackrel{d}{\to} N\left( \left(
\begin{array}{c}
0 \\
\vdots \\
0
\end{array}
\right), g(\bm{0})\left\{\kappa(\eta^2(\bm{0}) + \sigma_\varepsilon^2(\bm{0}))+\eta^2(\bm{0})g(\bm{0})\int \sigma_{\bm{e}}(\bm{v})d\bm{v}\right\}\mathcal{K}\right).
\end{align*} 
This and the result in Step 1 yield the desired result. 
\end{proof}

\section{Supplement}

The supplement contains simulation results (Section \ref{sec:simulation}), discussion on our assumptions and possible extensions (Section \ref{sec:discuss}), two-sample test for spatially dependent data (Section \ref{sec:TST}), discussion on examples of random fields to which our theoretical results can be applied (Section \ref{sec:examples}),  proofs for Section \ref{sec:main} (Section \ref{Appendix: Sec4}), Section \ref{sec:LP-unif} (Section \ref{Appendix: Sec5}), Section \ref{sec:TST} (Section \ref{sec: TST-proof}), Section \ref{sec:examples} (Section \ref{Appendix: Sec6}), and a list of technical tools (Section \ref{Appendix: tool}).

\section{Simulation}\label{sec:simulation}

This section examines the finite sample properties of the local polynomial estimation
applied to simulated spatial data. Focusing on the local linear estimation with $p=2$ on
$\mathbb{R}^2$, we simulate spatially correlated data that includes mean function on irregularly spaced locations on a rectangular and fit a local linear function to confirm how the asymptotic normality established in Theorem 4.1 works empirically.

Let us introduce the simulation designs to simulate spatial data on the square of \[
R_n=\prod_{j=1}^2\left[-{A_{n,j}\over 2},{A_{n,j} \over 2}\right] \subset \mathbb{R}^2,\ A_{n,1}=A_{n,2}=10.
\]
Following the model in (\ref{eq:model}), we simulate them on $\bm{x} \in R_n$ by
\begin{align*}
    Y(\bm{x}_i)=m(\bm{x}_i/A_n)+e(\bm{x}_i)+\varepsilon_i, i=1,\ldots,n,
\end{align*}
where the mean function is designed on $[-1/2,1/2]^2$ by
\begin{align}
    m(x_1,x_2)=(10x_1+15)\cos(x_1+x_2+1), |x_1|<0.5, |x_2|<0.5
    \label{mean_func}
\end{align}
The next component $e$, the spatially correlated component, is designed by CAR(1) random fields of \cite{BrMa17}
driven by compound Poisson processes, which are described as
\begin{align*}
    e(\bm{x})&=\sum_{j}Z_je^{-\lambda||\bm{x}-\bm{a}_j||}, \bm{x}\in\mathbb{R}^2,\lambda>0,
\end{align*}
where $Z_j$ is i.i.d. Gaussian with mean 0 and variance $\tau^2$, while
$\{\bm{a}_j\}$, a set of knots over $\mathbb{R}^2$, 800 points of which
are designed in practice uniformly over $\prod_{j=1}^2[-A_{n,j}, A_{n,j}]$ to simulate the spatial component on $\prod_{j=1}^2[-A_{n,j}/2, A_{n,j}/2]$.
The final component $\varepsilon$ is Gaussian i.i.d. error with mean 0 and variance $\sigma^2$ that represents a measurement error.

We conduct here the simulations for the three cases of the error term with the mean function defined in (\ref{mean_func}). Specifically, they are given by:
\begin{enumerate}
\item[(i)] no spatial component of $e$ with Gaussian iid noise of $\varepsilon$ with mean 0 and variance $\sigma^2=1$,
\item[(ii)] spatial component of $\lambda=1, \tau^2=0.1^2$ with Gaussian iid noise with mean 0 and variance $\sigma^2=0.1^2$,
and \item[(iii)] spatial component of $\lambda=0.5, \tau^2=0.05^2$ with Gaussian iid noise with mean 0 and variance $\sigma^2=0.1^2$.
\end{enumerate}
Notice that (1) stands for independent error case, while the others are spatially correlated error cases in which (iii) is more correlated than (ii).

We conducted the local linear estimation at the origin $(0,0)$ 500 times for the three cases of simulated data
observed on 1000 uniformly distributed sampling points on $\prod_{j=1}^2[-A_{n,j}/2, A_{n,j}/2]$ with $A_{n,1}=A_{n,2}=10$,
where we employed $h_1=h_2=0.2$ for the bandwidth to conduct the local linear fit, while $h_1=h_2=0.25$ to estimate the asymptotic bias and $h_1=h_2=0.25$, $b_1=b_2=8$ to estimate the asymptotic variance in Theorem 4.1, with the product of triangular kernel and the Bartlett kernel given by
\begin{align*}
    K(x_1, x_2)=
        \left\{
    \begin{array}{ll}
    (1-|x_1|)(1-|x_2|),& |x_1|\leq 1, |x_2|\leq 1.\\
    0,&otherwise,
    \end{array}
    \right.
\end{align*}
\begin{align*}
    \bar{K}(x_1, x_2)=
        \left\{
    \begin{array}{ll}
    (1-\sqrt{x_1^2 + x_2^2}),& x_1^2 + x_2^2\leq 1.\\
    0,&otherwise.
    \end{array}
    \right.
\end{align*}
After normalizing the estimated intercept $\hat \beta_0 := \hat \beta_0(0,0)$ in the local linear estimation with the estimated bias and variance, {\it i.e.} as
\begin{align*}
\hat T=\frac{\hat\beta_0-\hat{bias}-\beta_0}{\sqrt{\hat{var}}},
\end{align*}
we list the  empirical mean, empirical variance, empirical coverage ratios by the 95\% confidence interval in Table \ref{table1}
and histogram in Figure \ref{fig1} of $\hat T$ for the cases of (i), (ii), and (iii), where one estimator in case (ii) smaller than $-10$ is excluded from the mean and variance evaluations as an outlier.
Notice that the mean and variance of $\hat T$ are asymptotically 0 and 1, respectively.

We find from Figure \ref{fig1} that the histograms for all three cases are well approximated by standard Normal distribution. The empirical coverage ratios are close to the asymptotic value of 0.95. We find, however, from Table \ref{table1} that empirical variance is greater than 1, the asymptotic value, for cases (ii) and (iii) together with some negative bias. The deviations come from underestimation for the variance estimator proposed in Section 5.1. The underestimation especially for case (ii) may be due to the unsatisfactory bandwidth choice. In other words, the bandwidth, which is fixed to be 0.25 for the bias and variance estimation, fails to estimate proper values in the sense of approximating the asymptotic distribution. It suggests that the bandwidth choice is difficult in practice, especially for the bias and variance estimation that has critical effects on the statistical inference performances.

There have been several ways proposed to select optimal bandwidth. One practical method is to find the value minimizing the mean squared error of the estimator, which is the summation of the squared bias and variance given in Remark 4.1. Since they include unknown quantities of derivatives of an unknown mean function, one more bandwidth is necessary to estimate the unknown quantities from which optimal bandwidth is fixed to minimize the mean squared error. In practice, several bandwidths should be tried for the estimation among which the optimal bandwidth should be selected to minimize the mean squared error.  

\begin{figure}
  \includegraphics[width=0.8\textwidth]{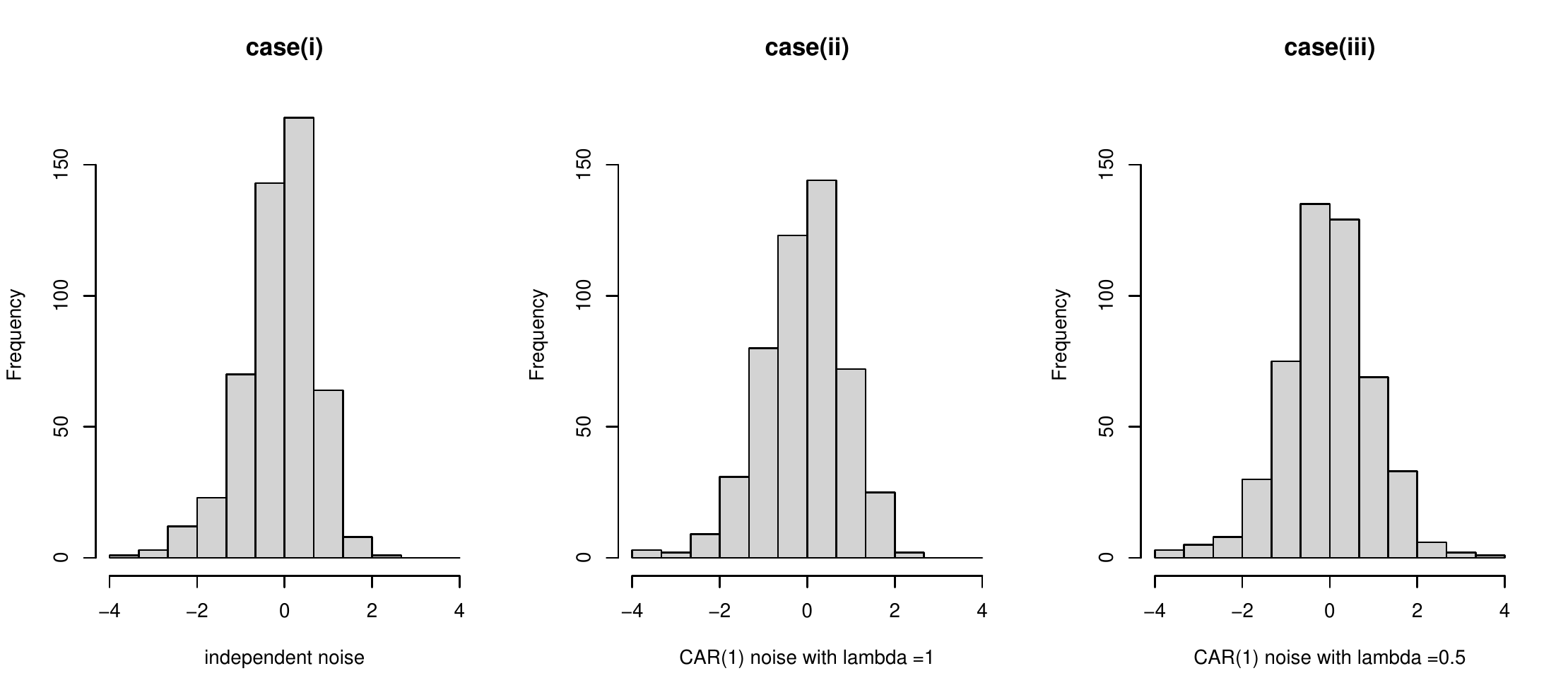}
  \caption{Histograms for the estimated intercepts normalized with the estimated bias and variance in the local linear fitting for the cases of (i) independent noise, (ii) CAR(1) noise with $\lambda=1$ and (iii) CAR(1) noise with $\lambda=0.5$.  }
    \label{fig1}
\end{figure}

\begin{table}[t]
 \centering
  \begin{tabular}{clll}
   \hline
   case & mean & variance & coverage \\
   \hline
   (i)   & $-0.199$ &1.012 &0.944\\
   (ii)  & $-0.179$ &1.517 &0.948\\
   (iii) & $-0.120$ &1.380 &0.940\\
   \hline
  \end{tabular}
 \caption{Empirical mean and variance of the estimated intercept in the local linear fitting together with the coverage ratio by 95\% confidence interval estimate for the cases of (i) independent noise, (ii) CAR(1) noise with $\lambda=1$ and (iii) CAR(1) noise with $\lambda=0.5$.}  
  \label{table1}
\end{table}

\section{Discussion}\label{sec:discuss}

In this section, we will discuss the conditions assumed for the main results as well as possible extensions.

\subsection{Discussion on the decay rate of mixing coefficients}

Assume that $\alpha(a;b) \leq a_1(a)\varpi_1(b)$ with $a_1(a) \lesssim a^{-r_1}$ and $\varpi_1(b) \lesssim b^{r_2}$ where $r_1$ and $r_2$ are positive constants. Moreover, let $A_{n,j} \sim n^{\zeta_{0}/d}$, $A_{n1,j} = A_{n,j}^{\zeta_1}$, $A_{n2,j} = A_{n1,j}^{\zeta_2}$, and $h_j \sim n^{-\zeta_3/d}$ where $\zeta_{0}$, $\zeta_{1}$, $\zeta_{2}$, and $\zeta_3$ are positive constants. Then, the assumptions of Theorem 4.1 are satisfied with  
\begin{align*}
\zeta_0 &\in \left(0, 1\right],\ \zeta_1 \in \left({2dp(1+r_2) \over (d+2p)(r_1+d)}, {2p+2 \over d + 2p+2}\right),\\ 
\zeta_2 &\in \left(\max\left\{{d(1+r_2(1-2/q)) \over r_1(1-2/q)}, \left({2p(1+r_2) \over \zeta_1(d+2p)}-1\right){d \over r_1}\right\}, 1\right), \\
\zeta_3 &\in \left(\max\left\{{\zeta_0(1+r_2 -\zeta_1(1+{r_1\zeta_2 \over d})) \over 1+r_2}, {\zeta_0 d \over d+2p+2}\right\}, \min\left\{\zeta_0(1-\zeta_1), {\zeta_0 d \over d+2p}\right\}\right).\\
r_1 &> \max\left\{{2d \over 1-4/q}, {d(2+r_2(1-2/q)) \over 1-2/q}, {d(p(d+2p+2)r_2 -d) \over (p+1)(d+2p)}\right\},\ r_2 >0.
\end{align*} 
Furthermore, assume that $\beta(a;b) \leq \beta_1(a)\varpi_2(b)$ with $\beta_1(a) \lesssim a^{-\bar{r}_1}$ and $\varpi_2(b) \lesssim b^{\bar{r}_2}$ where $s_1$ and $s_2$ are positive constants. Moreover, let $A_{n,j} \sim n^{\zeta_{0}/d}$, $A_{n1,j} = A_{n,j}^{\zeta_1}$, $A_{n2,j} = A_{n1,j}^{\zeta_2}$, and $h_j \sim n^{-\zeta_3/d}$ where $\zeta_{0}$, $\zeta_{1}$, $\zeta_{2}$, and $\zeta_3$ are positive constants. Then, the assumptions of Theorem 5.1 are satisfied with  
\begin{align*}
\zeta_0 &\in \left({2(1+{2 \over q_2}(\bar{r}_2 + {1 \over 2})) \over {\bar{r}_1 \over d} - 2\bar{r}_2} + {2 \over q_2}, 1\right],\ \zeta_1 \in \left({1 + {2 \over q_2}(\bar{r}_2 + {1 \over 2}) \over \zeta_0({\bar{r}_1 \over d} - 2\bar{r}_2)}, {1 \over 2} - {1 \over \zeta_0q_2}\right),\\ 
\zeta_2 &\in \left(\max\left\{{d(1+\zeta_0 \zeta_1(\bar{r}_2 - {1 \over 2})) \over \zeta_0 \zeta_1 \bar{r}_1}, {d(1+2\zeta_0 \zeta_1 \bar{r}_2 + {2 \over q_2}(\bar{r}_2 + {1 \over 2})) \over \zeta_0 \zeta_1 \bar{r}_1 }\right\}, 1\right), \\
\zeta_3 &\in \left(\max\left\{0, {1 + \zeta_0(\bar{r}_2 + {1 \over 2} - \zeta_1(1 + {\bar{r}_1 \zeta_2 \over d})) \over s_2 + {1 \over 2}}\right\}, \min\left\{\zeta_0(1-\zeta_1), \zeta_0(1-2\zeta_1) - {2 \over q_2}\right\}\right).\\
\bar{r}_1 &> 2d\left({q_2(1+{2 \over q_2}(\bar{r}_2 + {1 \over 2})) \over q_2 - 2}+\bar{r}_2\right),\ \bar{r}_2 >0.
\end{align*}

\subsection{Discussion on the definitions of $\alpha$- and $\beta$-mixing coefficients}

The definitions of the $\alpha$- and $\beta$-mixing coefficients are based on the argument in \cite{Br89}. It is crucial to restrict the size of the index sets $T_{1}$ and $T_{2}$ in the definition of $\alpha$- (or $\beta$-) mixing coefficients since no restrictions on $T_1$ and $T_2$ make the $\alpha$- and $\beta$-mixing be equivalent to $m$ dependent for a fixed $m>0$, which would not work for our asymptotic inference. Let us define the $\beta$-mixing coefficient of a random field $\bm{e}$ similarly to the time series as follows: For any subsets $T_{1}$ and $T_{2}$ of $\mathbb{R}^{d}$, the $\beta$-mixing coefficient between $\mathcal{F}_{\bm{e}}(T_1)$ and $\mathcal{F}_{\bm{e}}(T_2)$ is defined by $\tilde{\beta}(T_{1}, T_{2}) = \sup  {\textstyle \sum_{j=1}^{J}\sum_{k=1}^{K}}|P(A_{j}\cap B_{k}) - P(A_{j})P(B_{k})|/2$,
where the supremum is taken over all partitions $\{ A_j\}_{j=1}^{J} \subset \mathcal{F}_{\bm{e}}(T_1)$ and $\{ B_k\}_{k=1}^{K} \subset \mathcal{F}_{\bm{e}}(T_2)$ of $\mathbb{R}^d$. Let $\mathcal{O}_{1}$ and $\mathcal{O}_{2}$ be half-planes with boundaries $L_{1}$ and $L_{2}$, respectively. For each $a>0$, define $\beta(a) = \sup\{ \tilde{\beta}(\mathcal{O}_{1},\mathcal{O}_{2}) : d(\mathcal{O}_{1},\mathcal{O}_{2}) \geq a\}$. 
According to Theorem 1 in \cite{Br89}, if $\{e(\bm{x}): \bm{x} \in \mathbb{R}^2\}$ is strictly stationary, then $\beta(a) = 0$ or $1$ for $a >0$. This implies that if a random field $\bm{e}$ is $\beta$-mixing ($\lim_{a \to \infty}\beta(a)=0$), then it is automatically $m$ dependent, that is, $\beta(a)=0$ for some $a>m$, where $m$ is a positive constant. To allow a certain flexibility, we restrict the size of $T_1$ and $T_2$ in the definitions of $\alpha(a;b)$ and $\beta(a;b)$. We refer to \cite{Br93} and \cite{Do94} for more details on mixing coefficients for random fields.

\subsection{Discussion on $\beta$-mixing conditions}

\cite{La03b} established central limit theorems for weighted sample means of bounded spatial data under $\alpha$-mixing conditions. Lahiri's proof relies essentially on approximating the characteristic function of the weighted sample mean by that of independent blocks using the Volkonskii-Rozanov inequality (cf. Proposition 2.6 in \cite{FaYa03}) and then showing that
the characteristic function corresponding to the independent blocks converges to the characteristic function of its Gaussian limit. However, characteristic functions are difficult to capture the uniform behavior of LP estimators over a compact set so we rely on a different argument from that of \cite{La03b}. Indeed, we use a blocking argument tailored to $\beta$-mixing sequences (cf. Corollary 2.7 in \cite{Yu94}) and this enables us to compare the uniform convergence rates of LP estimators with that of a sum of independent blocks that approximates LP estimators. Another approach for handling spatial dependence is $m$-dependent approximation under a physical dependence structure (cf. \cite{MaVoWu13}), but this approach is designed for regularly spaced spatial data on $\mathbb{Z}^d$ and does not work in our framework. We also note that it is not known that the results corresponding to Corollary 2.7 in \cite{Yu94} hold for $\alpha$-mixing sequences; see Remark (ii) right after the proof of Lemma 4.1 in \cite{Yu94}.

\subsection{Discussion on our approach to prove main results}

If we use a result for proving a CLT without blocking, such as the result of Bolthausen (1982), as an alternative proof approach, the assumptions (Assumption 2.3(iii) and Assumption 4.1(v)) that arise from our blocking argument may become simpler. However, it is still expected that sufficient conditions for a CLT of LP estimators depending on the sample size, bandwidths, and spatial expansion rate $A_n$ will be required, as in Assumption (M) of \cite{LuTj14}, who used Bolthausen's technique to prove the CLT for kernel density estimator of stationary spatial processes. Furthermore, Bolthausen's result pertains to deterministic sampling sites that do not take into account stochastic sampling design. Therefore, how we can apply these results to our framework remains unclear. Even if it is applicable, the proof approach would be significantly different, we believe that additional substantial work would be necessary to derive detailed sufficient conditions for a CLT of our LP estimator in our framework.

\subsection{Discussion on the dependence structure of L\'evy-driven moving average random fields}

\cite{MaSt07} and \cite{ScSt12} have shown that for $d=1$ (continuous time process), a class of CARMA processes, which is a special class of L\'evy-driven moving average (MA) random fields, is exponentially $\beta$-mixing. From this, we expect that if the coefficients $\phi_{jk}$ decay exponentially fast, then the mixing coefficients $\alpha_1(a)$ (or $\beta_1(a)$) will decay (sub-)exponentially. However, it would be difficult to prove this in practice. This is because, first, the corresponding results for CARMA random fields on $d \geq 2$ are not known, and secondly, while examples of alpha-mixing Gaussian random fields on $\mathbb{Z}^d$ have been provided (see \cite{Do94} for example), there seem no general results known about the rate of decay of the alpha-mixing coefficients for Gaussian random fields on $\mathbb{R}^d$ that we are aware of.

One of the objectives of this paper is to demonstrate that a wide class of L\'evy-driven MA random fields satisfies the assumptions necessary for deriving the limiting theorem for our LP estimator. This class of random fields includes CARMA random fields as a special case and can handle a very broad class of random fields, not only Gaussian but also non-Gaussian ones (\cite{BrMa17}). Therefore, we leave the specific calculation of the mixing coefficient for L\'evy-driven MA random fields as future work.

\subsection{Construction of confidence surfaces of the mean function}

As an extension of Theorem 4.1, it is straightforward to show joint asymptotic normality of $\hat{m}$ over a finite number of design points $\{(\bm{z}_{\ell})\}_{\ell=1}^{L} \subset (-1/2,1/2)^d$ such that $\bm{z}_{\ell_1} \neq \bm{z}_{\ell_2}$ if $\ell_1 \neq \ell_2$ and verify that $\hat{m}(\bm{z}_{\ell})$ are asymptotically independent. Building on the result, we can construct a simple confidence surface by plug-in methods and linear interpolations of the following joint confidence intervals: Let $\xi_{1},\hdots, \xi_{L}$ be i.i.d. standard normal random variables, and let $\mathfrak{q}_{1-\tau}$ satisfy $P\left(\max_{1 \leq j \leq L}|\xi_{j}| > \mathfrak{q}_{1-\tau}\right) = \tau$ for $\tau \in (0,1)$ and take bandwidths $h_1,\dots, h_d$ that satisfies assumptions of Corollary 4.1 by replacing $\bm{0}$ with $\bm{z}_\ell$, $\ell=1,\dots, L$. 
Then, 
\[
\bar{C}_{n,\ell}(1-\tau) = \left[\hat{m}(\bm{z}_{\ell}) \pm \sqrt{{\hat{W}_{n}(\bm{z}_\ell) (S^{-1}\mathcal{K}S^{-1})_{11}\over A_nh_1\dots h_d}}\mathfrak{q}_{1-\tau}\right],\ \ell = 1,\hdots, L
\] 
are joint asymptotic $100(1-\tau)$\% confidence intervals of $m$. Here, we used the shorthand notation $[a \pm b] = [a-b,a+b]$ for $a \in \R$ and $b > 0$, and $(S^{-1}\mathcal{K}S^{-1})_{11}$ is the $(1,1)$-component of $S^{-1}\mathcal{K}S^{-1}$. More generally, We think there could be two possible ways to construct confidence bands of the regression function. The first way is based on a Gumbel approximation as considered in \cite{ZhWu08}, for example. For example, the second way is based on intermediate (high-dimensional) Gaussian approximations as considered in \cite{HoLe12}. However, we believe that both approaches require additional substantial work and as far as we could check, there would be no previous studies on the construction of uniform confidence surfaces for locally stationary random fields. Therefore, we leave the extension as a future research topic.


\section{Two-sample test for spatially dependent data}\label{sec:TST}

In this section, we discuss a two-sample test for the partial derivatives of the mean function as an application of our main results. Focusing on local linear estimation with $p=1$ on $\mathbb{R}^2$,

Consider the following nonparametric regression model: 
\begin{align*}
Y_1(\bm{x}_{1,\ell_1}) &= m_1\left({\bm{x}_{1,\ell_1} \over A_n}\right) + \eta_1\left({\bm{x}_{1,\ell_1} \over A_n}\right)e_1(\bm{x}_{1,\ell_1}) + \sigma_{\varepsilon,1}\left({\bm{x}_{1,\ell_1} \over A_n}\right)\varepsilon_{1,\ell_1},\ \ell_1 = 1,\dots,n_1\\
Y_2(\bm{x}_{2,\ell_2}) &= m_2\left({\bm{x}_{2,\ell_2} \over A_n}\right) + \eta_2\left({\bm{x}_{2,\ell_2} \over A_n}\right)e_2(\bm{x}_{2,\ell_2}) + \sigma_{\varepsilon,2}\left({\bm{x}_{2,\ell_2} \over A_n}\right)\varepsilon_{2,\ell_2},\ \ell_2 = 1,\dots,n_2,
\end{align*}
where $\bm{x}_{1,\ell_1}, \bm{x}_{2,\ell_2} \in R_n$, $\bm{e} = \{e(\bm{x}) = (e_1(\bm{x}), e_2(\bm{x}))': \bm{x} \in \mathbb{R}^d\}$ is a bivariate stationary random field such that $E[e_k(\bm{0})] = 0$, $E[e_k^2(\bm{0})] = 1$, and $\{\varepsilon_{k,\ell_k}\}$ is a sequence of i.i.d. random variables such that $E[\varepsilon_{k,\ell_k}] = 0$, $k=1,2$. 

Assume that $\{\bm{x}_{k,\ell_k}\}$ are realizations of a sequence of random variables $\{\bm{X}_{k, \ell_k}\}$ with density $A_n^{-1}g_k(\cdot/A_n)$ where $g_k(\cdot)$ is a probability density function with support $[-1/2,1/2]^d$, $k=1,2$. This allows the sampling sites $\{\bm{x}_{1,\ell_1}\}$ and $\{\bm{x}_{2,\ell_2}\}$ to be different. 

\begin{assumption}\label{Ass:RF2}
The bivariate random field $\bm{e}$ satisfies the following conditions: 
\begin{itemize}
\item[(i)] $E[|e_k(\bm{0})|^{q_2}]<\infty$, $k=1,2$ for some integer $q_2 > 4$. 
\item[(ii)] Define $\Sigma_{\bm{e}}(\bm{x}) = (\sigma_{\bm{e},jk}(\bm{x}))_{1 \leq j,k \leq 2}$ where $\sigma_{\bm{e}, jk}(\bm{x})= E[e_j(\bm{0})e_k(\bm{x})]$, $j,k =1, 2$. Assume that $\sigma_{\bm{e}, kk}(\bm{0})=1$, $k=1,2$ and $\int_{\mathbb{R}^d} |\sigma_{\bm{e},jk}(\bm{v})|d\bm{v}<\infty$, $j,k=1,2$.
\item[(iii)] The random field $\bm{e}$ is $\alpha$-mixing with mixing coefficients $\alpha(a;b)\leq \alpha_1(a)\varpi_1(b)$ such that as $n \to \infty$, 
\begin{align*}
A_n^{(1)} \left(\alpha_1^{1-2/q}(\underline{A}_{n2}) + \sum_{k=\underline{A}_{n1}}^{\infty}k^{d-1}\alpha_1^{1-2/q}(k)\right)\varpi_1^{1-2/q}(A_n^{(1)}) \to 0, 
\end{align*}
where $q = \min\{q_1,q_2\}$, 
\begin{align*}
A_n^{(1)}=\prod_{j=1}^{d}A_{n1,j},\ \underline{A}_{n1}=\min_{1 \leq j \leq d}A_{n1,j},\ \underline{A}_{n2}=\min_{1 \leq j \leq d}A_{n2,j}.  
\end{align*}
Here, $\{A_{n1,j}\}_{n \geq 1}$ and $\{A_{n2,j}\}_{n \geq 1}$ are sequences of constants with $\min\left\{A_{n2,j},  {A_{n1,j} \over A_{n2,j}}\right\} \to \infty$ as $n \to \infty$, and $q_1$ is the integer that appear in Assumption 2.1 
\item[(iv)] $\{\bm{X}_{1,\ell_1}\}_{\ell_1 = 1}^{n_1}$, $\{\bm{X}_{2,\ell_2}\}_{\ell_2 = 1}^{n_2}$, $\bm{e}$, $\{\varepsilon_{1,\ell_1}\}_{\ell_1=1}^{n_1}$, and $\{\varepsilon_{2,\ell_2}\}_{\ell_2 = 1}^{n_2}$ are mutually independent. 
\end{itemize}
\end{assumption}

In Section \ref{sec:examples}, we give examples of bivariate random fields that satisfy Assumptions 4.1 
and \ref{Ass:RF2}. We note that a wide class of bivariate L\'evy-driven MA random fields satisfies our assumptions.

We are interested in testing the null hypothesis
\begin{align}\label{null-TST}
\mathbb{H}_{0, j_1 \dots j_L}: \partial_{j_1 \dots j_L}m_1(\bm{0}) - \partial_{j_1 \dots j_L}m_2(\bm{0}) = 0
\end{align}
against the alternative $\mathbb{H}_{1, j_1 \dots j_L}: \partial_{j_1 \dots j_L}m_1(\bm{0}) - \partial_{j_1 \dots j_L}m_2(\bm{0}) \neq 0$.

Define $\bm{M}_{k}(\bm{0})$ as $\bm{M}(\bm{0})$ with $m = m_{k}$ and $\overline{\bm{\beta}}_k(\bm{0})$ as LP estimators of order $p$ for $\bm{M}_k(\bm{0})$ computed by using $\{(Y_k(\bm{x}_{k, \ell_k}), \bm{x}_{k,\ell_k})\}$, bandwidths $h_1,\dots, h_d$, and a common kernel function $K$, $k=1,2$, respectively. The next theorem is a building block of the two-sample test (\ref{null-TST}). 

\begin{proposition}\label{prp: LP-TST}
Suppose Assumptions 2.2, 2.2 (i), 3.1, 4.1, 
and \ref{Ass:RF2} hold with $m = m_k$, $\eta = \eta_k$, $\sigma_{\varepsilon} = \sigma_{\varepsilon,k}$, $\{\varepsilon_j\} = \{\varepsilon_{k,\ell_k}\}$, $g = g_k$, $k=1,2$. Moreover, assume that $n = n_1$, $n_1/n_2 \to \theta \in (0,\infty)$ as $n_1 \to \infty$ and $(\eta_1(\bm{0}), -\eta_2(\bm{0}))\left(\int \Sigma_{\bm{e}}(\bm{v})d\bm{v}\right)(\eta_1(\bm{0}), -\eta_2(\bm{0}))'\geq 0$. 
Then, as $n \to \infty$, 
\begin{align*}
&\sqrt{A_{n} h_1 \dots h_d}\left\{H\left((\overline{\bm{\beta}}_1(\bm{0}) - \overline{\bm{\beta}}_2(\bm{0}))- (\bm{M}_1(\bm{0}) - \bm{M}_2(\bm{0}))\right) - (\overline{B}_{n1}(\bm{0}) - \overline{B}_{n2}(\bm{0}))\right\}\\ 
&\quad \stackrel{d}{\to} N\left(\left(
\begin{array}{c}
0 \\
\vdots \\
0
\end{array}
\right), \left(\overline{V}_1(\bm{0}) + \overline{V}_2(\bm{0}) - 2\overline{V}_3(\bm{0}) \right)S^{-1}\mathcal{K}S^{-1}\right),
\end{align*}
where
\begin{align*}
\overline{B}_{n1}(\bm{0}) &= S^{-1}B^{(d,p)}\bm{M}_{n1}^{(d,p)}(\bm{0}),\ \overline{B}_{n2}(\bm{0}) = S^{-1}B^{(d,p)}\bm{M}_{n2}^{(d,p)}(\bm{0}),\\
\overline{V}_1(\bm{0}) &= \left({\kappa(\eta_1^2(\bm{0}) + \sigma_{\varepsilon,1}^2(\bm{0})) \over g_1(\bm{0})} + \eta_1^2(\bm{0})\int \sigma_{\bm{e},11}(\bm{v})d\bm{v}\right),\\ 
\overline{V}_2(\bm{0}) &= \left({\theta\kappa(\eta_2^2(\bm{0}) + \sigma_{\varepsilon,2}^2(\bm{0})) \over g_2(\bm{0})} + \eta_2^2(\bm{0})\int \sigma_{\bm{e},22}(\bm{v})d\bm{v}\right),\\
\overline{V}_3(\bm{0}) &= \eta_1(\bm{0})\eta_2(\bm{0})\int \sigma_{\bm{e},12}(\bm{v})d\bm{v},
\end{align*}
where $\bm{M}_{nk}^{(d,p)}(\bm{0})$ are defined as $\bm{M}_{n}^{(d,p)}(\bm{0})$ with $m = m_k$. 
\end{proposition}

An estimator of the asymptotic variance of the statistics $\overline{\bm{\beta}}_1(\bm{0}) - \overline{\bm{\beta}}_2(\bm{0})$ can be constructed as follows. For $\bm{z} \in (-1/2,1/2)^d$, let $\widehat{m}_k(\bm{z})$ be the LP estimator (of order $p$) of $m_k(\bm{z})$, $k=1,2$.

Define 
\begin{align*}
\overline{g}_{n_k}(\bm{0}) &= {1 \over n_kh_1 \dots h_d}\sum_{\ell_k=1}^{n_k}K_{Ah}(\bm{X}_{k,\ell_k}),\ k=1,2,\\
\overline{V}_{n,k}(\bm{0}) &= {A_n \over n_k^2h_1 \dots h_d}\sum_{\ell_{k,1}, \ell_{k,2}=1}^{n_k}K_{Ah}(\bm{X}_{k,\ell_{k,1}})K_{Ah}(\bm{X}_{k,\ell_{k,2}})\bar{K}_b(\bm{X}_{k,\ell_{k,1}} - \bm{X}_{k,\ell_{k,1}})\\
&\quad \times \left(Y_k(\bm{X}_{k,\ell_{k,1}}) - \widehat{m}_k(\bm{X}_{k,\ell_{k,1}}/A_n)\right)\left(Y_k(\bm{X}_{k,\ell_{k,2}}) - \widehat{m}_k(\bm{X}_{k,\ell_{k,2}}/A_n)\right),\ k=1,2,\\
\overline{V}_{n,3}(\bm{0}) &= {A_n \over n_1n_2h_1 \dots h_d}\sum_{\ell_1=1}^{n_1}\sum_{\ell_2=1}^{n_2}K_{Ah}(\bm{X}_{1,\ell_1})K_{Ah}(\bm{X}_{2,\ell_2})\bar{K}_b(\bm{X}_{1,\ell_1} - \bm{X}_{2,\ell_2})\\
&\quad \times \left(Y_1(\bm{X}_{1,\ell_1}) - \widehat{m}_{1}(\bm{X}_{1,\ell_1}/A_n)\right)\left(Y_2(\bm{X}_{2,\ell_2}) - \widehat{m}_{2}(\bm{X}_{2,\ell_2}/A_n)\right).
\end{align*}

\begin{proposition}\label{prp: var-est-TST}
Suppose that Assumptions 4.1, 5.1, 5.2 (i), 5.3 (iii), (iv), 5.4, 5.5, and \ref{Ass:RF2} hold with $\kappa=0$, $q_1 \geq q_2$, $m = m_k$, $\eta = \eta_k$, $\sigma_{\varepsilon} = \sigma_{\varepsilon,k}$, $\{\varepsilon_j\} = \{\varepsilon_{k,\ell_k}\}$, $g = g_k$, $k=1,2$ and with $\alpha$-mixing coefficients replaced by $\beta$-mixing coefficients. Moreover, assume that $n = n_1$, $n_1/n_2 \to \theta \in (0,\infty)$ as $n_1 \to \infty$ and $(\eta_1(\bm{0}), -\eta_2(\bm{0}))\left(\int \Sigma_{\bm{e}}(\bm{v})d\bm{v}\right)(\eta_1(\bm{0}), -\eta_2(\bm{0}))'\geq 0$. Then, as $n \to \infty$, the following result holds:
\begin{align*}
\check{V}_n(\bm{0}) &:= {\!\overline{V}_{n,1}(\bm{0})/\kappa_0^{(2)} \over \overline{g}_{n_1}^2(\bm{0})\!} \!+\! {\overline{V}_{n,2}(\bm{0})/\kappa_0^{(2)} \over \overline{g}_{n_2}^2(\bm{0})}- 2{(\overline{V}_{n,3}(\bm{0})/\kappa_0^{(2)}) \over \overline{g}_{n_1}(\bm{0})\overline{g}_{n_2}(\bm{0})}\\ 
&\stackrel{p}{\to} \overline{V}_1(\bm{0}) + \overline{V}_2(\bm{0}) - 2\overline{V}_3(\bm{0}).
\end{align*}
\end{proposition}

Define the test statistics
\begin{align*}
T_{n, j_1 \dots j_L} := {\sqrt{ A_n h_1 \dots h_d \left(\prod_{\ell=1}^{L}h_{j_\ell}\right)^2}\left(\partial_{j_1 \dots j_L}\widehat{m}_1(\bm{0}) - \partial_{j_1 \dots j_L}\widehat{m}_2(\bm{0})\right) \over \sqrt{\check{V}_n(\bm{0})\left(\bm{s}_{j_1\dots j_L}!\right)^2\left(e'_{j_1 \dots j_L}S^{-1}\mathcal{K}S^{-1}e_{j_1 \dots j_L}\right)}}.
\end{align*}

The asymptotic properties of the test statistics under both null and alternative hypotheses are given as follows: 
\begin{corollary}\label{cor: LP-CI-TST}
Let $\tau \in (0,1/2)$. Assume that $n = n_1$, $n_1/n_2 \to \theta \in (0,\infty)$ as $n_1 \to \infty$ and $(\eta_1(\bm{0}), -\eta_2(\bm{0}))\left(\int \Sigma_{\bm{e}}(\bm{v})d\bm{v}\right)(\eta_1(\bm{0}), -\eta_2(\bm{0}))'\geq 0$. Under the assumptions of Proposition \ref{prp: var-est-TST} with 
\[
A_n h_1 \dots h_d \left((S^{-1}e_{j_1\dots j_L})'B^{(d,p)}M_{n1}^{(d,p)}(\bm{0}) \right)^2 \to 0,\ n \to \infty,
\] 
we have $\lim_{n \to \infty}P(|T_{n,j_1 \dots j_L}| \geq q_{1-\tau/2}) = \tau$ under $\mathbb{H}_{0, j_1 \dots j_L}$ and $\lim_{n \to \infty}P(|T_{n,j_1 \dots j_L}| \geq q_{1-\tau/2}) = 1$ under $\mathbb{H}_{1, j_1 \dots j_L}$, where $q_{1-\tau}$ is the $(1-\tau)$-quantile of the standard normal random variable.
\end{corollary}

\section{Examples}\label{sec:examples}

In this section, we discuss examples of random fields to which our theoretical results can be applied. To this end, we consider L\'evy-driven moving average (MA) random fields and discuss their dependence structure.  L\'evy-driven MA random fields include many Gaussian and non-Gaussian random fields and constitute a flexible class of models for spatial data.  We refer to \cite{Be96} and \cite{Sa99} for standard references on L\'evy processes, and \cite{RaRo89} and \cite{Ku22a} for details on the theory of infinitely divisible measures and fields. In particular, we show that a broad class of L\'evy-driven MA random fields, which includes continuous autoregressive and moving average (CARMA) random fields as special cases (cf. \cite{BrMa17}), satisfies our assumptions. 

For the two-sample test discussed in Section 4, 
we considered nonparametric regression models for spatial data $\{Y_1(\bm{x}_{1,\ell_1}), Y_2(\bm{x}_{2,\ell_2})\}$ with bivariate random field $\bm{e} = \{e(\bm{x})=(e_1(\bm{x}),e_2(\bm{x}))':\bm{x} \in \mathbb{R}^d\}$. Hence, we give examples of bivariate random fields that satisfy Assumptions 4.1 
and \ref{Ass:RF2}. The examples of univariate random fields that satisfy Assumptions 2.3 and 4.1, 
and Assumption 5.3 
can be given as special class of bivariate cases. Indeed, for univariate cases, it is sufficient to consider the first component of the examples of bivariate random fields.

Let $\bm{L} = \{\bm{L}(A) = (L_1(A), L_2(A))': A \in \mathcal{B}(\mathbb{R}^d)\}$ be an $\mathbb{R}^2$-valued random measure on the Borel subsets $\mathcal{B}(\R^{d})$ that satisfies the following conditions: 
\begin{itemize}
\item[1.] For each sequence $\{A_m\}_{m \geq 1}$ of disjoint sets in $\mathbb{R}^{d}$,  
\begin{itemize}
\item[(a)] $\bm{L}(\cup_{m \geq 1}A_m) = \sum_{m \geq 1}\bm{L}(A_m)$ a.s. whenever $\cup_{m \geq 1}A_m \in \mathcal{B}(\mathbb{R}^d)$, 
\item[(b)] $\{\bm{L}(A_m)\}_{m \geq 1}$ is a sequence of independent random variables. 
\end{itemize}
\item[2.] For every Borel subset $A$ of $\mathbb{R}^{d}$ with finite Lebesgue measure $|A|$, $\bm{L}(A)$ has an infinitely divisible distribution, that is, 
\begin{align}\label{Levy-exp}
E[\exp(\mathrm{i}\bm{\theta}' \bm{L}(A))] &= \exp(|A|\psi(\bm{\theta})),\ \bm{\theta} \in \mathbb{R}^2,
\end{align}
where $\mathrm{i} = \sqrt{-1}$ and $\psi$ is the logarithm of the characteristic function of an $\mathbb{R}^2$-valued infinitely divisible distribution, which is given by 
\begin{align*}
\psi(\bm{\theta}) &= \mathrm{i} \bm{\theta}'\bm{\gamma}_{0} - {1 \over 2}\bm{\theta}'\Sigma_0 \bm{\theta} + \int_{\mathbb{R}^2}\left\{e^{\mathrm{i}\bm{\theta}' \bm{x}}-1-\mathrm{i} \bm{\theta}' \bm{x} 1_{\{\|\bm{x}\| \leq 1\}}\right\}\nu_{0}(d\bm{x}),
\end{align*}
where $\bm{\gamma}_0 = (\gamma_{0,1}, \gamma_{0,2})' \in \mathbb{R}^2$, $\Sigma_0 = (\sigma_{0,jk})_{1 \leq j,k \leq 2}$ is a $2\times 2$ positive semi-definite matrix, and $\nu_{0}$ is a L\'evy measure with $\int_{\mathbb{R}^2}\min\{1,\|\bm{x}\|^{2}\}\nu_{0}(d\bm{x})<\infty$. If $\nu_{0}(d\bm{x})$ has a Lebesgue density, i.e., $\nu_{0}(d\bm{x}) = \nu_{0}(\bm{x})d\bm{x}$, we call $\nu_{0}(\bm{x})$ as the L\'evy density. The triplet $(\bm{\gamma}_0, \Sigma_0, \nu_{0})$ is called the L\'evy characteristic of $\bm{L}$ and uniquely determines the distribution of $\bm{L}$. 
\end{itemize}

By equation (\ref{Levy-exp}), the first and second moments of the random measure $L$ are determined by
\[
E[L_j(A)] = \mu_j^{(\bm{L})}|A|,\ \Cov(L_j(A), L_k(A)) = \sigma_{j,k}^{(\bm{L})}|A|,
\] 
where $\mu_j^{(\bm{L})} = -\mathrm{i} {\partial \psi(\bm{0}) \over \partial \theta_j}$ and $\sigma_{j,k}^{(\bm{L})} = -{\partial^{2}\psi(\bm{0}) \over \partial \theta_j \partial \theta_k}$. 

The following are a couple of examples of L\'evy random measures. 
\begin{itemize}
\item If $\psi(\bm{\theta}) = -\bm{\theta}' \Sigma_0^2 \bm{\theta}/2$ with a $2 \times 2$ positive semi-definite matrix $\Sigma_0$, then $\bm{L}$ is a Gaussian random measure. 

\item If $\psi(\bm{\theta}) = \lambda \int_{\mathbb{R}^2}(\exp(\mathrm{i}\bm{\theta}' \bm{x}) - 1)F(d\bm{x})$, where $\lambda>0$ and $F$ is a probability distribution function with no jump at the origin, then $\bm{L}$ is a compound Poisson random measure with intensity $\lambda$ and jump size distribution $F$. More specifically, 
\begin{align*}
\bm{L}(A) &= \sum_{i=1}^{\infty}\bm{J}_{i}1_{\{\bm{s}_{i}\}}(A),\ A \in \mathcal{B}(\R^{d}),
\end{align*}
where $\bm{s}_{i}$ denotes the location of the $i$th unit point mass of a Poisson random measure on $\mathbb{R}^{d}$ with intensity $\lambda>0$ and $\{\bm{J}_{i}\}$ is a sequence of  i.i.d. random vectors in $\mathbb{R}^2$ with distribution function $F$ independent of $\{\bm{s}_{i}\}$. 
\end{itemize}

Let $\bm{\phi} = (\phi_{j,k})_{1 \leq j,k \leq 2}$ be a measurable function on $\mathbb{R}^{d}$ with $\phi_{j,k} \in L^{1}(\R^{d}) \cap L^{\infty}(\R^{d})$.  A bivariate L\'evy-driven MA random field with kernel $\bm{\phi}$ driven by a L\'evy random measure $\bm{L}$ is defined by
\begin{align}\label{CARMA}
\bm{e}(\bm{x}) &= \int_{\mathbb{R}^{d}}\bm{\phi}(\bm{x} - \bm{u})\bm{L}(d\bm{u}),\ \bm{x} \in \mathbb{R}^{d}. 
\end{align}
Define $\bm{\mu}_{\bm{L}} = (\mu_1^{(\bm{L})},\mu_2^{(\bm{L})})'$ and $\Sigma_{\bm{L}} = (\sigma_{j,k}^{(\bm{L})})_{1 \leq j,k \leq 2}$. The first and second moments of $\bm{e}(\bm{x})$ satisfy
\begin{align*}
E[\bm{e}(\bm{0})] &= \bm{\mu}_{(L)}\int_{\mathbb{R}^{d}}\bm{\phi}(\bm{u})d\bm{u},\ \Cov(\bm{e}(\bm{0}), \bm{e}(\bm{x})) = \int_{\mathbb{R}^{d}}\bm{\phi}(\bm{x}-\bm{u})\Sigma_{\bm{L}}\bm{\phi}(\bm{u})d\bm{u}.
\end{align*}
We refer to \cite{BrMa17} for more details on the computation of moments of L\'evy-driven MA processes.

Before discussing theoretical results, we look at some examples of univariate random fields defined by (\ref{CARMA}). Let $a_*(z) = z^{p_0} + a_{1}z^{p_0-1} + \cdots+a_{p_0} = \prod_{i=1}^{p_0}(z - \lambda_{i})$ be a polynomial of degree $p_0$ with real coefficients and distinct negative zeros $\lambda_{1},\hdots,\lambda_{p_0}$, and let $b_*(z) = b_{0} + b_{1}z + \cdots + b_{q_0}z^{q_0} = \prod_{i=1}^{q_0}(z - \xi_{i})$ be a polynomial of degree $q_0$ with real coefficients and real zeros $\xi_{1},\hdots, \xi_{q_0}$ such that $b_{q_0}=1$ and $0\leq q_0 < p_0$ and $\lambda_{i}^{2} \neq \xi_{j}^{2}$ for all $i$ and $j$. Define $a(z) = \prod_{i=1}^{p_0}(z^{2} - \lambda_{i}^{2})$ and $b(z) = \prod_{i=1}^{q_0}(z^{2} - \xi_{i}^{2})$. Then, the L\'evy-driven MA random field driven by an infinitely divisible random measure $L$ with 
\begin{align*}\label{CARMA_kernel_general}
\phi(\bm{x}) = \sum_{i=1}^{p_0}{b(\lambda_{i}) \over a'(\lambda_{i})}e^{\lambda_{i}\|\bm{x}\|},
\end{align*}
where $a'$ denotes the derivative of the polynomial $a$, is called a univariate (isotropic) CARMA($p_0,q_0$) random field. For example, if the L\'evy random measure of a CARMA random field is compound Poisson, then the resulting random field is called a compound Poisson-driven CARMA random field. In particular, when
\[
\phi(\bm{x}) = (1-\varsigma)\exp(\lambda_{1}\|\bm{x}\|) + \varsigma\exp(\lambda_{2}\|\bm{x}\|),
\]
where $\varsigma$ is a parameter that satisfies 
\[
-{\lambda_{2}^{2} - \xi^{2}\lambda_{1} \over \lambda_{1}^{2} - \xi^{2}\lambda_{2}} = {\varsigma \over 1-\varsigma},\ \lambda_{1}<\lambda_{2}<0,\ \xi \leq 0,
\]
then the random field (\ref{CARMA}) is called a CARMA($2,1$) random field. This random field includes normalized CAR($1$) (when $\varsigma=0$) and CAR($2$) (when $\varsigma = -\lambda_{1}/(\lambda_{2} - \lambda_{1})$) as special cases. See \cite{BrMa17} for more details. We note that although we focus on isotropic case, it is possible to extend the results in this section to anisotropic L\'evy-driven MA random fields.  

\begin{remark}[Connections to  Mat\'ern covariance functions]\label{Matern-cov-rem}
In spatial statistics, Gaussian random fields with the following Mat\'ern covariance functions play an important role \cite[cf.][]{Ma86,St99,GuGn06}:
\[
M(\bm{x}; \nu, a, \sigma) = \sigma^{2}\|a\bm{x}\|^{\nu}F_{\nu}(\|a\bm{x}\|),\ \nu>0, a>0, \sigma>0,
\]
where $F_{\nu}$ denotes the modified Bessel function of the second kind of order $\nu$ (we call $\nu$ the index of Mat\'ern covariance function). \cite{BrMa17} showed that in the univariate case, when the kernel function is $\phi(\bm{x}) = \|a\bm{x}\|^{\nu}F_{\nu}(\|a\bm{x}\|)$, which they call a Mat\'ern kernel with index $\nu$,  then the Levy-driven MA random field has a Mat\'ern covariance function with index $d/2 + \nu$. For example, a normalized CAR($1$) random field has a Mat\'ern covariance function since its kernel function is given by $\phi(\bm{x}) =  \exp(-\|\lambda_{1}\bm{x}\|) = \sqrt{(2/\pi)}\|\lambda_{1}\bm{x}\|^{1/2}F_{1/2}(\|\lambda_{1}\bm{x}\|)$ for some $\lambda_{1}<0$.
\end{remark}

In general, if $\bm{\phi}$ depends only on $\|\bm{x}\|$, i.e., $\bm{\phi}(\bm{x}) = \bm{\phi}(\|\bm{x}\|)$, then $\bm{e}$ is a strictly stationary isotropic random field and the second moment of $\bm{e}(\bm{x})$ satisfies
\begin{align*}
\Cov(\bm{e}(\bm{0}), \bm{e}(\bm{x})) = \int_{\mathbb{R}^{d}}\bm{\phi}(\|\bm{x}-\bm{u}\|)\Sigma_{\bm{L}}\bm{\phi}(\|\bm{u}\|)d\bm{u}.
\end{align*}

Consider the following decomposition: 
\begin{align*}
\bm{e}(\bm{x}) &= \!\!\int_{\mathbb{R}^{d}}\!\!\!\!\bm{\phi}(\bm{x} - \bm{u})\psi_{0}\left(\|\bm{x} - \bm{u}\| : m_{n}\right)\!\bm{L}(d\bm{u}) \!+\! \int_{\mathbb{R}^{d}}\!\!\!\!\bm{\phi}(\bm{x} - \bm{u})\!\!\left(1 - \psi_{0}\left(\|\bm{x} - \bm{u}\| : m_{n}\right)\right)\!\bm{L}(d\bm{u})\\
&=: \bm{e}_{1, m_{n}}(\bm{x}) + \bm{e}_{2, m_{n}}(\bm{x}),
\end{align*}
where $m_{n}$ is a sequence of positive constants with $m_{n} \to \infty$ as $n \to \infty$ and $\psi_{0}(\cdot:c) : \mathbb{R} \to [0,1]$ is a truncation function defined by
\begin{align*}
\psi_{0}(x:c) = 
\begin{cases}
1 & \text{if $|x| \leq c/4$},\\
-{4 \over c}\left(x-{c \over 2}\right) & \text{if $c/4 < |x| \leq c/2$},\\
0 & \text{if $x>c/2$}.
\end{cases}
\end{align*} 
The random field $\bm{e}_{1,m_n} = \{ \bm{e}_{1,m_{n}}(\bm{x}) = (e_{11,m_{n}}(\bm{x}),e_{12,m_{n}}(\bm{x}))'  : \bm{x} \in \mathbb{R}^d \}$ is $m_{n}$-dependent (with respect to the $\ell^{2}$-norm), i.e., $\bm{e}_{1, m_n}(\bm{x}_{1})$ and $\bm{e}_{1, m_n}(\bm{x}_{2})$ are independent if $\| \bm{x}_1-\bm{x}_2\| \geq m_{n}$. Also, if the tail of the kernel function $\bm{\phi}(\cdot)$ decays sufficiently fast, then the random field $\bm{e}_{2,m_n}=\{ \bm{e}_{2, m_{n}}(\bm{x})  = (e_{21,m_{n}}(\bm{x}),e_{22,m_{n}}(\bm{x}))' : \bm{x} \in \mathbb{R}^{d}\}$ is asymptotically negligible. In such cases, we can approximate $\bm{e}$ by the $m_{n}$-dependent process $\bm{e}_{1,m_n}$ and verify conditions on mixing coefficients in Assumptions 2.3, 4.1, 
and \ref{Ass:RF2} as shown in the following proposition. 

\begin{proposition}\label{Levy-MA-moment}
Consider a L\'evy-driven MA random field $\bm{e}$ defined by (\ref{CARMA}). Assume that $\phi_{j,k}(\bm{x}) = r_{0,jk} e^{-r_{1,jk} \|\bm{x}\|}$ where $|r_{0,jk}| > 0$ and $r_{1,jk}>0$, $j,k=1,2$. Additionally, assume that 
\begin{itemize}
\item[(a)] the random measure $\bm{L}(\cdot)$ is Gaussian with triplet $(0, \Sigma_0 , 0)$ or
\item[(b)] the random measure $\bm{L}(\cdot)$ is non-Gaussian with triplet $(\bm{\gamma}_0, 0 , \nu_{0})$, $\bm{\mu}_{(L)} = (0,0)'$, and the marginal L\'evy density $\nu_{0,j}(x)$ of $L_j(\cdot)$ is given by
\begin{align}
\nu_{0,j}(x) &= {1 \over |x|^{1 + \beta_{0,j}}}\left(C_{0,j}e^{-c_{0,j}|x|^{\alpha_{0,j}}} + {C_{1,j} \over 1 + |x|^{\beta_{1,j}}}\right)1_{\mathbb{R} \backslash \{0\}}(x),\ \label{LS-Levy} 
\end{align}
where $\alpha_{0,j} > 0$, $\beta_{0,j} \in (-1,2]$, $\beta_{1,j} >0$, $\beta_{0,j} + \beta_{1,j}>6$, $c_{0,j}>0$, $C_{0,j} \geq 0$, $C_{1,j} \geq 0$, and $C_{0,j} + C_{1,j}>0$, $j=1,2$.
\end{itemize}
Then $\bm{e}_{2,m_n}$ is asymptotically negligible, that is, we can replace $\bm{e}$ with $\bm{e}_{1,m_n}$ in the results in Section 4. 
Further, $\bm{e}_{1,m_n}$ satisfies Assumptions 2.3, 4.1, 
and \ref{Ass:RF2} with $A_{n,j} \sim n^{\zeta_{0}/d}$, $A_{n1,j} = A_{n,j}^{\zeta_1}$, $A_{n2,j} = A_{n1,j}^{\zeta_2}$, $m_{n} = \underline{A}_{n2}^{1/2}$, and $h_j \sim n^{-\zeta_3/d}$ where $\zeta_{0}$, $\zeta_{1}$, $\zeta_{2}$, and $\zeta_3$ are positive constants such that 
\begin{align*}
\zeta_0 &\in \left(0, \min\left\{1,{2p+2 \over d}\right\}\right],\ \zeta_1 \in \left({\zeta_0d \over d+2p+2}, {2p+2 \over d + 2p+2}\right),\\ 
\zeta_2 &\in \left(0, \min\left\{{2 \over 2+d \max\{1, \zeta_0\}}, 1-{\zeta_0 d \over \zeta_1(d + 2p+2)}, {2p+2 \over \zeta_1(d+2p+2)}-1\right\}\right), \\
\zeta_3 &\in \left({d\zeta_0 \over 2p+d+2}, \min\left\{{d\zeta_0 \over 2p+d}, \zeta_0\left(1-\zeta_1(1+\zeta_2)\right), \zeta_1\left(1-\left(1+{d \over 2}\zeta_0\right)\zeta_2\right)\right\}\right).
\end{align*} 
\end{proposition}

\begin{remark}
When $d = 2$ and $p \geq 1$, the conditions on $\{\zeta_j\}_{j=0}^{3}$ are typically satisfied when $\zeta_0 = 1$, $\zeta_1 = {3 \over 2p+4}$, $\zeta_2 \in \left(0, {1 \over 6}\right)$. The L\'evy density of the form (\ref{LS-Levy}) corresponds to a compound Poisson random measure if $\beta_{0,j} \in [-1,0)$, a Variance Gamma random measure if $\alpha_{0,j}=1$, $\beta_{0,j}=0$, $C_{1,j} = 0$, and a tempered stable random measure if $\beta_{0,j} \in (0,1)$, $C_{1,j}=0$ (cf. Section 5 in \cite{KaKu20}). It is straight forward to extend Proposition \ref{Levy-MA-moment} to the case that $\phi$ is a finite sum of kernel functions with exponential decay. Therefore, our results in Section 4 
can be applied to a wide class of CARMA($p_0,q_0$) random fields and extending the results to anisotropic CARMA random fields (cf. \cite{BrMa17}) is straightforward.  
\end{remark}

The next result provides examples of L\'evy-driven MA random fields that satisfy assumptions in Theorem 5.1. 

\begin{proposition}\label{Levy-MA-moment-unif}
Consider a univariate L\'evy-driven MA random field $\bm{e}$ defined by (\ref{CARMA}). Assume that $\phi(\bm{x}) = r_{0} e^{-r_{1} \|\bm{x}\|}$ where $|r_{0}| > 0$ and $r_{1}>0$. Additionally, assume Conditions (a) or (b) in Proposition \ref{Levy-MA-moment}. Then $\bm{e}_{2,m_n}$ is asymptotically negligible, that is, we can replace $\bm{e}$ with $\bm{e}_{1,m_n}$ in Theorem 5.1. 
Further, $\bm{e}_{1,m_n}$ satisfies Assumption 5.3 
with $A_{n,j} \sim n^{\zeta_{0}/d}$, $A_{n1,j} = A_{n,j}^{\zeta_1}$, $A_{n2,j} = A_{n1,j}^{\zeta_2}$, $m_{n} = \underline{A}_{n2}^{1/2}$, and $h_j \sim n^{-\zeta_3/d}$ where $\zeta_{0}$, $\zeta_{1}$, $\zeta_{2}$, and $\zeta_3$ are positive constants such that $\zeta_0  \in \left({2 \over q_2},1\right]$, $\zeta_1 \in \left(0, {1 \over 2} - {1 \over \zeta_0q_2}\right)$, $\zeta_2 \in (0,1)$, and $\zeta_3 \in \left(0, \min\{1,\zeta_0(1-2\zeta_1)-{2 \over q_2}\}\right)$.
\end{proposition}

\section{Proofs for Section 4}\label{Appendix: Sec4}

\subsection{Proof of Theorem 4.1}

In this section, we prove Steps 1 and 3 in the proof of Theorem 4.1. 

\noindent
(Step 1) Now we evaluate $S_n(\bm{0})$. By a change of variables and the dominated convergence theorem, we have 
\begin{align*}
E[S_n(\bm{0})] &= {A_n^{-1} \over h_1 \dots h_d}\int K_{Ah}(\bm{x})
H^{-1}\left(
\begin{array}{c}
1 \\
\check{(\bm{x}/A_n)}
\end{array}
\right)
(1\ \check{(\bm{x}/A_n)}')H^{-1}g(\bm{x}/A_n)d\bm{x}\\
&= {A_n^{-1} \over h_1 \dots h_d} A_nh_1\dots h_d \int K(\bm{w})
\left(
\begin{array}{c}
1 \\
\check{\bm{w}}
\end{array}
\right)
(1\ \check{\bm{w}}')g(\bm{w} \circ \bm{h})d\bm{w}\\
&= \left(g(\bm{0})\int K(\bm{w})
\left(
\begin{array}{c}
1 \\
\check{\bm{w}}
\end{array}
\right)
(1\ \check{\bm{w}}')d\bm{w}\right)(1 + o(1)). 
\end{align*}
For $1 \leq j_{1,1} \leq \dots \leq j_{1,L_1} \leq d$, $1 \leq j_{2,1}\leq \dots \leq j_{2,L_2} \leq d, 0 \leq L_1,L_2 \leq p$, we define
\begin{align*}
I_{n,j_{1,1}\dots j_{1,L_1},j_{2,1}\dots j_{2,L_2}} &:= {1 \over nh_1 \dots h_d}\sum_{i=1}^{n}K_{Ah}\left(\bm{X}_i \right)\prod_{\ell_1=1}^{L_1}\left({X_{i,j_{1,\ell_1}}\over A_{n,j_{1,\ell_1}}h_{j_{1,\ell_1}}}\right)\prod_{\ell_2=1}^{L_2}\left({X_{i,j_{2,\ell_2}} \over A_{n,j_{2,\ell_2}}h_{j_{2,\ell_2}}}\right). 
\end{align*}
Then, by a change of variables and the dominated convergence theorem, we have
\begin{align*}
&\Var(I_{n,j_{1,1}\dots j_{1,L_1},j_{2,1}\dots j_{2,L_2}})\\ 
&\quad = {1 \over n(h_1 \dots h_d)^2}\Var\left(K_{Ah}\left(\bm{X}_1\right)\prod_{\ell_1=1}^{L_1}\left({X_{i,j_{1,\ell_1}} \over A_{n,j_{1,\ell_1}}h_{j_{1,\ell_1}}}\right)\prod_{\ell_2=1}^{L_2}\left({X_{i,j_{2,\ell_2}} \over A_{n,j_{2,\ell_2}}h_{j_{2,\ell_2}}}\right)\right)\\
&\quad = {1 \over nh_1 \dots h_d}\left\{\int \prod_{\ell_1=1}^{L_1}z_{j_{1,\ell_1}}^2\prod_{\ell_2=1}^{L_2}z_{j_{2,\ell_2}} ^2K^2(\bm{z})g(\bm{z} \circ \bm{h})d\bm{z} \right. \\ 
&\left. \quad \quad - h_1 \dots h_d\left(\int \prod_{\ell_1=1}^{L_1}z_{j_{1,\ell_1}}\prod_{\ell_2=1}^{L_2}z_{j_{2,\ell_2}} K(\bm{z})g(\bm{z} \circ \bm{h})d\bm{z}\right)^2\right\}\\
&\quad = {1 \over nh_1 \dots h_d}\left(g(\bm{0})\kappa_{j_{1,1}\dots j_{1,L_1}j_{2,1}\dots j_{2,L_2}j_{1,1}\dots j_{1,L_1}j_{2,1}\dots j_{2,L_2}}^{(2)} + o(1)\right)\\
&\quad \quad   - {1 \over n}(g(\bm{0})\kappa_{j_{1,1}\dots j_{1,L_1}j_{2,1}\dots j_{2,L_2}}^{(1)} + o(1))^2\\
&\quad = {g(\bm{0})\kappa_{j_{1,1}\dots j_{1,L_1}j_{2,1}\dots j_{2,L_2}j_{1,1}\dots j_{1,L_1}j_{2,1}\dots j_{2,L_2}}^{(2)} \over nh_1 \dots h_d} + o\left({1 \over nh_1 \dots h_d}\right).
\end{align*}
Then for any $\rho>0$, 
\begin{align*}
&P\left(|I_{n,j_{1,1}\dots j_{1,L_1},j_{2,1}\dots j_{2,L_2}} - g(\bm{0})\kappa_{j_{1,1}\dots j_{1,L_1}j_{2,1}\dots j_{2,L_2}}^{(1)}|>\rho \right)\\ 
&\quad \leq \rho^{-1}\left\{\Var(I_{n,j_{1,1}\dots j_{1,L_1},j_{2,1}\dots j_{2,L_2}}) + \left(E[I_{n,j_{1,1}\dots j_{1,L_1},j_{2,1}\dots j_{2,L_2}}] - g(\bm{0})\kappa_{j_{1,1}\dots j_{1,L_1}j_{2,1}\dots j_{2,L_2}}^{(1)}\right)^2\right\}\\
&= O\left({1 \over nh_1 \dots h_d}\right) + o(1) = o(1).
\end{align*}
This yields $I_{n,j_{1,1}\dots j_{1,L_1},j_{2,1}\dots j_{2,L_2}} \stackrel{p}{\to} g(\bm{0})\kappa_{j_{1,1}\dots j_{1,L_1}j_{2,1}\dots j_{2,L_2}}^{(1)}$. Hence we have
\begin{align*}
S_{n}(\bm{0}) &\stackrel{p}{\to} g(\bm{0})S.
\end{align*}

\noindent
(Step 3) Now we evaluate $B_n(\bm{0})$. Decompose
\begin{align*}
B_{n,j_1\dots j_L}(\dot{\bm{X}})&= \left\{B_{n,j_1\dots j_L}(\dot{\bm{X}}) - B_{n,j_1\dots j_L}(\bm{0}) - E\left[B_{n,j_1\dots j_L}(\dot{\bm{X}}) - B_{n,j_1\dots j_L}(\bm{0})\right]\right\}\\
&\quad + E\left[B_{n,j_1\dots j_L}(\dot{\bm{X}}) - B_{n,j_1\dots j_L}(\bm{0})\right]\\
&\quad + \left\{B_{n,j_1\dots j_L}(\bm{0}) - E\left[B_{n,j_1\dots j_L}(\bm{0})\right]\right\}\\
&\quad + E\left[B_{n,j_1\dots j_L}(\bm{0})\right]\\
&=: \sum_{\ell=1}^{4}B_{n,j_1\dots j_L\ell}. 
\end{align*}
Define $N_{\bm{x}}(h):= \prod_{j=1}^{d}[x_j-h_j, x_j + h_j]$ and $\bm{x}=(x_1,\dots,x_d) \in (-1/2,1/2)^d$. For $B_{n,j_1\dots j_L1}$, by a change of variables and the dominated convergence theorem, we have
\begin{align}\label{B_n01}
&\Var(B_{n,j_1\dots j_L1})  \nonumber \\ 
&\leq {A_n \over \{(p+1)!\}^2nh_1 \dots h_d}E\left[K_{Ah}^2\left(\bm{X}_i \right)\prod_{\ell=1}^{L}\left({X_{i,j_\ell} \over A_{n,j_\ell}h_{j_\ell}}\right)^2 \right. \nonumber \\
&\left. \quad \times \sum_{1 \leq j_{1,1} \leq \dots \leq j_{1,p+1} \leq d, 1 \leq j_{2,1} \leq \dots \leq j_{2,p+1} \leq d}{1 \over \bm{s}_{j_{1,1} \dots j_{1,p+1}}! }{1 \over \bm{s}_{j_{2,1}\dots j_{2,p+1}}!} \right. \nonumber  \\
&\left. \quad \times (\partial_{j_{1,1}\dots j_{1,p+1}}m(\dot{\bm{X}}_i/A_n) \!-\! \partial_{j_{1,1}\dots j_{1,p+1}}m(\bm{0}))(\partial_{j_{2,1}\dots j_{2,p+1}}m(\dot{\bm{X}}_i/A_n) \!-\! \partial_{j_{2,1}\dots j_{2,p+1}}m(\bm{0})) \right.  \nonumber \\
&\left. \quad  \times \prod_{\ell_1=1}^{p+1}{X_{i,j_{1,\ell_1}} \over A_{n,j_{\ell_1}}}\prod_{\ell_2=1}^{p+1}{X_{i,j_{2,\ell_2}} \over A_{n,j_{\ell_2}}}\right]  \nonumber \\ 
&\leq {A_n \over \{(p+1)!\}^2n}\max_{1 \leq j_1 \leq \dots \leq j_{p+1} \leq d}\sup_{\bm{y} \in N_{\bm{0}}(h)}|\partial_{j_1\dots j_{p+1}}m(\bm{y}) - \partial_{j_1\dots j_{p+1}}m(\bm{0})|^2 \nonumber \\
&\quad \quad \times  \sum_{1 \leq j_{1,1} \leq \dots \leq j_{1,p+1} \leq d, 1 \leq j_{2,1} \leq \dots \leq j_{2,p+1} \leq d}\prod_{\ell_1=1}^{p+1}h_{j_{1,\ell_1}}\prod_{\ell_2=1}^{p+1}h_{j_{2,\ell_2}} \nonumber \\
&\quad \quad \times \int \left(\prod_{\ell=1}^{L}z_{j_\ell}^2\prod_{\ell_1=1}^{p+1}|z_{j_{1,\ell_1}}|\prod_{\ell_2=1}^{p+1}|z_{j_{2,\ell_2}}|\right)K^2(\bm{z})g(\bm{z} \circ \bm{h})d\bm{z} \nonumber \\
&= o\left({A_n \over n} \sum_{1 \leq j_{1,1} \leq \dots \leq j_{1,p+1} \leq d, 1 \leq j_{2,1} \leq \dots \leq j_{2,p+1} \leq d}\prod_{\ell_1=1}^{p+1}h_{j_{1,\ell_1}}\prod_{\ell_2=1}^{p+1}h_{j_{2,\ell_2}}\right) \nonumber \\
&= o(1).
\end{align}
Then we have $B_{n,j_1\dots j_L1} = o_p(1)$. 

For $B_{n,j_1\dots j_L2}$, 
\begin{align}\label{B_n02}
&|B_{n,j_1\dots j_L2}| \nonumber \\ 
&\leq {1 \over (p+1)!}\max_{1 \leq j_1,\dots,j_{p+1} \leq d}\sup_{\bm{y} \in N_{\bm{0}}(h)}|\partial_{j_1\dots j_{p+1}}m(\bm{y}) - \partial_{j_1\dots j_{p+1}}m(\bm{0})| \nonumber \\
&\quad \times \!\! \sqrt{\! A_nh_1 \dots h_d}\!\!\!\sum_{1 \leq j_{1,1} \leq \dots \leq j_{1,p+1} \leq d}\prod_{\ell_1=1}^{p+1}\!\!h_{j_{1,\ell_1}} \!\! \int \!\! \left(\prod_{\ell=1}^{L}|z_{j_\ell}|\prod_{\ell_1=1}^{p+1}|z_{j_{1,\ell_1}}| \! \right)\!\! |K(\bm{z})|g(\bm{z} \circ \bm{h})d\bm{z} \nonumber \\
&= o(1). 
\end{align}
For $B_{n,j_1\dots j_L3}$, 
\begin{align}\label{B_n03}
&\Var(B_{n,j_1\dots j_L3}) \nonumber \\ 
&\leq {A_nh_1 \dots h_d \over \{(p+1)!\}^2nh_1 \dots h_d} \sum_{1 \leq j_{1,1} \leq \dots \leq j_{1,p+1} \leq d, 1 \leq j_{2,1} \leq \dots \leq j_{2,p+1} \leq d}\!\!\!\!\!\!\!\!\!\!\!\!\!\!\!\!\partial_{j_{1,1}\dots j_{1,p+1}}m(\bm{0})\partial_{j_{2,1}\dots j_{2,p+1}}m(\bm{0}) \nonumber \\
&\quad \quad \times \prod_{\ell_1=1}^{p+1}h_{j_{1,\ell_1}}\prod_{\ell_2=1}^{p+1}h_{j_{2,\ell_2}} \int \left(\prod_{\ell=1}^{L}z_{j_\ell}^2\prod_{\ell_1}^{p+1}|z_{j_{1,\ell_1}}|\prod_{\ell_2=1}^{p+1}|z_{j_{2,\ell_2}}|\right)K^2(\bm{z})g(\bm{z} \circ \bm{h})d\bm{z} \nonumber  \\ 
&= O\left({A_n \over n}\sum_{1 \leq j_{1,1} \leq \dots \leq j_{1,p+1} \leq d, 1 \leq j_{2,1} \leq \dots \leq j_{2,p+1} \leq d}\prod_{\ell_1=1}^{p+1}h_{j_{1,\ell_1}}\prod_{\ell_2=1}^{p+1}h_{j_{2,\ell_2}}\right).
\end{align}
Then we have $B_{n,j_1\dots j_L3} = o_p(1)$. 

For $B_{n,j_1\dots j_L4}$, 
\begin{align}\label{B_n04}
B_{n,j_1\dots j_L4} &\!=\! \sqrt{A_nh_1 \dots h_d}\sum_{1 \leq j_{1,1} \leq \dots \leq j_{1,p+1} \leq d}{\partial_{j_{1,1}\dots j_{1,p+1}}m(\bm{0}) \over \bm{s}_{j_{1,1} \dots j_{1,p+1}}!}\nonumber \\&\quad \times \prod_{\ell_1=1}^{p+1}h_{j_{1,\ell_1}} \int \left(\prod_{\ell=1}^{L}z_{j_\ell} \prod_{\ell_1=1}^{p+1}z_{j_{1,\ell_1}}\right) K(\bm{z})g(\bm{z} \circ \bm{h})d\bm{z} \nonumber \\
&\!=\! g(\bm{0}) \sqrt{\!\! A_nh_1 \dots h_d}\!\!\!\!\!\!\!\!\!\!\sum_{1 \leq j_{1,1} \leq \dots \leq j_{1,p+1} \leq d}\!\!\!\!\!\!\!\!{\partial_{j_{1,1}\dots j_{1,p+1}}m(\bm{0}) \over \bm{s}_{j_{1,1} \dots j_{1,p+1}}!} \!\! \prod_{\ell_1=1}^{p+1} \!\! h_{j_{1,\ell_1}} \! \kappa_{j_1\dots j_Lj_{1,1}\dots j_{1,p+1}}^{(1)} \!\!\!+\! o(1). 
\end{align} 
Combining (\ref{B_n01})-(\ref{B_n04}), 
\begin{align*}
B_{n,j_1\dots j_L}(\dot{\bm{X}}) &= g(\bm{0})\sqrt{A_nh_1 \dots h_d}\sum_{1 \leq j_{1,1} \leq \dots \leq j_{1,p+1} \leq d}{\partial_{j_{1,1}\dots j_{1,p+1}}m(\bm{0}) \over \bm{s}_{j_{1,1} \dots j_{1,p+1}}!}\\
&\quad \times \prod_{\ell_1=1}^{p+1}h_{j_{1,\ell_1}} \kappa_{j_1\dots j_Lj_{1,1}\dots j_{1,p+1}}^{(1)} + o_p(1)\\
&= g(\bm{0})\sqrt{A_nh_1 \dots h_d}(B^{(d,p)}\bm{M}_n^{(d,p)}(\bm{0}))_{j_1\dots j_L} + o_p(1). 
\end{align*}

\section{Proofs for Section 5}\label{Appendix: Sec5}

In this section, we prove Theorem 5.1, Proposition 5.1, and Corollary 5.1. Before we prove Theorem 5.1, 
we consider general kernel estimators and derive their uniform convergence rates (Section \ref{sec:general-est-R1}). Since the estimators include many kernel-based estimators such as, kernel density, LC, LL, and LP estimators for random fields on $\mathbb{R}^d$ with irregularly spaced sampling sites, the results are of independent theoretical interest. As applications of the results, we derive uniform convergence rates of LP estimators (Section \ref{sec: Thm LPR-unif proof}). The proofs of Proposition 5.1 and Corollary 5.1 are given in Sections \ref{sec: Prp5.1-proof} and \ref{sec: Cor5.1-proof}, respectively. 

\subsection{Uniform convergence rates for general kernel estimators}\label{sec:general-est-R1}
For $j=1,2,3$, let $f_j: \mathbb{R}^d \to \mathbb{R}$ be functions such that $f_j$ is continuous on $R_{0,\delta} := (-1/2-\delta, 1/2+\delta)^d$ for some $\delta>0$. Define 
\begin{align}
\hat{\Psi}_{\mathrm{I}}(\bm{z}) &= {1 \over n^2A_n^{-1}h_1 \dots h_d }\sum_{i=1}^{n}K_{Ah}(\bm{X}_i-A_n \bm{z}) \nonumber \\
&\quad \times f_{1,Ah}\left(\bm{X}_i - A_n \bm{z}\right)f_{2,A}\left(\bm{X}_i - A_n \bm{z}\right)f_{3,A}\left(\bm{X}_i\right)Z_{\bm{X}_i}, \label{general-est-Psi1}\\
\hat{\Psi}_{\mathrm{II}}(\bm{z}) &= {1 \over nh_1 \dots h_d }\sum_{i=1}^{n}K_{Ah}(\bm{X}_i-A_n \bm{z}) \nonumber \\
&\quad \times f_{1,Ah}\left(\bm{X}_i - A_n \bm{z}\right)f_{2,A}\left(\bm{X}_i - A_n \bm{z}\right)f_{3,A}\left(\bm{X}_i\right), \label{general-est-Psi2}
\end{align}
where $f_{j,Aa}(\bm{x}) = f_j\left({x_1 \over A_{n,1}a_1},\dots, {x_d \over A_{n,d}a_d}\right)$ for $\bm{a} = (a_1,\dots a_d)' \in (0,\infty)^d$ and $\{Z_{\bm{X}_i}\}_{i=1}^{n}$ is a sequence of real-valued random variables. Many kernel estimators, such as kernel density, Nadaraya-Watson, and LP estimators, can be represented by combining special cases of estimators (\ref{general-est-Psi1}) or (\ref{general-est-Psi2}). In this study, we use the uniform convergence rates of these estimators with 
\begin{align*}
f_1 &\in \left\{e'_{j_1 \dots j_L}
\left(
\begin{array}{c}
1 \\
\check{\bm{x}}
\end{array}
\right), e'_{j_{1,1} \dots j_{1,L_1}}
\left(
\begin{array}{c}
1 \\
\check{\bm{x}}
\end{array}
\right)
(1\ \check{\bm{x}}')e_{j_{2,1} \dots j_{2,L_2}}
\right\},\\ 
f_2 &\in \left\{1, \prod_{\ell=1}^{L}x_{j_\ell}\right\},\ f_3 \in \left\{1,\eta, \sigma_{\varepsilon}, \{\partial_{j_1\dots j_{p+1}}m\}_{1\leq j_1 \leq \dots \leq j_{p+1} \leq d}\right\},\ Z_{\bm{X}_i} \in \left\{e(\bm{X}_i), \varepsilon_i \right\}.
\end{align*}

We assume the following conditions for the sampling sites $\{\bm{X}_i\}_{i=1}^{n}$: 
\begin{assumption}\label{Ass:sample-unif-gen}
Let $g$ be a probability density function with support $R_0=[-1/2,1/2]^d$. 
\begin{itemize}
\item[(i)] $A_n/n \to \kappa \in [0,\infty)$ as $n \to \infty$, 
\item[(ii)] $\{\bm{X}_{i}=(X_{i,1},\dots,X_{i,d})'\}_{i = 1}^{n}$ is a sequence of i.i.d. random vectors with density $A_n^{-d}g(\cdot/A_n)$ and $g$ is continuous and positive on $R_0$. 
\item[(iii)] $\{\bm{X}_i\}_{i=1}^{n}$ and $\{Z_{\bm{x}}: \bm{x} \in \mathbb{R}^d\}$ are independent. 
\end{itemize}
\end{assumption}

We also assume the following conditions on the bandwidth $h_j$, the random field $\{Z_{\bm{x}}: \bm{x} \in \mathbb{R}^d\}$, and functions $f_j$: 
\begin{assumption}\label{Ass:RF-unif-gen}
For $j = 1,\dots, d$, let $\{A_{n1,j}\}_{n \geq 1}$, $\{A_{n2,j}\}_{n \geq 1}$ be sequence of positive numbers.  
\begin{itemize}
\item[(i)] The random field $\{Z_{\bm{x}}: \bm{x} \in \mathbb{R}^d\}$ is stationary and $E[|Z_{\bm{0}}|^{q_2}]<\infty$ for some integer $q_2 > 4$. 
\item[(ii)] Define $\sigma_{\bm{Z}}(\bm{x}) = E[Z_{\bm{0}}Z_{\bm{x}}]$. Assume that  $\int_{\mathbb{R}^d} |\sigma_{\bm{Z}}(\bm{v})|d\bm{v}<\infty$.
\item[(iii)] $\min\left\{A_{n2,j},  {A_{n1,j} \over A_{n2,j}}, {A_{n,j}h_j \over A_{n1,j}}\right\} \to \infty$ as $n \to \infty$.
\item[(iv)] The random field $\{Z_{\bm{x}}: \bm{x} \in \mathbb{R}^d\}$ is $\beta$-mixing with mixing coefficients $\beta(a;b) \leq \beta_1(a)\varpi_2(b)$ such that as $n \to \infty$, $h_j \to 0$, $1 \leq j \leq d$, 
\begin{align}
&\sup_{\bm{v} \in R_{0,\delta}}\left|f_2(h_1v_1,\dots,h_dv_d) \over f_2(h_1, \dots, h_d)\right| \in (c_{f_2}, C_{f_2})\  \text{for some}\ 0<c_{f_2}<C_{f_2}<\infty, \label{unif-F1-gen}\\
&{A_n^{(1)} \over (\overline{A}_{n1})^d} \sim 1,\ {A_n^{{1 \over 2}}(h_1 \dots h_d)^{{1 \over 2}} \over n^{1/q_2}(\overline{A}_{n1})^d (\log n)^{{1 \over 2}+\iota}} \gtrsim 1\ \text{for some $\iota \in (0,\infty)$}, \label{unif-F2-gen}\\
&\sqrt{{n^2A_n h_1 \dots h_d \over (A_n^{(1)})^2 \log n}}\beta_1(\underline{A}_{n2})\varpi_2(A_n h_1 \dots h_d) \to 0, \label{unif-F3-gen}
\end{align}
where 
\begin{align*}
A_n^{(1)} &=\prod_{j=1}^{d}A_{n1,j},\ \overline{A}_{n1}=\max_{1 \leq j \leq d}A_{n1,j},\ \underline{A}_{n1}=\min_{1 \leq j \leq d}A_{n1,j},\\ 
\overline{A}_{n2} &=\max_{1 \leq j \leq d}A_{n2,j},\ \underline{A}_{n2}=\min_{1 \leq j \leq d}A_{n2,j}.
\end{align*} 
\item[(v)] $f_1 :\mathbb{R}^d \to \mathbb{R}$ is Lipschitz continuous on $\mathbb{R}^d$, i.e., $|f_1(\bm{v}_1) - f_1(\bm{v}_2)| \leq L_{f_1}|\bm{v}_1 - \bm{v}_2|$ for some $L_{f_1} \in(0,\infty)$ and all $\bm{v}_1,\bm{v}_2 \in \mathbb{R}^d$, and $f_2$ and $f_3$ are continuous on $R_{0,\delta}$. 
\end{itemize}
\end{assumption}

When $Z_{\bm{X}_i} = \varepsilon_i$, we interpret $\{Z_{\bm{x}}: \bm{x} \in \mathbb{R}^d\}$ as a set of i.i.d. random variables and in this case $\sigma_{\bm{Z}}(\bm{x}) = 0$ if $\bm{x} \neq 0$.

The next result provides uniform convergence rates of $\hat{\Psi}_{\mathrm{I}}$ and $\hat{\Psi}_{\mathrm{II}}$. 
\begin{proposition}\label{prp:LP-unif}
Suppose that Assumptions \ref{Ass:sample-unif-gen}, \ref{Ass:RF-unif-gen}, and 5.4 
hold. Then as $n \to \infty$, we have
\begin{align}
\sup_{\bm{z} \in [-1/2,1/2]^d}\left|\hat{\Psi}_{\mathrm{I}}(\bm{z}) - E[\hat{\Psi}_{\mathrm{I}}(\bm{z})]\right| &= O_p\left( \left|f_2(h_1,\dots, h_d)\right|\sqrt{{\log n \over n^{2}A_n^{-1} h_1 \dots h_d }}\right), \label{LP-unif-R1}\\
\sup_{\bm{z} \in [-1/2,1/2]^d}\left|\hat{\Psi}_{\mathrm{II}}(\bm{z}) - E[\hat{\Psi}_{\mathrm{II}}(\bm{z})]\right| &= O_p\left( \left|f_2(h_1,\dots, h_d)\right|\sqrt{{\log n \over n h_1 \dots h_d }}\right).\label{LP-unif-R2}
\end{align}
\end{proposition}

\begin{proof} 
We only provide the proof of (\ref{LP-unif-R1}) since the proof of (\ref{LP-unif-R2}) is almost the same. 
Let $a_n = \sqrt{{\log n \over n^2A_n^{-1}h_1 \dots h_d}}$ and $\tau_n = \rho_n n^{1/q_2}$ with $\rho_n = (\log n)^{\iota}$ for some $\iota>0$. Define 
\begin{align*}
\hat{\Psi}_1(\bm{z}) &= {|f_2^{-1}(h_1,\dots,h_d)| \over n^2A_n^{-1}h_1 \dots h_d }\sum_{i=1}^{n}K_{Ah}(\bm{X}_i-A_n \bm{z}) \\
&\quad \times f_{1,Ah}\left(\bm{X}_i - A_n \bm{z}\right)f_{2,A}\left(\bm{X}_i - A_n \bm{z}\right)f_{3,A}\left(\bm{X}_i\right)Z_{\bm{X}_i}1\{|Z_{\bm{X}_i}| \leq \tau_n\}, \\
\hat{\Psi}_2(\bm{z}) &= {|f_2^{-1}(h_1,\dots,h_d)| \over n^2A_n^{-1}h_1 \dots h_d }\sum_{i=1}^{n}K_{Ah}(\bm{X}_i-A_n \bm{z}) \\
&\quad \times f_{1,Ah}\left(\bm{X}_i - A_n \bm{z}\right)f_{2,A}\left(\bm{X}_i - A_n \bm{z}\right)f_{3,A}\left(\bm{X}_i\right)Z_{\bm{X}_i}1\{|Z_{\bm{X}_i}| > \tau_n\}.
\end{align*}
Note that 
\[
\hat{\Psi}(\bm{z}) - E[\hat{\Psi}(\bm{z})] = \hat{\Psi}_1(\bm{z}) - E[\hat{\Psi}_1(\bm{z})] + \hat{\Psi}_2(\bm{z}) - E[\hat{\Psi}_2(\bm{z})].
\]

(Step 1) First we consider the term $\hat{\Psi}_{2}(\bm{z}) - E[\hat{\Psi}_{2}(\bm{z})]$. Observe that 
\begin{align*}
P\left(\sup_{\bm{z} \in R_0}|\hat{\Psi}_2(\bm{z})|>a_{n}\right) &\leq P\left(|Z_{\bm{X}_{i}}| > \tau_{n}\ \text{for some $i=1,\hdots,n$}\right)\\
&\leq \tau_{n}^{-q_2}\sum_{i=1}^{n}E\left[E_{\cdot|\bm{X}}[|Z_{\bm{X}_{i}}|^{q_2}]\right] \leq n\tau_{n}^{-q_2} = \rho_{n}^{-q_2} \to 0.
\end{align*}
Further, for $\bm{z} \in [-1/2,1/2]^d$,
\begin{align*}
&E\left[\left|\hat{\Psi}_2(\bm{z})\right|\right]\\ 
&\leq {|f_2^{-1}(h_1,\dots,h_d)| \over n^2A_n^{-1}h_1 \dots h_d}\sum_{i=1}^{n}E\left[|K_{Ah}(\bm{X}_i-A_n \bm{z})| \right. \\
&\left. \quad \times \left|f_{1,Ah}\left(\bm{X}_i - A_n \bm{z}\right)f_{2,A}\left(\bm{X}_i - A_n \bm{z}\right)\right|f_{3,A}\left(\bm{X}_i\right)E_{\cdot \mid \bm{X}}[|Z_{\bm{X}_i}|1\{|Z_{\bm{X}_i}| > \tau_n\}]\right]\\
&\lesssim {nA_n^{-1}|f_2^{-1}(h_1,\dots,h_d)| \over n^2A_n^{-1}h_1 \dots h_d\tau_n^{q_2-1}}\int_{R_n} |K_{Ah}(\bm{x}-A_n \bm{z})|  \left|f_{1,Ah}\left(\bm{x} - A_n \bm{z}\right)f_{2,A}\left(\bm{X}_i - A_n \bm{z}\right)\right|\\
&\quad \times f_{3,A}\left(\bm{x}\right)g(\bm{x}/A_n)d\bm{x}\\
& = {|f_2^{-1}(h_1,\dots,h_d)| \over nA_n^{-1}\tau_n^{q_2-1}}\int_{\bm{h}^{-1}(R_0-\bm{z})} |K(\bm{v})|  \left|f_{1}\left(\bm{v}\right)f_{2}\left(\bm{v} \circ \bm{h}\right)\right|f_{3}\left(\bm{z} + \bm{v} \circ \bm{h}\right)g(\bm{z} + \bm{v} \circ \bm{h})d\bm{v}\\
&\lesssim {1 \over nA_n^{-1}\tau_n^{q_2-1}} \lesssim {1 \over \tau_n^{q_2-1}} \lesssim a_n. 
\end{align*}
Then we have
\[
\sup_{\bm{z} \in R_0}\left|\hat{\Psi}(\bm{z}) - E[\hat{\Psi}(\bm{z})]\right| = O_p(a_n). 
\]

(Step 2) Now we consider the term $\hat{\Psi}_{1}(\bm{z}) - E[\hat{\Psi}_{1}(\bm{z})]$.

Define 
\begin{align*}
\Psi_{1,\bm{X}_i}(\bm{z}) &\!=\! K_{Ah}(\bm{X}_i\!-\!A_n \bm{z})\!f_{1,Ah}\!\left(\bm{X}_i \!-\! A_n \bm{z}\right)\!f_{2,A}\!\left(\bm{X}_i \!-\! A_n \bm{z}\right)\!f_{3,A}\!\left(\bm{X}_i\right)\!Z_{\bm{X}_i}\!1\{|Z_{\bm{X}_i}| \!\leq\! \tau_n\}\\
& - E\!\left[K_{Ah}(\bm{X}_i\!-\!A_n \bm{z})\!f_{1,Ah}\!\left(\bm{X}_i \!- \!A_n \bm{z}\right)\!f_{2,A}\!\left(\bm{X}_i \!-\! A_n \bm{z}\right)\!f_{3,A}\!\left(\bm{X}_i\right)\!Z_{\bm{X}_i}\!1\{|Z_{\bm{X}_i}| \!\leq \!\tau_n\}\right]. 
\end{align*}
Observe that 
\begin{align*}
\sum_{i=1}^{n}\Psi_{1,\bm{X}_i}(\bm{z})& \!=\! \sum_{\bm{\ell} \in L_{n1}(\bm{z})}\!\!\!\!\Psi_1^{(\bm{\ell};\bm{\Delta}_0)}(\bm{z}) \!+\! \sum_{\bm{\Delta} \neq \bm{\Delta}_0}\sum_{\bm{\ell} \in L_{n1}(\bm{z})}\!\!\!\!\Psi_1^{(\bm{\ell};\bm{\Delta})}(\bm{z}) \!+ \!\sum_{\bm{\Delta} \in \{1,2\}^d}\sum_{\bm{\ell} \in L_{n2}(\bm{z})}\!\!\!\!\Psi_1^{(\bm{\ell};\bm{\Delta})}(\bm{z}),
\end{align*}
where 
\[
\Psi_1^{(\bm{\ell};\bm{\Delta})}(\bm{z}) = \sum_{i=1}^{n}\Psi_{1,\bm{X}_i}(\bm{z})1\{\bm{X}_i \in \Gamma_{n,\bm{z}}(\bm{\ell};\bm{\Delta}) \cap R_n \cap (\bm{h}R_n + A_n \bm{z})\}. 
\]
For $\bm{\Delta} \in \{1,2\}^d$, let $\{\tilde{\Psi}_1^{(\bm{\ell};\bm{\Delta})}(\bm{z})\}_{\bm{\ell} \in L_{n1}(\bm{z}) \cup L_{n2}(\bm{z})}$ be independent random variables such that $\Psi_1^{(\bm{\ell};\bm{\Delta})}(\bm{z}) \stackrel{d}{=} \tilde{\Psi}_1^{(\bm{\ell};\bm{\Delta})}(\bm{z})$. Applying Lemma \ref{indep_lemma} below with $M_h = 1$, $m \sim \left({A_n h_1 \dots h_d \over A_n^{(1)}}\right)$ and $\tau \sim \beta(\underline{A}_{n2}; A_n h_1 \dots h_d)$, we have that for $\bm{\Delta} \in \{1,2\}^d$, 
\begin{align}
&\sup_{t >0}\left|P\left(\left|\sum_{\bm{\ell} \in L_{n1}(\bm{z})}\Psi_1^{(\bm{\ell};\bm{\Delta})}(\bm{z})\right|>t\right) - P\left(\left|\sum_{\bm{\ell} \in L_{n1}(\bm{z})}\tilde{\Psi}_1^{(\bm{\ell};\bm{\Delta})}(\bm{z})\right|>t\right)\right| \nonumber \\ 
&\quad \lesssim \left({A_n h_1 \dots h_d \over A_n^{(1)}}\right)\beta(\underline{A}_{n2}; A_n h_1 \dots h_d), \label{Psi1-beta-block1}\\
&\sup_{t >0}\left|P\left(\left|\sum_{\bm{\ell} \in L_{n2}(\bm{z})}\Psi_1^{(\bm{\ell};\bm{\Delta})}(\bm{z})\right|>t\right) - P\left(\left|\sum_{\bm{\ell} \in L_{n2}(\bm{z})}\tilde{\Psi}_1^{(\bm{\ell};\bm{\Delta})}(\bm{z})\right|>t\right)\right|\nonumber \\ 
&\quad \lesssim \left({A_n h_1 \dots h_d \over A_n^{(1)}}\right)\beta(\underline{A}_{n2}; A_n h_1 \dots h_d) \label{Psi1-beta-block2}.
\end{align}
Since $\left({A_n h_1 \dots h_d \over A_n^{(1)}}\right)\beta(\underline{A}_{n2}; A_n h_1 \dots h_d) \to 0$ as $n \to \infty$, these results imply that
\begin{align*}
\sum_{\bm{\ell} \in L_{n1}(\bm{z})}\Psi_1^{(\bm{\ell};\bm{\Delta})}(\bm{z}) &= O_p\left(\sum_{\bm{\ell} \in L_{n1}(\bm{z})}\tilde{\Psi}_1^{(\bm{\ell};\bm{\Delta})}(\bm{z})\right),\\
\sum_{\bm{\ell} \in L_{n2}(\bm{z})}\Psi_1^{(\bm{\ell};\bm{\Delta})}(\bm{z}) &= O_p\left(\sum_{\bm{\ell} \in L_{n2}(\bm{z})}\tilde{\Psi}_1^{(\bm{\ell};\bm{\Delta})}(\bm{z})\right). 
\end{align*}
Now we show $\sup_{\bm{z} \in R_0}\left|\hat{\Psi}_{1}(\bm{z}) - E[\hat{\Psi}_1(\bm{z})]\right| = O_p\left(a_{n}\right)$. Cover the region $R_0$ with $N \leq (h_1 \dots h_d)^{-1}a_{n}^{-d}$ balls $B_k = \{\bm{z} \in \mathbb{R}^{d}: |z_j - z_{k,j}| \leq a_{n}h_j\}$ and use $\bm{z}_k = (z_{k,1},\dots, z_{k,d})$ to denote the mid point of $B_k$, $k=1,\dots,N$. In addition, let $K^{\ast}(\bm{v}) = C^{\ast}\prod_{j=1}^{d}I(|v_{j}| \leq 2C_{K})$ for $\bm{v} \in \mathbb{R}^d$ and sufficiently large $C^{\ast}>0$. Note that for $\bm{z} \in B_k$ and sufficiently large $n$,
\begin{align*}
&\left|K_{Ah}\left(\bm{X}_i - A_n \bm{z}\right)f_{1,Ah}(\bm{X}_i - A_n \bm{z})   - K_{Ah}\left(\bm{X}_i - A_n \bm{z}_k\right)f_{1,Ah}(\bm{X}_i - A_n \bm{z}_k)\right|\\
&\quad \leq a_{n}K_{Ah}^{\ast}\left(\bm{X}_i - A_n \bm{z}_k\right).
\end{align*} 

For $\bm{\ell} \in L_{n1}(\bm{z}) \cup L_{n2}(\bm{z})$ and $\bm{\Delta} \in \{1,2\}^d$, define
\begin{align*}
\Psi_2^{(\bm{\ell};\bm{\Delta})}(\bm{z}) = \sum_{i=1}^{n}\Psi_{2,\bm{X}_i}(\bm{z})1\{\bm{X}_i \in \Gamma_{n,\bm{z}}(\bm{\ell};\bm{\Delta}) \cap R_n \cap (\bm{h}R_n + A_n \bm{z})\},
\end{align*}
where 
\begin{align*}
\Psi_{2,\bm{X}_i}(\bm{z}) &= K_{Ah}^{\ast}\left(\bm{X}_i - A_n \bm{z}_n\right)f_{2,A}\left(\bm{X}_i - A_n \bm{z}\right)f_{3,A}\left(\bm{X}_i\right)Z_{\bm{X}_i}1\{|Z_{\bm{X}_i}| \leq \tau_n\}\\
&\quad - E\left[K_{Ah}^{\ast}\left(\bm{X}_i - A_n \bm{z}_n\right)f_{2,A}\left(\bm{X}_i - A_n \bm{z}\right)f_{3,A}\left(\bm{X}_i\right)Z_{\bm{X}_i}1\{|Z_{\bm{X}_i}| \leq \tau_n\}\right]. 
\end{align*}
Moreover, define
\begin{align*}
\bar{\Psi}_1(\bm{z}) &\!=\! {|f_2^{-1}(h_1,\dots,h_d)| \over n^2A_n^{-1}h_1 \dots h_d }\!\sum_{i=1}^{n}\!K_{Ah}^{\ast}(\bm{X}_i\!-\!A_n \bm{z})\!f_{2,A}\!\left(\bm{X}_i \!-\! A_n \bm{z}\right)\!f_{3,A}\!\left(\bm{X}_i\right)\!Z_{\bm{X}_i}\!1\{|Z_{\bm{X}_i}| \leq \tau_n\}. 
\end{align*}
Observe that for $\bm{z} \in R_{0}$,
\begin{align*}
E\left[|\bar{\Psi}_1(\bm{z})|\right] &\lesssim {A_n^{-1}|f_2^{-1}(h_1,\dots,h_d)| \over nA_n^{-1}h_1 \dots h_d}\!\!\!\int_{R_n}\!\!\!\! |K_{Ah}^{\ast}(\bm{x} \!-\! A_n \bm{z})f_{2,A}(\bm{x} \!-\! A_n \bm{z})f_{3,A}(\bm{x})|g(\bm{x}/A_n)d\bm{x}\\
&= {|f_2^{-1}(h_1,\dots,h_d)| \over nA_n^{-1}}\!\!\!\int_{\bm{h}^{-1}(R_0 - \bm{z})}\!\!\!\!\!\!\!\!\!\!|K^{\ast}(\bm{v})||f_2(\bm{v} \circ \bm{h})||f_3(\bm{z} + \bm{v}\circ \bm{h})g(\bm{z} + \bm{v}\circ \bm{h})|d\bm{v}\\
&\lesssim {1 \over nA_n^{-1}} \leq M.
\end{align*}
for sufficiently large $M>0$.
Then we have
\begin{align*}
&\sup_{\bm{z} \in B_k}\left|\hat{\Psi}_{1}(\bm{z}) - E[\hat{\Psi}_{1}(\bm{z})]\right|\\
&\leq \left|\hat{\Psi}_{1}(\bm{z}_k) - E[\hat{\Psi}_{1}(\bm{z}_k)]\right| + a_{n}\left(\left|\bar{\Psi}_{1}(\bm{z}_k)\right| + E\left[\left|\bar{\Psi}_{1}(\bm{z}_k)\right|\right]\right)\\
&\leq \left|\hat{\Psi}_{1}(\bm{z}_k) - E[\hat{\Psi}_{1}(\bm{z}_k)]\right| + \left|\bar{\Psi}_{1}(\bm{z}_k) - E[\bar{\Psi}_{1}(\bm{z}_k)]\right| + 2Ma_{n}\\
&\leq \!{|f_2^{-1}\!(h_1,\dots,h_d)\!| \over n^2\!A_n^{-1}\!h_1 \dots h_d } \!\!\!\left(\left|\sum_{\bm{\ell} \in L_{n1}\!(\bm{z}_k)} \!\!\!\!\!\!\!\!\! \Psi_1^{(\bm{\ell};\bm{\Delta}_0)}\!(\! \bm{z}_k \!)\right| \!\!+\!\!\! \sum_{\bm{\Delta} \neq \bm{\Delta}_{0}} \! \left|\sum_{\bm{\ell} \in L_{n1}\!(\bm{z}_k)} \!\!\!\!\!\!\!\!\! \Psi_1^{(\bm{\ell};\bm{\Delta})}\!(\! \bm{z}_k \!)\right| \!\!+\!\!\! \sum_{\bm{\Delta} \in \{1,2\}^d}\!\left|\sum_{\bm{\ell} \in L_{n2}\!(\bm{z}_k)} \!\!\!\!\!\!\!\!\! \Psi_1^{(\bm{\ell};\bm{\Delta})}\!(\! \bm{z}_k \!)\right|\right) \\
&+\! {|f_2^{-1}\!(h_1,\dots,h_d)\!| \over n^2\!A_n^{-1}\!h_1 \dots h_d } \!\!\!\left(\left|\sum_{\bm{\ell} \in L_{n1}\!(\bm{z}_k)} \!\!\!\!\!\!\!\!\! \Psi_2^{(\bm{\ell};\bm{\Delta}_0)}\!(\! \bm{z}_k \!)\right| \!\!+\!\!\! \sum_{\bm{\Delta} \neq \bm{\Delta}_{0}}\!\left|\sum_{\bm{\ell} \in L_{n1}\!(\bm{z}_k)} \!\!\!\!\!\!\!\!\! \Psi_2^{(\bm{\ell};\bm{\Delta})}\!(\! \bm{z}_k \!)\right| \!\!+\!\!\! \sum_{\bm{\Delta} \in \{1,2\}^d}\!\left|\sum_{\bm{\ell} \in L_{n2}\!(\bm{z}_k)} \!\!\!\!\!\!\!\!\! \Psi_2^{(\bm{\ell};\bm{\Delta})}\!(\! \bm{z}_k \!)\right|\right) \\
&+ 2Ma_{n}. 
\end{align*}
For $\bm{\Delta} \in \{1,2\}^d$, let $\{\tilde{\Psi}_2^{(\bm{\ell};\bm{\Delta})}(\bm{z})\}_{\bm{\ell} \in L_{n1}(\bm{z}) \cup L_{n2}(\bm{z})}$ be independent random variables such that $\Psi_2^{(\bm{\ell};\bm{\Delta})}(\bm{z}) \stackrel{d}{=} \tilde{\Psi}_2^{(\bm{\ell};\bm{\Delta})}(\bm{z})$. From (\ref{Psi1-beta-block1}) and (\ref{Psi1-beta-block2}), and applying Lemma \ref{indep_lemma} below to $\{\tilde{\Psi}_2^{(\bm{\ell};\bm{\Delta})}(\bm{z})\}_{\bm{\ell} \in L_{n1}(\bm{z}) \cup L_{n2}(\bm{z})}$, we have
\begin{align*}
&P\left(\sup_{\bm{z} \in R_0}\left|\hat{\Psi}_{1}(\bm{z}) - E[\hat{\Psi}_{1}(\bm{z})]\right|>2^{d+3}Ma_{n}\right)\\
&\leq N\max_{1 \leq k \leq N}P\left(\sup_{\bm{z} \in B_{k}}\left|\hat{\Psi}_{1}(\bm{z}) - E[\hat{\Psi}_{1}(\bm{z})]\right|>2^{d+3}Ma_{n}\right)\\
&\leq \sum_{\bm{\Delta} \in \{1,2\}^{d}}\hat{Q}_{n1}(\bm{\Delta}) + \sum_{\bm{\Delta} \in \{1,2\}^{d}}\hat{Q}_{n2}(\bm{\Delta}) + \sum_{\bm{\Delta}\in \{1,2\}^{d}}\bar{Q}_{n1}(\bm{\Delta}) + \sum_{\bm{\Delta}\in \{1,2\}^{d}}\bar{Q}_{n2}(\bm{\Delta})\\
&\quad + 2^{d+2}N\left({A_{n}h_1 \dots h_d \over A_n^{(1)}}\right)\beta(\underline{A}_{n2};A_{n}h_1 \dots h_d),
\end{align*}
where 
\begin{align*}
\hat{Q}_{nj}(\bm{\Delta}) &= N\max_{1 \leq k \leq N}P\left(\left|\sum_{\bm{\ell} \in L_{nj}(\bm{z}_k)}\tilde{\Psi}_1^{(\bm{\ell};\bm{\Delta})}(\bm{z}_k)\right|>Ma_{n}{n^2A_n^{-1}h_1 \dots h_d \over |f_2^{-1}(h_1,\dots,h_d)|}\right),\ j=1,2,\\
\bar{Q}_{nj}(\bm{\Delta}) &= N\max_{1 \leq k \leq N}P\left(\left|\sum_{\bm{\ell} \in L_{nj}(\bm{z}_k)}\tilde{\Psi}_2^{(\bm{\ell};\bm{\Delta})}(\bm{z}_k)\right|>Ma_{n}{n^2A_n^{-1}h_1 \dots h_d \over|f_2^{-1}(h_1,\dots,h_d)|}\right),\ j=1,2.
\end{align*}
Now we restrict our attention to $\hat{Q}_{n1}(\bm{\Delta})$, $\bm{\Delta} \neq \bm{\Delta}_{0}$. The proofs for other cases are similar. Note that
\begin{align*}
&P\left(\left|\sum_{\bm{\ell} \in L_{n1}(\bm{z}_k)}\tilde{\Psi}_1^{(\bm{\ell};\bm{\Delta})}(\bm{z}_k)\right|>Ma_{n}{n^2A_n^{-1}h_1 \dots h_d \over |f_2^{-1}(h_1,\dots,h_d)|}\right)\\ 
&\leq 2P\left(\sum_{\bm{\ell} \in L_{n1}(\bm{z}_k)}\tilde{\Psi}_1^{(\bm{\ell};\bm{\Delta})}(\bm{z}_k)>Ma_{n}{n^2A_n^{-1}h_1 \dots h_d \over |f_2^{-1}(h_1,\dots,h_d)|}\right).
\end{align*}
Observe that $\tilde{\Psi}_1^{(\bm{\ell};\bm{\Delta})}(\bm{z}_k)$ are zero-mean independent random variables and  
\begin{align}
\left|\tilde{\Psi}_1^{(\bm{\ell};\bm{\Delta})}(\bm{z}_k)\right| &\leq C_{\tilde{\Psi}_1}(\overline{A}_{n1})^{d-1}\overline{A}_{n2}nA_n^{-1}|f_2(h_1,\dots,h_d)|\tau_{n},\ a.s.\ (\text{from Lemma \ref{n summands}})\nonumber \\
E\left[\left(\tilde{\Psi}_1^{(\bm{\ell};\bm{\Delta})}(\bm{z}_k)\right)^{2}\right] &\leq C_{\tilde{\Psi}_1}(\overline{A}_{n1})^{d-1}\overline{A}_{n2}n^2A_n^{-2}f_2^2(h_1,\dots,h_d), \label{bound-L2}
\end{align}
for some $C_{\tilde{\Psi}_1}>0$, where (\ref{bound-L2}) can be shown by applying the same argument in (Step 2-1) in the proof of Theorem 4.1. 
Then Lemma \ref{Bernstein} yields that 
\begin{align*}
P\left(\sum_{\bm{\ell} \in L_{n1}(\bm{z}_k)}\tilde{\Psi}_1^{(\bm{\ell};\bm{\Delta})}(\bm{z}_k)>Ma_{n}{n^2A_n^{-1}h_1 \dots h_d \over |f_2^{-1}(h_1,\dots,h_d)|}\right) &\leq \exp \left(-{{M^2n^2A_n^{-1}h_1 \dots h_d \log n \over 2|f_2^{-1}(h_1,\dots,h_d)|^2} \over E_{n1} + E_{n2} }\right),
\end{align*}
where 
\begin{align*}
E_{n1} &=  C_{\tilde{\Psi}_1}\!\!\left({A_{n}h_1 \dots h_d \over A_n^{(1)}}\right)\!\!(\overline{A}_{n1})^{d-1}\overline{A}_{n2}n^{2}A_n^{-2}f_2^2(h_1,\dots,h_d),\\
E_{n2} &= {MC_{\tilde{\Psi}_1}n^2A_n^{-3/2}(h_1\dots h_d)^{1/2}(\log n)^{1/2}(\overline{A}_{n1})^{d-1}\overline{A}_{n2}\tau_{n} \over 3|f_2^{-1}(h_1,\dots,h_d)|^2}.
\end{align*}
Since
\begin{align*}
{M^2n^2A_n^{-1}h_1 \dots h_d \log n \over 2|f_2^{-1}(h_1,\dots,h_d)|^2E_{n1}} &= {M^2 \over 2C_{\tilde{\Psi}_1}}\left({A_n^{(1)} \over (\overline{A}_{n1})^{d-1}\overline{A}_{n2}}\right)\log n,\\
{M^2n^2A_n^{-1}h_1 \dots h_d \log n \over 2|f_2^{-1}(h_1,\dots,h_d)|^2E_{n2}} &= {3M \over 2C_{\tilde{\Psi}_1}}{A_n^{1/2}(h_1 \dots h_d)^{1/2} \over n^{1/q_2}(\overline{A}_{n1})^{d-1}\overline{A}_{n2}(\log n)^{-1/2 + \iota}},
\end{align*}
by taking $M>0$ sufficiently large, we obtain the desired result. 
\end{proof}

\subsection{Proof of Theorem 5.1}\label{sec: Thm LPR-unif proof}

\begin{proof}
Define 
\begin{align*}
S_n(\bm{z}) &= {1 \over nh_1 \dots h_d}\sum_{i=1}^{n}K_{Ah}\left(\bm{X}_i-A_n \bm{z}\right)H^{-1}
\left(
\begin{array}{c}
1 \\
\check{(\bm{X}_i-A_n\bm{z})}
\end{array}
\right)
(1\ \check{(\bm{X}_i-A_n\bm{z})'})H^{-1},\\
V_n(\bm{z}) &= {1 \over nh_1 \dots h_d}\sum_{i=1}^{n}K_{Ah}\left(\bm{X}_i-A_n \bm{z}\right)H^{-1}
\left(
\begin{array}{c}
1 \\
\check{(\bm{X}_i-A_n\bm{z})}
\end{array}
\right)(e_{n,i}+\varepsilon_{n,i}),\\
B_n(\bm{z}) &= {1 \over nh_1 \dots h_d}\sum_{i=1}^{n}K_{Ah}\left(\bm{X}_i-A_n \bm{z}\right)H^{-1}
\left(
\begin{array}{c}
1 \\
\check{(\bm{X}_i-A_n\bm{z})}
\end{array}
\right) \\
&\quad \times \sum_{1 \leq j_1 \leq \dots \leq j_{p+1}\leq d}{1 \over \bm{s}_{j_1\dots j_{p+1}}!}\partial_{j_1,\dots,j_{p+1}}m(\dot{\bm{X}}_i/A_n)\prod_{\ell=1}^{p+1}\left({X_{i,j_\ell} \over A_{n,j_\ell}}-z_{j_\ell}\right).
\end{align*}
Note that 
\begin{align*}
H(\hat{\bm{\beta}}(\bm{z}) - \bm{M}(\bm{z})) &= S_n^{-1}(\bm{z})(V_n(\bm{z}) + B_n(\bm{z})).
\end{align*}
Applying Proposition \ref{prp:LP-unif} (\ref{LP-unif-R2}) to $e'_{j_{1,1} \dots j_{1,L_1}}S_n(\bm{z})e_{j_{2,1} \dots j_{2,L_2}}$ with
\[
f_1(\bm{x}) = e'_{j_{1,1} \dots j_{1,L_1}}
\left(
\begin{array}{c}
1 \\
\check{\bm{x}}
\end{array}
\right)
(1\ \check{\bm{x}}')e_{j_{2,1} \dots j_{2,L_2}},\ f_2(\bm{x}) = 1,\ f_3(\bm{x})=1,
\]
we have that  
\begin{align}
&\sup_{\bm{z} \in \mathrm{T}_n}|e'_{j_{1,1} \dots j_{1,L_1}}(S_n(\bm{z})-g(\bm{z})S)e_{j_{2,1} \dots j_{2,L_2}}| \nonumber \\
&\leq \sup_{\bm{z} \in \mathrm{T}_n}|e'_{j_{1,1} \dots j_{1,L_1}}(S_n(\bm{z})\!-\!E[S_n(\bm{z})])e_{j_{2,1} \dots j_{2,L_2}}| \nonumber \\
&\quad + \sup_{\bm{z} \in \mathrm{T}_n}|e'_{j_{1,1} \dots j_{1,L_1}}(E[S_n(\bm{z})]\!-\!g(\bm{z})S)e_{j_{2,1} \dots j_{2,L_2}}| \nonumber \\
&= O_p\left(\sqrt{\log n \over nh_1 \dots h_d}\right) + o(1) = o_p(1). \label{unif-Sn-rate}
\end{align}

Applying Proposition \ref{prp:LP-unif} (\ref{LP-unif-R1}) to $A_n n^{-1}e'_{j_1 \dots j_L}V_n(\bm{z})$ with 
\[
f_1(\bm{x}) = e'_{j_1 \dots j_L}
\left(
\begin{array}{c}
1 \\
\check{\bm{x}}
\end{array}
\right)
,\ f_2(\bm{x}) = 1, (f_3(\bm{x}),Z_{\bm{X}_i}) \in \left\{(\eta(\bm{x}), e(\bm{X}_i)), (\sigma_{\varepsilon}(\bm{x}), \varepsilon_i)\right\},
\]
we have that 
\begin{align}
{n \over A_n}\sup_{\bm{z} \in \mathrm{T}_n}\!\left|{A_n \over n}e'_{j_1 \dots j_L}\!(V_n(\bm{z})\!-\!E[V_n(\bm{z})])\right| &\leq \! {n \over A_n}\sup_{\bm{z} \in \mathrm{T}_n}\!\left|{A_n \over n}e'_{j_1 \dots j_L}\!V_n(\bm{z})\right| \!= O_p\!\left(\!\!\sqrt{ \log n \over A_n h_1 \dots h_d}\right). \label{unif-Vn-rate}
\end{align}

Applying Proposition \ref{prp:LP-unif} (\ref{LP-unif-R2}) to $e'_{j_1 \dots j_L}B_n(\bm{z})$ with 
\begin{align*}
f_1(\bm{x}) &= e'_{j_1 \dots j_L}
\left(
\begin{array}{c}
1 \\
\check{\bm{x}}
\end{array}
\right)
,\ f_2(\bm{x}) = \prod_{\ell=1}^{L}x_{j_\ell},\ f_3(\bm{x}) = \sum_{1 \leq j_1 \leq \dots \leq j_{p+1}\leq d}{1 \over \bm{s}_{j_1\dots j_{p+1}}!}\partial_{j_1,\dots,j_{p+1}}m(\bm{x}),
\end{align*}
we have that
\begin{align}
\sup_{\bm{z} \in \mathrm{T}_n}\left|e'_{j_1 \dots j_L}B_n(\bm{z})\right| &\leq \sup_{\bm{z} \in \mathrm{T}_n}\left|e'_{j_1 \dots j_L}(B_n(\bm{z})- E[B_n(\bm{z})])\right| + \sup_{\bm{z} \in \mathrm{T}_n}\left|e'_{j_1 \dots j_L}E[B_n(\bm{z})]\right|\nonumber \\
&= O_p\left(\prod_{\ell=1}^{L}h_{j_\ell}\sqrt{\log n \over nh_1 \dots h_d}\right) + O\left(\sum_{1 \leq j_1 \leq \dots \leq j_{p+1}\leq d} \prod_{\ell=1}^{p+1}h_{j_\ell}\right) \label{unif-Bn-rate}
\end{align}
Combining (\ref{unif-Sn-rate})-(\ref{unif-Bn-rate}), we have that
\begin{align*}
&\sup_{\bm{z} \in \mathrm{T}_n}\left|\partial_{j_1 \dots j_L}\hat{m}(\bm{z}) - \partial_{j_1 \dots j_L}m(\bm{z})\right|\\ 
&\leq \left(\prod_{\ell=1}^{L}h_{j_\ell}\right)^{-1}{1 \over \inf_{\bm{z} \in R_0}g(\bm{z})}\sup_{\bm{z} \in \mathrm{T}_n}\left|e'_{j_1 \dots j_L}S^{-1}(V_n(\bm{z}) + B_n(\bm{z}))\right|\\
&\quad + \left(\prod_{\ell=1}^{L}h_{j_\ell}\right)^{-1}\sup_{\bm{z} \in \mathrm{T}_n}\left|e'_{j_1 \dots j_L}(S_n^{-1}(\bm{z}) - g^{-1}(\bm{z})S^{-1})(V_n(\bm{z}) + B_n(\bm{z}))\right|\\
&\lesssim \left(\prod_{\ell=1}^{L}h_{j_\ell}\right)^{-1}\left(\max_{1 \leq j_1\leq \dots \leq j_L \leq d, 0\leq L\leq p}\sup_{\bm{z} \in \mathrm{T}_n}|e'_{j_1 \dots j_L}V_n(\bm{z})| \right. \\
&\left. \quad+ \max_{1 \leq j_1\leq \dots \leq j_L \leq d, 0\leq L\leq p}\sup_{\bm{z} \in \mathrm{T}_n}|e'_{j_1 \dots j_L}B_n(\bm{z})|\right)\\
&= O_p\left( {\sum_{1 \leq j_1 \leq \dots \leq j_{p+1}\leq d}\prod_{\ell=1}^{p+1}h_{j_\ell} \over \prod_{\ell=1}^{L}h_{j_\ell}} + \sqrt{{\log n \over A_n h_1 \dots h_d \left(\prod_{\ell=1}^{L}h_{j_\ell}\right)^2 }}\right).
\end{align*}
\end{proof}

\subsection{Proof of Proposition 5.1}\label{sec: Prp5.1-proof}
\begin{proof}

It is easy to see that $\hat{g}(\bm{0}) \stackrel{p}{\to} g(\bm{0})$ as $n \to \infty$. For $\hat{W}_{n,1}(\bm{0})$, applying Theorem 5.1, 
we have 
\begin{align*}
&\hat{W}_{n,1}(\bm{0})\\ 
&= {A_n \over n^2h_1 \dots h_d}\sum_{i,j=1}^{n}K_{Ah}(\bm{X}_i)K_{Ah}(\bm{X}_j)\bar{K}_b(\bm{X}_i - \bm{X}_j)(e_{n,i} + \varepsilon_{n,i})(e_{n,j} + \varepsilon_{n,j}) + o_p(1)\\
&= {A_n \over n^2h_1 \dots h_d}\sum_{i,j=1}^{n}K_{Ah}(\bm{X}_i)K_{Ah}(\bm{X}_j)\bar{K}_b(\bm{X}_i - \bm{X}_j)e_{n,i}e_{n,j}\\
&\quad + {2A_n \over n^2h_1 \dots h_d}\sum_{i,j=1}^{n}K_{Ah}(\bm{X}_i)K_{Ah}(\bm{X}_j)\bar{K}_b(\bm{X}_i - \bm{X}_j)e_{n,i}\varepsilon_{n,j}\\
&\quad + {A_n \over n^2h_1 \dots h_d}\sum_{i,j=1}^{n}K_{Ah}(\bm{X}_i)K_{Ah}(\bm{X}_j)\bar{K}_b(\bm{X}_i - \bm{X}_j)\varepsilon_{n,i}\varepsilon_{n,j} + o_p(1)\\
&=: W_{n,1} + W_{n,2} + W_{n,3} + o_p(1).
\end{align*}

For $W_{n,3}$, observe that 
\begin{align*}
W_{n,3} &= {A_n \over n^2h_1 \dots h_d}\sum_{i=1}^{n}K_{Ah}^2(\bm{X}_i)\varepsilon_{n,i}^2\\ 
&\quad + {A_n \over n^2h_1 \dots h_d}\sum_{i \neq j}^{n}K_{Ah}(\bm{X}_i)K_{Ah}(\bm{X}_j)\bar{K}_b(\bm{X}_i - \bm{X}_j)\varepsilon_{n,i}\varepsilon_{n,j}\\
&=: W_{n,31} + W_{n,32}. 
\end{align*}

For $W_{n,31}$, we have 
\begin{align*}
E[W_{n,31}] &= {A_n \over nh_1 \dots h_d}\int K_{Ah}^2(\bm{x})\sigma_{\varepsilon}^2(\bm{x}/A_n)A_n^{-1}g(\bm{x}/A_n)d\bm{x}\\
&= {A_n \over n}\int K^2(\bm{z})\sigma_{\varepsilon}^2(\bm{z} \circ \bm{h})g(\bm{z} \circ \bm{h})d\bm{z}\\
&= \kappa \sigma_{\varepsilon}^2(\bm{0})g(\bm{0})\kappa_0^{(2)} + o(1). 
\end{align*}
\begin{align*}
\Var(W_{n,31}) &= \left({A_n \over n^2h_1 \dots h_d}\right)^2n\Var(K_{Ah}^2(\bm{X}_1)\varepsilon_{n,1}^2)\\
&\leq \left({A_n \over n^2h_1 \dots h_d}\right)^2nE[K_{Ah}^4(\bm{X}_1)\varepsilon_{n,1}^4]\\
&\lesssim {A_n^2 \over n^3 (h_1 \dots h_d)^2}\int K_{Ah}^4(\bm{x})\sigma_{\varepsilon}^4(\bm{x}/A_n)A_n^{-1}g(\bm{x}/A_n)d\bm{x}\\
&= O\left({A_n^2 \over n^2}{1 \over nh_1 \dots h_d}\right) = o(1). 
\end{align*}
Then we have $W_{n,31} = \kappa \sigma_{\varepsilon}^2(\bm{0})g(\bm{0})\kappa_0^{(2)} + o_p(1)$. 

For $W_{n,32}$, applying similar arguments in the proof of Theorem 4.1, 
we have $E[W_{n,32}] = 0$ and 
\begin{align}
&\left({A_n \over n^2h_1 \dots h_d}\right)^{-2}E[W_{n,32}^2] \nonumber \\
&= \sum_{i \neq j, k \neq \ell}\!\!\!\!E\left[K_{Ah}(\bm{X}_i)K_{Ah}(\bm{X}_j)\bar{K}_b^2(\bm{X}_i - \bm{X}_j)K_{Ah}(\bm{X}_k)K_{Ah}(\bm{X}_\ell)\bar{K}_b(\bm{X}_k - \bm{X}_\ell)\varepsilon_{n,i}\varepsilon_{n,j}\varepsilon_{n,k}\varepsilon_{n,\ell}\right] \nonumber \\
&= \sum_{i \neq j}E\left[K_{Ah}^2(\bm{X}_i)K_{Ah}^2(\bm{X}_j)\bar{K}_b^2(\bm{X}_i - \bm{X}_j)\varepsilon_{n,i}^2\varepsilon_{n,j}^2\right] \nonumber \\
&= n(n-1)(h_1 \dots h_d)^2\int_{\bm{h}^{-1}R_0^2} \bar{K}_b(A_n(\bm{z}_1 - \bm{z}_2) \circ \bm{h})K^2(\bm{z}_1)K^2(\bm{z}_2)\sigma_{\varepsilon}^2(\bm{z}_1 \circ \bm{h})\sigma_{\varepsilon}^2(\bm{z}_2 \circ \bm{h}) \nonumber \\
&\quad \times A_n^{-2}g(\bm{z}_1 \circ \bm{h})g(\bm{z}_2 \circ \bm{h})d\bm{z}_1 d\bm{z}_2 \nonumber \\
&= n(n-1)A_n^{-1}h_1 \dots h_d b_1 \dots b_d\int_{A_n \bm{h}R'_{\bm{h},0}/\bm{b}}\bar{K}^2(\bm{v})\left(\int_{R_{\bm{h},0}\left({\bm{v} \circ \bm{b} \over A_n\bm{h}}\right)}K^2\left(\bm{z}_2 + {\bm{v} \circ \bm{b} \over A_n\bm{h}}\right)K^2(\bm{z}_2) \right. \nonumber \\
&\left. \quad \times \sigma_{\varepsilon}^2\left(\bm{z}_2 \circ \bm{h} + {\bm{v} \circ \bm{b} \over A_n}\right)\sigma_{\varepsilon}^2\left(\bm{z}_2 \circ \bm{h}\right)g\left(\bm{z}_2 \circ \bm{h} + {\bm{v} \circ \bm{b} \over A_n}\right)g\left(\bm{z}_2 \circ \bm{h}\right)d\bm{z}_2 \right)d\bm{v} \nonumber \\
&= n(n-1)A_n^{-1}h_1 \dots h_d b_1 \dots b_d \left(\sigma_{\varepsilon}^4(\bm{0})g^2(\bm{0})\kappa_0^{(4)}\int \bar{K}^2(\bm{v})d\bm{v} + o(1)\right). \label{Vn32-new}
\end{align}
Then we have $E[W_{n,32}] = O\left({A_n \over n}{b_1 \dots b_d \over n h_1 \dots h_d}\right) = o(1)$ and this yields $W_{n,32} = o_p(1)$. The results on $W_{n,31}$ and $W_{n,32}$ yield
\begin{align}\label{Vn3-new}
W_{n,3} &\stackrel{p}{\to} \kappa \sigma_{\varepsilon}^2(\bm{0})g(\bm{0})\kappa_0^{(2)}.
\end{align}

For $W_{n,2}$, observe that
\begin{align*}
&\left({2A_n \over n^2h_1 \dots h_d}\right)^{-2}E_{\cdot|\bm{X}}[W_{n,2}^2]\\ 
&= \sum_{i=1}^{n}K_{Ah}^4(\bm{X}_i)\eta^2(\bm{X}_i/A_n)\sigma_{\varepsilon}^2(\bm{X}_j/A_n)\\
&\quad + \sum_{i\neq j}^{n}K_{Ah}^2(\bm{X}_i)K_{Ah}^2(\bm{X}_{j})\bar{K}_b(\bm{X}_i - \bm{X}_j)\eta(\bm{X}_i/A_n)\eta(\bm{X}_j/A_n)\sigma_{\varepsilon}^2(\bm{X}_j/A_n)\\
&\quad + \sum_{i \neq j}^{n}K_{Ah}(\bm{X}_i)K_{Ah}^3(\bm{X}_{j})\bar{K}_b(\bm{X}_i - \bm{X}_j)\eta(\bm{X}_i/A_n)\eta(\bm{X}_j/A_n)\sigma_{\bm{e}}(\bm{X}_i - \bm{X}_j)\sigma_{\varepsilon}^2(\bm{X}_j/A_n)\\
&\quad + \sum_{i \neq j \neq  \ell}^{n}K_{Ah}(\bm{X}_i)K_{Ah}^2(\bm{X}_{j})K_{Ah}(\bm{X}_\ell)\bar{K}_b(\bm{X}_i - \bm{X}_j)\bar{K}_b(\bm{X}_\ell - \bm{X}_j)\\
&\quad \quad \times \eta(\bm{X}_i/A_n)\eta(\bm{X}_j/A_n)\sigma_{\bm{e}}(\bm{X}_i - \bm{X}_\ell)\sigma_{\varepsilon}^2(\bm{X}_j/A_n)\\
&=: W_{n,21} + W_{n,22} + W_{n,23} + W_{n,24}. 
\end{align*}

For $W_{n,21}$, we have $E[W_{n,21}] = O(nh_1 \dots h_d)$ and this yields $\left({2A_n \over n^2h_1 \dots h_d}\right)^2E[W_{n,21}] = O\left({A_n^2 \over n^2}{1 \over nh_1 \dots h_d}\right) = o(1)$.

For $W_{n,22}$, applying similar arguments to show (\ref{Vn32-new}), we have
\begin{align*}
E[W_{n,22}] &= n(n-1)A_n^{-1}h_1 \dots h_d b_1 \dots b_d \left(\eta^2(\bm{0})\sigma_{\varepsilon}^2(\bm{0})g^2(\bm{0})\kappa_0^{(4)}\int \bar{K}(\bm{v})d\bm{v} + o(1)\right).
\end{align*}
This yields $E\left({2A_n \over n^2h_1 \dots h_d}\right)^2[W_{n,22}]  = O\left({A_n \over n}{b_1 \dots b_d \over n h_1 \dots h_d}\right) = o(1)$.

For $W_{n,23}$, applying similar arguments in the proof of Theorem 4.1, we have
\begin{align*}
E[W_{n,23}] &\lesssim \sum_{i \neq j}^{n}E\left[K_{Ah}(\bm{X}_i)K_{Ah}^3(\bm{X}_{j})\eta(\bm{X}_i/A_n)\eta(\bm{X}_j/A_n)\sigma_{\bm{e}}(\bm{X}_i - \bm{X}_j)\sigma_{\varepsilon}^2(\bm{X}_j/A_n)\right]\\
&= O\left(n^2A_n^{-1}h_1 \dots h_d \right).
\end{align*}
This yields $\left({2A_n \over n^2h_1 \dots h_d}\right)^2E[W_{n,23}] = O\left({A_n \over  n}{1 \over n h_1 \dots h_d}\right) = o(1)$.

For $W_{n,24}$, applying similar arguments in the proof of Theorem 4.1, we have
\begin{align*}
&E[W_{n,24}] \\
&= n(n-1)(n-2)(h_1 \dots h_d)^3\int_{\bm{h}^{-1}R_0^3} \bar{K}_b(A_n(\bm{z}_1 - \bm{z}_2) \circ \bm{h})\bar{K}_b(A_n(\bm{z}_3 - \bm{z}_2) \circ \bm{h})\\
&\quad \times \sigma_{\bm{e}}(A_n(\bm{z}_1 - \bm{z}_3) \circ \bm{h})K(\bm{z}_1)K^2(\bm{z}_2)K(\bm{z}_3)\eta(\bm{z}_1 \circ \bm{h})\eta(\bm{z}_2 \circ \bm{h})\sigma_{\varepsilon}^2\left(\bm{z}_2 \circ \bm{h}\right) \\
&\quad \times A_n^{-3}g(\bm{z}_1 \circ \bm{h})g(\bm{z}_2 \circ \bm{h})g(\bm{z}_3 \circ \bm{h})d\bm{z}_1 d\bm{z}_2d\bm{z}_3\\
&= n(n-1)(n-2)A_n^{-1}(h_1 \dots h_d)^2\int_{\bm{h}^{-1}R_0}K^2(\bm{z}_2)\sigma_{\varepsilon}^2(\bm{z}_2 \circ \bm{h})g(\bm{z}_2 \circ \bm{h})\\
&\quad \left\{\int_{A_n \bm{h}R'_{\bm{h},0}}\sigma_{\bm{e}}(\bm{v})\left(\int_{R_{\bm{h},0}\left({\bm{v} \over A_n\bm{h}}\right)}\bar{K}\left({\bm{v} + A_n(\bm{z}_3-\bm{z}_2) \circ \bm{h} \over \bm{b}}\right)\bar{K}\left({A_n(\bm{z}_3-\bm{z}_2) \circ \bm{h} \over \bm{b}}\right) \right. \right.  \\
&\left. \left .\quad \times K\!\!\left(\!\bm{z}_3 + {\bm{v}  \over A_n\bm{h}}\!\right)\!K(\bm{z}_3)\eta \! \left(\!\bm{z}_3\! \circ \! \bm{h} + {\bm{v} \over A_n}\!\right)\!\eta\left(\bm{z}_3\! \circ \! \bm{h}\right)\!g\!\left(\bm{z}_3\! \circ \! \bm{h} + {\bm{v}  \over A_n}\right)\!g\left(\bm{z}_3\! \circ \! \bm{h}\right)d\bm{z}_3\! \right)\!d\bm{v}\! \right\}\! d\bm{z}_2\\
&= n(n-1)(n-2)A_n^{-2}h_1 \dots h_db_1 \dots b_d\int_{\bm{h}^{-1}R_0}K^2(\bm{z}_2)\sigma_{\varepsilon}^2(\bm{z}_2 \circ \bm{h})g(\bm{z}_2 \circ \bm{h})\\
&\quad \left\{\int_{A_n \bm{h}R'_{\bm{h},0}}\sigma_{\bm{e}}(\bm{v})\left(\int_{A\bm{h}R_{\bm{h},0}\left({\bm{v} \over A_n\bm{h}}\right)/\bm{b}}\bar{K}\left({\bm{v} \over \bm{b}} + \bm{w} - {A_n\bm{z}_2 \circ \bm{h} \over \bm{b}}\right)\bar{K}\left(\bm{w} - {A_n\bm{z}_2 \circ \bm{h} \over \bm{b}}\right) \right. \right.  \\
&\left. \left .\quad \times K\!\!\left(\!{\bm{w}\circ \bm{b} \over A_n\bm{h}} + {\bm{v}  \over A_n\bm{h}}\!\right)\!K\!\left({\bm{w}\! \circ\! \bm{b} \over A_n\bm{h}}\right)\!\eta \! \left(\!{\bm{w}\! \circ \! \bm{b} \over A_n} + {\bm{v} \over A_n}\!\right)\!\eta\!\left({\bm{w} \! \circ \! \bm{b} \over A_n}\right) \right. \right. \\
&\left. \left. \quad \times g\!\left({\bm{w} \! \circ \! \bm{b} \over A_n} + {\bm{v}  \over A_n}\right)\!g\! \left({\bm{w} \! \circ \! \bm{b} \over A_n}\right)\!d\bm{w}\! \right)\!d\bm{v}\! \right\}\! d\bm{z}_2\\
&= O\left(n^3A_n^{-2}h_1 \dots h_db_1 \dots b_d\right).
\end{align*}
This yields $\left({2A_n \over n^2h_1 \dots h_d}\right)^2E[W_{n,24}] = O\left({b_1 \dots b_d \over n h_1 \dots h_d}\right) = o(1)$. The results on $W_{n,21}$ $W_{n,22}$, $W_{n,23}$, and $W_{n,24}$ yield
\begin{align}\label{Vn2-new}
W_{n,2} &\stackrel{p}{\to} 0.
\end{align}

For $W_{n,1}$, applying similar arguments to show (\ref{Vn32-new}), we have
\begin{align*}
&\left({A_n \over n^2h_1 \dots h_d}\right)^{-1}E[W_{n,1}] \\
&= nh_1\dots h_d\int K^2(\bm{z})\eta^2(\bm{z} \circ \bm{h})g(\bm{z})d\bm{z}\\
&\quad + n(n-1)A_n^{-1}h_d \dots h_d\int_{A_n \bm{h}R'_{\bm{h},0}}\!\!\!\!\!\!\!\!\!\!\! \!\sigma_{\bm{e}}(\bm{v})\bar{K}_b(\bm{v}) \left(\int_{R_{\bm{h},0}((\bm{v} \circ \bm{h}^{-1})/A_n)}\!\!K\!\! \left(\!\bm{z}_2  \!+\! {\bm{v} \circ \bm{h}^{-1}\over A_n}\!\right)\!\! K(\bm{z}_2)  \right. \\
&\left. \quad  \times \eta\! \left(\!\bm{z}_2 \circ \bm{h} \!+\! {\bm{v} \over A_n}\!\right)\!\eta(\bm{z}_2 \circ \bm{h})g\! \left(\! \bm{z}_2 \circ \bm{h} \!+\! {\bm{v} \over A_n}\!\right)g(\bm{z}_2 \circ \bm{h})d\bm{z}_2 \!\right)d\bm{v}\\
&=: W_{n,11} + W_{n,12}. 
\end{align*}

For $W_{n,11}$, we have $W_{n,11} = nh_1\dots h_d(\eta^2(0)g(0)\kappa_0^{(2)} + o(1))$.

For $W_{n,12}$, we have
\begin{align*}
W_{n,12} &= n(n-1)A_n^{-1}h_d \dots h_d\int_{A_n \bm{h}R'_{\bm{h},0}}\!\!\!\!\!\!\!\!\!\!\! \!\sigma_{\bm{e}}(\bm{v}) \left(\int_{R_{\bm{h},0}((\bm{v} \circ \bm{h}^{-1})/A_n)}\!\!K\!\! \left(\!\bm{z}_2  \!+\! {\bm{v} \circ \bm{h}^{-1}\over A_n}\!\right)\!\! K(\bm{z}_2)  \right. \\
&\left. \quad  \times \eta\! \left(\!\bm{z}_2 \circ \bm{h} \!+\! {\bm{v} \over A_n}\!\right)\!\eta(\bm{z}_2 \circ \bm{h})g\! \left(\! \bm{z}_2 \circ \bm{h} \!+\! {\bm{v} \over A_n}\!\right)g(\bm{z}_2 \circ \bm{h})d\bm{z}_2 \!\right)d\bm{v}\\
&+ n(n-1)A_n^{-1}h_d \dots h_d\int_{A_n \bm{h}R'_{\bm{h},0}}\!\!\!\!\!\!\!\!\!\!\! \!\sigma_{\bm{e}}(\bm{v})(\bar{K}_b(\bm{v})-1) \left(\int_{R_{\bm{h},0}((\bm{v} \circ \bm{h}^{-1})/A_n)}\!\!K\!\! \left(\!\bm{z}_2  \!+\! {\bm{v} \circ \bm{h}^{-1}\over A_n}\!\right)\!\! K(\bm{z}_2)  \right. \\
&\left. \quad  \times \eta\! \left(\!\bm{z}_2 \circ \bm{h} \!+\! {\bm{v} \over A_n}\!\right)\!\eta(\bm{z}_2 \circ \bm{h})g\! \left(\! \bm{z}_2 \circ \bm{h} \!+\! {\bm{v} \over A_n}\!\right)g(\bm{z}_2 \circ \bm{h})d\bm{z}_2 \!\right)d\bm{v}\\
&=: W_{n,121} + W_{n,122}. 
\end{align*}

For $W_{n,121}$, from the proof of Theorem 4.1, we have
\begin{align*}
W_{n,121} &= n^2A_n^{-1}h_1 \dots h_d\left(\eta^2(\bm{0})g^2(\bm{0})\kappa_0^{(2)}\int \sigma_{\bm{e}}(\bm{v})d\bm{v} + o(1)\right).
\end{align*}

For $W_{n,122}$, observe that  for any $M>0$, 
\begin{align*}
&W_{n,122} \\
&= n(n-1)A_n^{-1}h_d \dots h_d\int_{A_n \bm{h}R'_{\bm{h},0} \cap \{\|\bm{v}\| \leq M\}}\!\!\!\!\!\!\!\!\!\!\! \!\!\!\!\!\sigma_{\bm{e}}(\bm{v})(\bar{K}_b(\bm{v})-1)\\ 
&\quad \left(\int_{R_{\bm{h},0}((\bm{v} \circ \bm{h}^{-1})/A_n)}\!\!\!\!\!K\!\! \left(\!\bm{z}_2  \!+\! {\bm{v} \circ \bm{h}^{-1}\over A_n}\!\right)\!\! K(\bm{z}_2) \eta\! \left(\!\bm{z}_2 \circ \bm{h} \!+\! {\bm{v} \over A_n}\!\right)\!\eta(\bm{z}_2 \circ \bm{h}) \right. \\
&\left. \quad \times g\! \left(\! \bm{z}_2 \circ \bm{h} \!+\! {\bm{v} \over A_n}\!\right)g(\bm{z}_2 \circ \bm{h})d\bm{z}_2 \!\right)d\bm{v}\\
& + n(n-1)A_n^{-1}h_d \dots h_d\int_{A_n \bm{h}R'_{\bm{h},0} \cap \{\|\bm{v}\| > M\}}\!\!\!\!\!\!\!\!\!\!\! \!\!\!\!\!\sigma_{\bm{e}}(\bm{v})(\bar{K}_b(\bm{v})-1) \\ &\quad \left(\int_{R_{\bm{h},0}((\bm{v} \circ \bm{h}^{-1})/A_n)}\!\!\!\!\!K\!\! \left(\!\bm{z}_2  \!+\! {\bm{v} \circ \bm{h}^{-1}\over A_n}\!\right)\!\! K(\bm{z}_2)\eta\! \left(\!\bm{z}_2 \circ \bm{h} \!+\! {\bm{v} \over A_n}\!\right)\!\eta(\bm{z}_2 \circ \bm{h})  \right. \\
&\left. \quad  \times g\! \left(\! \bm{z}_2 \circ \bm{h} \!+\! {\bm{v} \over A_n}\!\right)g(\bm{z}_2 \circ \bm{h})d\bm{z}_2 \!\right)d\bm{v}\\
&=: W_{n,1221} + W_{n,1222}. 
\end{align*}

Observe that
\begin{align*}
|W_{n,1221}| &\lesssim n^2A_n^{-1}h_1\dots h_d{M \over \min_{1 \leq j \leq d}b_j},\\
|W_{n,1222}| &\lesssim n^2A_n^{-1}h_1\dots h_d\int_{\|\bm{v}\|>M}|\sigma_{\bm{e}}(\bm{v})|d\bm{v}
\end{align*}
Then by taking $M = \min_{1 \leq j \leq d}b_j^{1/2}$, we have 
\[
W_{n,1221} = o(n^2A_n^{-1}h_1\dots h_d),\ W_{n,1222} = o(n^2A_n^{-1}h_1\dots h_d). 
\]
The results on $W_{n,121}$, $W_{n,1221}$, and $W_{n,1222}$ yield
\[
W_{n,12} = n^2A_n^{-1}h_1 \dots h_d\left(\eta^2(\bm{0})g^2(\bm{0})\kappa_0^{(2)}\int \sigma_{\bm{e}}(\bm{v})d\bm{v} + o(1)\right).
\]
This and the result on $W_{n,11}$ yield
\begin{align}\label{Vn1-new}
E[W_{n,1}] &= \kappa\eta^2(\bm{0})g(\bm{0})\kappa_0^{(2)}  + \eta^2(\bm{0})g^2(\bm{0})\kappa_0^{(2)}\int \sigma_{\bm{e}}(\bm{v})d\bm{v} + o(1).
\end{align}

Combining (\ref{Vn3-new}), (\ref{Vn2-new}), (\ref{Vn1-new}) and the results in the proof of Theorem 4.1, we have
\begin{align*}
\hat{W}_{n,1}(\bm{0}) &\stackrel{p}{\to} \kappa(\eta^2(\bm{0}) + \sigma_{\varepsilon}^2(\bm{0}))g(\bm{0}) + \eta^2(\bm{0})g^2(\bm{0})\int \sigma_{\bm{e}}(\bm{v})d\bm{v}
\end{align*} 
and this yields the desired result.
\end{proof}

\subsection{Proof of Corollary 5.1}\label{sec: Cor5.1-proof}

\begin{proof}
Corollary 5.1 follows immediately from Theorem 4.1 and Proposition 5.1. 
\end{proof}

\section{Proofs for Section \ref{sec:TST}}\label{sec: TST-proof}

In this section, we prove Proposition \ref{prp: LP-TST} (Section \ref{sec: prp-TST-proof}), Proposition \ref{prp: var-est-TST} (Section \ref{sec: prp var-est-TST proof}), and Corollary \ref{cor: LP-CI-TST} (Sections \ref{sec: cor CI-TST proof}). 

\subsection{Proof of Proposition \ref{prp: LP-TST}}\label{sec: prp-TST-proof}

\begin{proof}
For any $\bm{t} = (t_0,t_1,\dots, t_d, t_{11},\dots,t_{dd},\dots,t_{1\dots1},\dots, t_{d\dots d})' \in \mathbb{R}^D$, we define
\begin{align*}
\overline{W}_{n1}(\bm{0}) &:= \sum_{\ell_1=1}^{n_1}K_{Ah}\left(\bm{X}_{1,\ell_1}\right)\left[\bm{t}'H^{-1}
\left(
\begin{array}{c}
1 \\
\check{\bm{X}}_{1,\ell_1}
\end{array}
\right)\right]\underbrace{\left(\eta_1\left({\bm{X}_{1, \ell_1} \over A_n}\right)e_1(\bm{X}_{1,\ell_1}) + \sigma_{\varepsilon,1}\left({\bm{X}_{1,\ell_1} \over A_n}\right)\varepsilon_{1,\ell_1}\right)}_{=: \overline{e}_{n1,\ell_1} + \overline{\varepsilon}_{n1,\ell_1}}, \\
\overline{W}_{n2}(\bm{0}) &:= \sum_{\ell_2=1}^{n_2}K_{Ah}\left(\bm{X}_{2,\ell_2}\right)\left[\bm{t}'H^{-1}
\left(
\begin{array}{c}
1 \\
\check{\bm{X}}_{2,\ell_2}
\end{array}
\right)\right]\underbrace{\left(\eta_2\left({\bm{X}_{2,\ell_2} \over A_n}\right)e_2(\bm{X}_{2,\ell_2}) + \sigma_{\varepsilon,2}\left({\bm{X}_{2,\ell_2} \over A_n}\right)\varepsilon_{2,\ell_2}\right)}_{=: \overline{e}_{n2,\ell_2} + \overline{\varepsilon}_{n2,\ell_2}}. 
\end{align*}
By inspecting the proof of Theorem 4.1, 
to show Proposition \ref{prp: LP-TST}, it is sufficient to verify
\begin{align*}
&E\left[\left(\overline{W}_{n1}(\bm{0}) -\overline{W}_{n1}(\bm{0})\right)^2\right]/(h_1 \dots h_d )\\ 
&= \left(n_1\left\{(\eta_1^2(\bm{0}) + \sigma_{\varepsilon,1}^2(\bm{0}))g_1(\bm{0}) + n_1A_n^{-1}\eta_1^2(\bm{0})g_1^2(\bm{0})\int \sigma_{\bm{e},11}(\bm{v})d\bm{v}\right\} \right.\\
&\left. \quad + n_2\left\{(\eta_2^2(\bm{0}) + \sigma_{\varepsilon,2}^2(\bm{0}))g_2(\bm{0}) + n_2A_n^{-1}\eta_2^2(\bm{0})g_2^2(\bm{0})\int \sigma_{\bm{e},22}(\bm{v})d\bm{v}\right\} \right. \\
&\left. \quad -2 n_1 n_2A_n^{-1}\eta_1(\bm{0})\eta_2(\bm{0})g_1(\bm{0})g_2(\bm{0})\int \sigma_{\bm{e},12}(\bm{v})d\bm{v} \right)\\
&\quad \times \left(\int K^2(\bm{z})\left[\bm{t}'
\left(
\begin{array}{c}
1 \\
\check{\bm{z}}
\end{array}
\right)\right]^2d\bm{z}\right)(1 + o(1)),\ n \to \infty. 
\end{align*}
Let $E_{\bm{X}_{12}}$ denote the expectation with respect to $\{\bm{X}_{1,\ell_1}\}$ and $\{\bm{X}_{2,\ell_2}\}$ and let $E_{\cdot \mid \bm{X}_{12}}$ denote the conditional expectation given $\sigma(\{\bm{X}_{1,\ell_1}\} \cup \{\bm{X}_{2,\ell_2}\})$. Observe that 
\begin{align*}
&E_{\cdot \mid \bm{X}_{12}}\left[\left(\overline{W}_{n1}(\bm{0}) -\overline{W}_{n2}(\bm{0})\right)^2\right]\\
&= \sum_{\ell_{11},\ell_{12}=1}^{n_1}E_{\cdot \mid \bm{X}_{12}}\left[K_{Ah}\left(\bm{X}_{1,\ell_{11}}\right)K_{Ah}\left(\bm{X}_{1,\ell_{12}}\right)\left[\bm{t}'H^{-1}
\left(
\begin{array}{c}
1 \\
\check{\bm{X}}_{1,\ell_{11}}
\end{array}
\right)\right]\left[\bm{t}'H^{-1}
\left(
\begin{array}{c}
1 \\
\check{\bm{X}}_{1,\ell_{12}}
\end{array}
\right)\right] \right. \\
&\left. \quad \times (\overline{e}_{n1,\ell_{11}} + \overline{\varepsilon}_{n1,\ell_{11}})(\overline{e}_{n1,\ell_{12}} + \overline{\varepsilon}_{n1,\ell_{12}})\right]\\
&\quad + \sum_{\ell_{21}, \ell_{22}=1}^{n_2}E_{\cdot \mid \bm{X}_{12}}\left[K_{Ah}\left(\bm{X}_{2,\ell_{21}}\right)K_{Ah}\left(\bm{X}_{2,\ell_{22}}\right)\left[\bm{t}'H^{-1}
\left(
\begin{array}{c}
1 \\
\check{\bm{X}}_{2,\ell_{21}}
\end{array}
\right)\right]\left[\bm{t}'H^{-1}
\left(
\begin{array}{c}
1 \\
\check{\bm{X}}_{2,\ell_{22}}
\end{array}
\right)\right] \right. \\
&\left. \quad \times (\overline{e}_{n2,\ell_{21}} + \overline{\varepsilon}_{n2,\ell_{21}})(\overline{e}_{n2,\ell_{22}} + \overline{\varepsilon}_{n2,\ell_{22}})\right]\\
&\quad -2\sum_{\ell_1=1}^{n_1}\sum_{\ell_2=1}^{n_2}E_{\cdot \mid \bm{X}_{12}}\left[K_{Ah}\left(\bm{X}_{1,\ell_1}\right)K_{Ah}\left(\bm{X}_{2,\ell_2}\right)\left[\bm{t}'H^{-1}
\left(
\begin{array}{c}
1 \\
\check{\bm{X}}_{1,\ell_1}
\end{array}
\right)\right]\left[\bm{t}'H^{-1}
\left(
\begin{array}{c}
1 \\
\check{\bm{X}}_{2,\ell_2}
\end{array}
\right)\right] \right. \\
&\left. \quad \times (\overline{e}_{n1,\ell_1} + \overline{\varepsilon}_{n1,\ell_1})(\overline{e}_{n2,\ell_2} + \overline{\varepsilon}_{n2,\ell_2})\right]\\
&=: \overline{W}_{n11} +\overline{W}_{n12} -2\overline{W}_{n13}. 
\end{align*}
Applying the same argument in Step 2 of the proof of Theorem 4.1, 
we have
\begin{align*}
E_{\bm{X}_{12}}[\overline{W}_{n1\ell}] &= n_{\ell}h_1 \dots h_d g_{\ell}(\bm{0})\left\{(\eta_\ell^2(\bm{0}) + \sigma_{\varepsilon,\ell}^2(\bm{0})) + n_{\ell}A_n^{-1}\eta_\ell^2(\bm{0})g_{\ell}(\bm{0})\int \sigma_{\bm{e},\ell \ell}(\bm{v})d\bm{v}\right\}\\
&\quad \times \left(\int K^2(\bm{z})\left[\bm{t}'
\left(
\begin{array}{c}
1 \\
\check{\bm{z}}
\end{array}
\right)\right]^2d\bm{z}\right)(1 + o(1)),\ \ell = 1,2,\\
E_{\bm{X}_{12}}[\overline{W}_{n13}] &= n_1n_2A_n^{-1}h_1 \dots h_d \left(\eta_1(\bm{0})\eta_2(\bm{0})g_1(\bm{0})g_2(\bm{0})\int \sigma_{\bm{e},12}(\bm{v})d\bm{v}\right)(1 + o(1))
\end{align*}
as $n \to \infty$. Therefore, we obtain the desired result. 
\end{proof}

\subsection{Proof of Proposition \ref{prp: var-est-TST}}\label{sec: prp var-est-TST proof}

\begin{proof}

Applying the same argument in the proof of Proposition 5.1, 
we have that as $n \to \infty, $
\begin{align*}
\overline{g}_{n_k}(\bm{0}) &= g_k(\bm{0}) + o_p(1),\ k=1,2, \\
\overline{V}_{n,1}(\bm{0}) &= \kappa_0^{(2)}\left(\kappa(\eta_k^2(\bm{0}) + \sigma_{k,\varepsilon}^2(\bm{0})) + \eta_k^2(\bm{0})g_k^2(\bm{0})\int \sigma_{\bm{e},kk}(\bm{v})d\bm{v}\right) + o_p(1),\\
\overline{V}_{n,2}(\bm{0}) &= \kappa_0^{(2)}\left(\theta \kappa(\eta_k^2(\bm{0}) + \sigma_{k,\varepsilon}^2(\bm{0})) + \eta_k^2(\bm{0})g_k^2(\bm{0})\int \sigma_{\bm{e},kk}(\bm{v})d\bm{v}\right) + o_p(1),\\
\overline{V}_{n,3}(\bm{0}) &= \kappa_0^{(2)}\eta_1(\bm{0})\eta_2(\bm{0})g_1(\bm{0})g_2(\bm{0})\int \sigma_{\bm{e},12}(\bm{v})d\bm{v} + o_p(1). 
\end{align*}
Therefore, $\check{V}_n(\bm{0}) \stackrel{p}{\to} \overline{V}_1(\bm{0}) + \overline{V}_2(\bm{0}) - 2\overline{V}_3(\bm{0})$ as $n \to \infty$. 

\end{proof}

\subsection{Proof of Corollary \ref{cor: LP-CI-TST}}\label{sec: cor CI-TST proof}

\begin{proof}
Corollary \ref{cor: LP-CI-TST} follows immediately from Propositions \ref{prp: LP-TST} and \ref{prp: var-est-TST}. 
\end{proof}

\section{Proofs for Section \ref{sec:examples}}\label{Appendix: Sec6}

In this section, we prove Proposition \ref{Levy-MA-moment} (Section \ref{sec: Levy-MA-moment proof}) and Proposition \ref{Levy-MA-moment-unif} (Section \ref{sec: Levy-MA-moment-unif proof}).

\subsection{Proof of Proposition \ref{Levy-MA-moment}}\label{sec: Levy-MA-moment proof}

\begin{proof}
Define $r_1 = \min_{1\leq j,k \le 2}r_{1,jk}$. We first check the asymptotic negligibility of the random field $\bm{e}_{2,m_n}$, that is, 
\begin{align}\label{e2m-asy-neg}
\max_{1 \leq i \leq n}e_{2j,m_{n}}(\bm{X}_{i}) = O_p\left(\exp\left(-{r_1n^{\zeta_0\zeta_1\zeta_2} \over 2}\right)\right),\ n \to \infty,
\end{align}
Note that under Condition (a), we have $E[|e_{j}(\bm{0})|^6]<\infty$ since $\bm{e}$ is Gaussian. Under Condition (b), we also have $E[|L_j([0,1]^d)|^{6}] < \infty$ since $\int_{|x|>1}|x|^{6}\nu_{0,j}(x)dx <\infty$ (cf. Theorem 25.3 in \cite{Sa99}). Define $\sigma_{\bm{e}_{1,m_n}}^{(j,k)}(\bm{x}) = E[e_{1j,m_n}(\bm{0})e_{1k,m_n}(\bm{x})]$, $j,k = 1,2$. 
Then we have that 
\begin{align*}
E[|e_{1j,m_n}(\bm{0})|^6] & \leq E[|e_j(\bm{0})|^6] \lesssim \int e^{-6r_1\|\bm{u}\|}d\bm{u}<\infty, \\ 
|\sigma_{\bm{e}_{1,m_n}}^{(j,k)}(\bm{x})| & \lesssim |E[e_j(\bm{0})e_k(\bm{x})]| \lesssim \int e^{-r_1\|\bm{u}\|}e^{-r_1\|\bm{x}-\bm{u}\|}d\bm{u} \nonumber \\ 
&\leq \int e^{-r_1\|\bm{u}\|}e^{-{r_1 \over 2}(\|\bm{x}\|-\|\bm{u}\|)}d\bm{u} \lesssim  e^{-{r_1 \over 2}\|\bm{x}\|}. 
\end{align*}
The latter implies that $\int |\sigma_{\bm{e}_{1,m_n}}^{(j,k)}(\bm{v})|d\bm{v}<\infty$, $j,k=1,2$. Likewise, 

\begin{align*}
E[(e_{2j,m_n}(\bm{0}))^{4}] &\lesssim \int_{\mathbb{R}^{d}}e^{-4r_1\|\bm{u}\|}\left(1- \psi_{0}\left(\|\bm{u}\| : m_{n}\right)\right)^{4}d\bm{u}\\
& \lesssim  \int_{\|\bm{u}\| \geq m_{n}/4}e^{-4r_1\|\bm{u}\|}\left|1 + {4 \over m_{n}}\left(\|\bm{u}\| - {m_{n} \over 2}\right)\right|^{4}d\bm{u}\\
& \lesssim \int_{\|\bm{u}\| \geq m_{n}/4}e^{-4r_1\|\bm{u}\|}\left|1 + {4\|\bm{u}\| \over m_{n}}\right|^{4}d\bm{u}\\
&\leq 2^{q-1}  \int_{\|\bm{u}\| \geq m_{n}/4}e^{-4r_1\|\bm{u}\|}\left(1 + {4^{4}\|\bm{u}\|^{4} \over m_{n}^{4}}\right)d\bm{u}\\
&\lesssim  \int_{m_{n}/4}^{\infty}e^{-4r_1t}\left(1 + {4^{4}t^{4} \over m_{n}^{4}}\right)t^{d-1}dt \\
&\lesssim  m_{n}^{d-1}e^{-r_1m_{n}}.
\end{align*}
By Markov's inequality and Lemma 2.2.2 in \cite{vaWe96}, we have 
\begin{align*}
P_{\cdot \mid \bm{X}}\left(\left|\max_{1 \leq i \leq n}e_{2j,m_{n}}(\bm{X}_{i})\right| > \varrho \right) &\leq \varrho^{-1}E_{\cdot \mid \bm{X}}\!\left[\max_{1 \leq i \leq n}\left|e_{2j,m_{n}}(\bm{X}_{i})\right|\right]\\ 
&\leq \varrho^{-1}n^{1/4}\max_{1 \leq i \leq n}\left(\!E_{\cdot \mid \bm{X}}\left[\left|e_{2j,m_n}(\bm{0})\right|^{4}\right]\right)^{1/4}\\
&\lesssim \varrho^{-1}n^{1/4}m_n^{(d-1)/4}e^{-r_1m_n/4}. 
\end{align*}
Therefore, under the assumptions of Proposition \ref{Levy-MA-moment}, we have (\ref{e2m-asy-neg}), which implies that $\bm{e}_{2,m_n}$ is asymptotically negligible. Hence we can replace $\bm{e}$ with $\bm{e}_{1,m_n}$ in the results in Section 4. 

Next we check the mixing conditions on $\bm{e}_{1,m_n}$. Let $\alpha_{\bm{e}_1}(a;b)$ be the $\alpha$-mixing coefficients of $\bm{e}_{1,m_n}$. Note that $\alpha_{\bm{e}_1}(a;b) \leq \alpha(a;b)$. Since $\bm{e}_{1,m_n}$ is $m_n$-dependent, under the assumptions of Proposition \ref{Levy-MA-moment}, we have $\alpha_1(\underline{A}_{n2}) = 0$, which yields
\begin{align*}
&\left({A_n h_1 \dots h_d \over A_n^{(1)}}\right)\alpha_1(\underline{A}_{n2})\varpi_1(A_n h_1 \dots h_d) = 0,\\
&A_n^{(1)} \left(\alpha_1^{1-2/q}(\underline{A}_{n2}) + \sum_{k=\underline{A}_{n1}}^{\infty}k^{d-1}\alpha_1^{1-2/q}(k)\right)\varpi_1^{1-2/q}(A_n^{(1)}) = 0. 
\end{align*}
Moreover, 
\begin{align*}
\left({A_n^{(1)} \over A_n h_1 \dots h_d}\right)\sum_{k=1}^{\overline{A}_{n1}}k^{2d-1}\alpha_1^{1-4/q}(k) & \lesssim \left({A_n^{(1)} \over A_n h_1 \dots h_d}\right)\sum_{k=1}^{m_n}k^{2d-1}\\
&\leq \left({A_n^{(1)} \over A_n h_1 \dots h_d}\right)m_n^{2d}\\
&\lesssim n^{-\zeta_0\left\{1-\zeta_1(1+\zeta_2)\right\}+ \zeta_3 } = o(1).
\end{align*}
\begin{align*}
&\left\{\left({\overline{A}_{n1} \over \underline{A}_{n1}}\right)^d\left({\overline{A}_{n2} \over \overline{A}_{n1}}\right) + \left({A_n^{(1)} \over \underline{A}_{n1}^d}\right)\left({\left(\overline{A_n h}\right)^d \over A_n h_1 \dots h_d}\right)\left({\overline{A}_{n1} \over \overline{A_n h}}\right)\right\}\sum_{k=1}^{\overline{A}_{n1}}k^{d-1}\alpha_1^{1-2/q}(k)\\
&\lesssim \left\{\left({\overline{A}_{n2} \over \overline{A}_{n1}}\right) + \left({\overline{A}_{n1} \over \overline{A_n h}}\right)\right\}m_n^{d} \lesssim \left(n^{{\zeta_0 \zeta_1 \zeta_2 \over d} - {\zeta_0 \zeta_1 \over d}} + n^{-{\zeta_0 \zeta_1 \over d} - {\zeta_0 \over d}+{\zeta_3 \over d}}\right)n^{{\zeta_0 \zeta_1 \zeta_2 \over 2}}\\
& = n^{\zeta_0 \zeta_1 \left\{ \left({d+2 \over 2d}\right) \zeta_2 - {1 \over d}\right\}} + n^{-\zeta_1 \left\{1 - \left(1 + {d \over 2}\zeta_0\right)\zeta_2\right\} + \zeta_3} = o(1). 
\end{align*}
We can also check that $A_{n,j}h_j/A_{n1,j} \to \infty$ as $n \to \infty$ and that Assumptions 4.1 
(ii), (iii), and (iv) are satisfied. Therefore, we obtain the desired result. 
\end{proof}

\subsection{Proof of Proposition \ref{Levy-MA-moment-unif}}\label{sec: Levy-MA-moment-unif proof}

\begin{proof}
Define 
\begin{align*}
\Psi_{1,\bm{e}_2}(\bm{z}) &= {1 \over n^2A_n^{-1}h_1 \dots h_d }\sum_{i=1}^{n}\!K_{Ah}(\bm{X}_i \!-\! A_n \bm{z})H^{-1}\!
\left(\!
\begin{array}{c}
1 \\
\check{(\bm{X}_i-A_n\bm{z})}
\end{array}
\! \right)\! \eta \! \left({\bm{X}_i \over A_n}\right)\! e_{2,m_n}\!(\bm{X}_i).
\end{align*}
By the same argument in the proof of Proposition \ref{Levy-MA-moment}, we can show that 
\begin{align}\label{e2m-asy-neg-2}
\max_{1 \leq i \leq n}|e_{2,m_{n}}(\bm{X}_{i})| = O_p\left(\exp\left(-{r_1n^{\zeta_0\zeta_1\zeta_2} \over 2}\right)\right),\ n \to \infty.
\end{align}
Then we have
\begin{align*}
\left|\Psi_{1,\bm{e}_2}(\bm{z})\right| &= O_p \! \left(\!{\exp\left(-{r_1n^{\zeta_0\zeta_1\zeta_2} \over 2}\right) \over n^2A_n^{-1}h_1 \dots h_d } \! \right)\!\! \left|\sum_{i=1}^{n}K_{Ah}(\bm{X}_i-A_n \bm{z})H^{-1}\!
\left(\!
\begin{array}{c}
1 \\
\check{(\bm{X}_i-A_n\bm{z})}
\end{array}
\! \right)\! \eta \! \left(\!{\bm{X}_i \over A_n}\!\right)\right|. 
\end{align*}
Applying Proposition \ref{prp:LP-unif} (\ref{LP-unif-R2}) with 
\begin{align*}
f_1(\bm{x}) &= e'_{j_1 \dots j_L}
\left(
\begin{array}{c}
1 \\
\check{\bm{x}}
\end{array}
\right)
,\ f_2(\bm{x}) = 1,\ f_3(\bm{x}) = \eta(\bm{z}), 
\end{align*}
we have that
\begin{align*}
\sup_{\bm{z} \in \mathrm{T}_n}\!\!\left|\Psi_{1,\bm{e}_2}(\bm{z})\right| &\!\leq O_p \!\! \left(\!\!{A_n \over n}\!\exp \! \left(\!-{r_1n^{\zeta_0\zeta_1\zeta_2} \over 2}\!\right)\!\!\right)\!\!\left(\!\sup_{\bm{z} \in \mathrm{T}_n}\!\! \left|\hat{\Psi}_{\mathrm{II}}(\bm{z}) \!-\! E[\hat{\Psi}_{\mathrm{II}}(\bm{z})]\right| \!+\! \sup_{\bm{z} \in \mathrm{T}_n}\!\! \left|\! E[\hat{\Psi}_{\mathrm{II}}(\bm{z})]\right|\right)\\
&\!= O_p\!\! \left({A_n \over n}\exp\left(-{r_1n^{\zeta_0\zeta_1\zeta_2} \over 2}\right)\right)\left( O_p\left(\sqrt{\log n \over nh_1\dots h_d}\right) + O(1)\right)\\
&\!= O_p\!\! \left(\exp\left(-{r_1n^{\zeta_0\zeta_1\zeta_2} \over 2}\right)\right)
\end{align*}
and this implies that $\bm{e}_{2,m_n}$ is asymptotically negligible. Further, under the assumptions in Proposition \ref{Levy-MA-moment-unif} we have that $\beta_1(\underline{A}_{n2}) = 0$, $A_{n,j}h_j/A_{n1,j} \sim n^{{\zeta_0(1-\zeta_1) - \zeta_3 \over d}} \gg 1$, 
\begin{align*}
\left(\!\!{A_n^{(1)} \over (\overline{A}_{n1})^d}\!\! \right) \sim 1,\ {A_n^{{1 \over 2}}(h_1 \dots h_d)^{{1 \over 2}} \over n^{{1 \over q_2}}(\overline{A}_{n1})^d} \sim n^{{\zeta_0(1-2\zeta_1) - \zeta_3 \over 2} - {1 \over q_2}}  \gg (\log n)^{{1 \over 2}+\iota}.
\end{align*}
Therefore, we can replace $\bm{e}$ with $\bm{e}_{1,m_n}$ in Theorem 5.1. 
\end{proof}

\section{Technical tools}\label{Appendix: tool}

We refer to the following lemmas without those proofs. 
\begin{lemma}[(5.19) in \cite{La03b}]\label{n summands}
Under Assumption 2.2, 
we have
\[
P\left(\sum_{i=1}^{n}1\{\bm{X}_{i} \in \Gamma_{n,\bm{z}}(\bm{\ell}; \bm{\Delta})\} > C|\Gamma_{n,\bm{z}}(\bm{\ell}; \bm{\Delta})|nA_{n}^{-1}\ \text{for some $\bm{\ell} \in L_{n1}(\bm{z})$, i.o.}\right) = 0
\]
for any $\bm{\Delta} \in \{1,2\}^{d}$, where $C>0$ is a sufficiently large constant. 
\end{lemma}

\begin{remark}
Lemma \ref{n summands} implies that each $\Gamma_{n,\bm{z}}(\bm{\ell};\bm{\Delta})$ contains at most \\$C|\Gamma_{n,\bm{z}}(\bm{\ell}; \bm{\Delta})|nA_{n}^{-1}$ samples almost surely. 
\end{remark}

\begin{lemma}[Corollary 2.7 in \cite{Yu94}]\label{indep_lemma}
Let $m \in \mathbb{N}$ and let $Q$ be a probability measure on a product space $(\prod_{i=1}^{m}\Omega_{i}, \prod_{i=1}^{m}\Sigma_{i})$ with marginal measures $Q_{i}$ on $(\Omega_{i}, \Sigma_{i})$. Suppose that $h$ is a bounded measurable function on the product probability space such that $|h|\leq M_{h}<\infty$. For $1 \leq a\leq b \leq m$, let $Q_{a}^{b}$ be the marginal measure on $(\prod_{i=a}^{b}\Omega_{i}, \prod_{i=a}^{b}\Sigma_{i})$. For a given $\tau>0$, suppose that, for all $1 \leq k \leq m-1$,  
\begin{align}\label{prod_bound}
\|Q - Q_{1}^{k}\times Q_{k+1}^{m}\|_{TV} \leq 2\tau,
\end{align}
where $Q_{1}^{k}\times Q_{k+1}^{m}$ is a product measure and $\| \cdot \|_{TV}$ is the total variation. Then
\[
| Qh -  Ph| \leq 2M_{h}(m-1)\tau.
\]
where $P=\prod_{i=1}^{m}Q_{i}$, $Qh=\int hdQ$, and $Ph=\int hdP$. 
\end{lemma}

\begin{lemma}[Bernstein's inequality]\label{Bernstein}
Let $X_{1},\hdots, X_{n}$ be independent zero-mean random variables. Suppose that  $\max_{1 \leq i \leq n}|X_{i}|\leq M<\infty$ a.s. Then, for all $t>0$,
\begin{align*}
P\left(\sum _{i=1}^{n}X_{i}\geq t\right)\leq \exp \left(-{{t^{2} \over 2} \over \sum_{i=1}^{n}E[X_{i}^{2}] + {Mt \over 3}}\right). 
\end{align*}
\end{lemma}




\bibliography{LPR-spatial-Ref}
\bibliographystyle{apalike}



\end{document}